\documentclass[reqno, 12pt]{amsart}

\usepackage{amsfonts, amsthm, amsmath, amssymb}

\usepackage{hyperref}
\usepackage{a4wide}
\usepackage{enumitem}

\RequirePackage{mathrsfs} \let\mathcal\mathscr

\DeclareRobustCommand{\SkipTocEntry}[5]{} 

\numberwithin{equation}{section}

\newtheorem{theorem}{Theorem}[section]
\newtheorem{lemma}[theorem]{Lemma}
\newtheorem{proposition}[theorem]{Proposition}
\newtheorem{corollary}[theorem]{Corollary}

\newtheorem*{rc}{\rm \textsc{Recurrence Condition}}

\theoremstyle{definition}
\newtheorem*{ack}{Acknowledgements}
\newtheorem{rem}[theorem]{Remark}
\newtheorem{rems}[theorem]{Remarks}
\newtheorem*{rem*}{Remark}

\newtheorem{definition}[theorem]{Definition}

\renewcommand{\d}{\,\mathrm{d}}

\renewcommand{\rho}{\varrho}
\newcommand{\1}{\mathbf{1}}

\newcommand{\ZZ}{\mathbb{Z}}
\newcommand{\NN}{\mathbb{N}}
\newcommand{\QQ}{\mathbb{Q}}
\newcommand{\RR}{\mathbb{R}}
\newcommand{\CC}{\mathbb{C}}
\newcommand{\T}{\mathbb{T}}

\newcommand{\Q}{\mathcal Q}

\newcommand{\set}[1]{\left\{#1\right\}}
\newcommand{\bigset}[1]{\bigl\{ #1 \bigr\}}
\newcommand{\biggset}[1]{\biggl\{ #1 \biggr\}}

\newcommand{\bigabs}[1]{\bigl| #1 \bigr|}
\newcommand{\Bigabs}[1]{\Bigl| #1 \Bigr|}

\newcommand{\bigbrac}[1]{\bigl( #1 \bigr)}
\newcommand{\Bigbrac}[1]{\Bigl( #1 \Bigr)}
\newcommand{\biggbrac}[1]{\biggl( #1 \biggr)}

\newcommand{\norm}[1]{\left\| #1\right\|}

\newcommand{\twosum}[2]{ \sum_{\substack{#1\\ #2}}}
\newcommand{\threesum}[3]{ \sum_{\substack{#1\\ #2\\ #3}}}

\DeclareMathOperator{\EE}{{\mathbb{E}}}

\newcommand{\X}{\mathcal X}
\newcommand{\W}{\widetilde W}

\renewcommand{\leq}{\leqslant}
\renewcommand{\geq}{\geqslant}
\renewcommand{\bar}{\overline}

\newcommand{\n}{\mathbf{n}}

\newcommand{\Mod}[1]{\;(\operatorname{mod}\,#1)}

\newcommand{\id}{{\rm{id}}}

\DeclareMathOperator{\vol}{vol}
\DeclareMathOperator{\lcm}{lcm}
\DeclareMathOperator{\supp}{supp}
\newcommand{\eps}{\varepsilon}

\begin{document}

\title{Smooth numbers are orthogonal to nilsequences}
\author{Lilian Matthiesen and Mengdi Wang}
\address{Mathematisches Institut\\
Georg-August Universit{\"a}t G{\"o}ttingen\\
Bunsenstra{\ss}e 3-5\\
D-37073 G{\"o}ttingen\\
Germany}
\email{lilian.matthiesen@math.kth.se}

\address{
Institute of Mathematics\\
École Polytechnique Fédérale de Lausanne (EPFL)\\
CH-1015 Lausanne\\
Switzerland}
\email{mengdi.wang@epfl.ch}


\begin{abstract}
The aim of this paper is to study distributional properties of integers without large or small prime factors. 
Define an integer to be $[y',y]$-smooth if all of its prime factors belong to the interval $[y',y]$.
We identify suitable weights $g_{[y',y]}(n)$ for the characteristic function of $[y',y]$-smooth numbers that allow us to establish strong asymptotic results on their distribution in short arithmetic progressions.
Building on these equidistribution properties, we show that (a $W$-tricked version of) the function 
$g_{[y',y]}(n) - 1$ is orthogonal to nilsequences. 
Our results apply in the almost optimal range $(\log N)^{K} < y \leq N$ of the smoothness parameter $y$, 
where $K \geq 2$ is sufficiently large, and to any $y' < \min(\sqrt{y}, (\log N)^c)$.

As a first application, we establish for any $y> N^{1/\sqrt{\log_9 N}}$ asymptotic results on the frequency with which an arbitrary finite complexity system of shifted linear forms
$\psi_j (\n) + a_j \in \ZZ[n_1, \dots, n_s]$, $1 \leq j \leq r$, simultaneously takes $[y',y]$-smooth values as the $n_i$ vary over integers below $N$. 

\end{abstract}

\maketitle

\tableofcontents

\section{Introduction}
Let $y > 0$ be a real number.
A positive integer $n$ is called $y$-smooth if its largest prime factor is at most $y$.
The $y$-smooth numbers below $N$ form a subset of the integers below $N$ which is, in general, sparse but enjoys good equidistribution properties in arithmetic progressions and short intervals.
These distributional properties turn $y$-smooth numbers into an important technical tool for many arithmetic questions.
As an example for one of the striking applications of smooth numbers within analytic number theory, we mention Vaughan's work \cite{Vaughan} which introduced smooth numbers in combination with a new iterative method to the study of bounds in Waring's problem. Wooley extended these methods in \cite{wooley-waring} and achieved substantial improvements on Waring's problem by working with smooth numbers. 
We refer to Granville's survey \cite{Granville-survey} (see in particular Section 6) for a more comprehensive overview of applications of smooth numbers within number theory.

In the present paper, we prove new results on the equidistribution of smooth numbers in short intervals and arithmetic progressions.
Our principal aim is to prove higher uniformity of $y$-smooth numbers in a sense that will be made precise below 
and for $y$ ranging over an almost optimal range.
In addition, we prove that provided $y$ is not too small, the set of $y$-smooth numbers is sufficiently well distributed to guarantee the existence of non-trivial solutions to arbitrary finite complexity systems of linear equations.

In order to state our main results precisely, we first introduce a subset of the $y$-smooth numbers as well as a weighted version of its characteristic function that are both central to the rest of this paper. 
Given any real numbers $0< y' \leq y$, we may consider the set of $y$-smooth numbers that are free from prime factors smaller than $y'$. 
We call such numbers $[y',y]$-smooth and denote their set by
$$S([y',y]) := \{n \in \NN: p\mid n \Rightarrow p \in [y',y]\}.$$
Given any $x > 0$, the subset of $[y',y]$-smooth numbers $\leq x$ and its cardinality are denoted by
$$
S(x,[y',y]) := S([y',y]) \cap [1,x] \quad \text{ and } \quad \Psi(x,[y',y]):= |S(x,[y',y])|~.
$$
Our notation extends the following standard notation for $y$-smooth numbers:
$$S(y) := S([1,y]), \quad S(x,y):= S(x,[1,y]), \quad \text{and} \quad  \Psi(x,y)=\Psi(x,[1,y]).$$
With this notation, we define the weighted characteristic function 
$$
g_{[y',y]}(n) = \frac{n}{\alpha(n,y)\Psi(n,[y',y])} \1_{S([y',y])}(n), \qquad (n \in \NN),
$$
of $[y',y]$-smooth numbers, where $\alpha(n,y)$ denotes the saddle point\footnote{We recall the definition in Section \ref{sec:background}.} associated to $S(n,y)$.
If $1 \leq A < W$ are coprime integers, we further define a \emph{$W$-tricked} version of $g_{[y',y]}$ by setting
$$
g_{[y',y]}^{(W,A)}(m) = \frac{\phi(W)}{W}g_{[y',y]}(Wm + A), \qquad (m \in \NN).
$$
In the definitions of the functions $g_{[y',y]}$ and $g_{[y',y]}^{(W,A)}$, the normalisations are chosen such that their average values are roughly $1$.

Following these preparations, we are now ready to state the main result of this paper.
The notation around nilsequences will be recalled in Section \ref{sec:nilsequences}.
Throughout this paper we write $\log_k x$ to denote the $k$-fold iterated logarithm of $x$.

\begin{theorem}[Higher uniformity] \label{thm:main-uniformity}
Let $N$ be a large positive parameter and let 
$K' \geq 1$, $K > 2K'$ and $d \geq 0$ be integers. 
Let $\frac{1}{2}\log_3 N \leq y' \leq (\log N)^{K'}$ and suppose that $(\log N)^K< y < N^{\eta}$ 
for some sufficiently small $\eta \in (0,1)$ depending on the value of $d$.
Let $(G/\Gamma, G_{\bullet})$ be a filtered nilmanifold of complexity $Q_0$ and degree $d$.
Finally, let $w(N) = \frac{1}{2}\log_3 N$, $W=\prod_{p<w(N)}p$ and define 
$\delta(N) = \exp(-\sqrt{\log_4 N}))$.

If $K$ is sufficiently large depending on the degree $d$ of $G_{\bullet}$, then the estimate 
\begin{align} \label{eq:thm-non-corr-bound-g-intro}
\bigg| \frac{W}{N}
\sum_{n \leq (N-A)/W} 
(g_{[y',y]}^{(W,A)}(n)-1)
F(g(n)\Gamma)
\bigg| 
\ll_{d} (1 +\|F\|_{\mathrm{Lip}}) \delta(N) Q_0 + \frac{1}{\log w(N)}
\end{align}
holds uniformly for all $1\leq A\leq W$ with $\gcd(A,W)=1$, all polynomial sequences
$g \in \mathrm{poly}(\ZZ,G_{\bullet})$ and all $1$-bounded Lipschitz functions $F:G/\Gamma \to \CC$.
\end{theorem} 

\begin{rem}
 Since $\gcd(A,W)=1$, we may, when working with $Wn+A \in S([y',y])$, restrict without loss of generality to the case $y' \geq w(N)$.
 In applications where $y$ is not too small, the contribution from $S(y)\setminus S([w(N),y])$ can often be taken care of separately, leading to a result for $S(y)$ in the end.
\end{rem}

Theorem \ref{thm:main-uniformity} constitutes the first of two parts necessary in order to establish asymptotic results on the number of $[y',y]$-smooth solutions to finite complexity systems of linear forms using the nilpotent circle method.
We emphasise that the setting of smooth numbers studied in the present paper is significantly more difficult than those settings considered in previous applications such as e.g.\ \cite{GT-linearprimes}, \cite{Mat-multiplicative, Mat-lcm} or \cite{Mat-binaryquadr}, which concern primes, a large class of multiplicative functions, and numbers representable by binary quadratic forms, respectively.
The reason for this increase in difficulty lies partly in the sparsity of the set of $y$-smooth numbers and partly in the unavailability of sieve methods to study this set.

The parameter $w(N)$ is determined by the distribution of $y$-smooth numbers in arithmetic progressions and chosen in such a way that the subset of $y$-smooth numbers in any fixed reduced residue class modulo $W(N) = \prod_{p <w(N)} p$ is sufficiently well-distributed in arithmetic progressions.
At the end of Section \ref{sec:overview}, we will discuss in more detail the need for applying such a `$W$-trick' when studying $y$-smooth numbers.
In addition to the $W$-trick,
the value $y'$ may be used to further influence how well the resulting set of $[y',y]$-smooth numbers is distributed in progressions.
Increasing parameter $y'$ beyond the value of $w(N)$ leads to better error terms for larger moduli in the distribution of $[y',y]$-smooth numbers in progressions.
The proof of Theorem \ref{thm:main-uniformity} relies on information on the distribution of $[y',y]$-smooth numbers
in short intervals and in arithmetic progressions, which we establish in Sections~\ref{sec:shortintervals} and
 \ref{sec:shortAPs}.

From a technical perspective, our focus in Theorem \ref{thm:main-uniformity} has been to establish a result in which the lower bound on the range of the smoothness parameter $y$ is as small as possible in terms of $N$, while the $W$-trick (i.e.\ the value of $W$) is still independent of $y$.
This allows one to combine this result in applications with inductive or recursive arguments in the $y$-parameter.
We remark that when focussing on larger values of $y$, the function $w(N)$ in the statement can be chosen larger and consequently the bounds in this result improve.

As a first application, we estimate, for large values of $y$, the frequency with which an arbitrary finite complexity system of shifted linear forms $\psi_j (\n) + a_j \in \ZZ[n_1, \dots, n_s]$, $1 \leq j \leq r$, simultaneously takes $[y',y]$-smooth values as the $n_i$ vary over integers below $N$.

\begin{theorem}[Linear equations in smooth numbers]\label{thm:main-theorem}
Let $r, s \geq 2$ be integers and let $N>2$ be a parameter. 
We consider the following setup:
\begin{enumerate}
 \item[(i)] Let $\psi_1, \dots, \psi_r: \ZZ^s \to \ZZ$ be linear forms that are pairwise linearly independent over 
 $\QQ$, and let $a_1, \dots, a_r \in [-N,N] \cap \ZZ$ be integers.
 Let $L$ denote the maximum of the absolute values of the coefficients of $\psi_1, \dots, \psi_r$.
 \item[(ii)] Let $\mathfrak K \subseteq [-1,1]^s$ be a fixed convex set of positive $s$-dimensional volume $\vol \mathfrak K > 0$, and suppose that
the dilated copy $$N \mathfrak K = \{ Nk \in \RR^s: k \in \mathfrak{K} \}$$ satisfies
$\psi_j(N\mathfrak K) + a_j \subseteq [1, N]$ for all $1 \leq j \leq r$.
 \item[(iii)] Let $\frac{1}{2} \log_3 N \leq y' \leq (\log N)^{K'}$ for some fixed $K' \geq 1$. 
 Let $\eta \in (0,1)$ be sufficiently small in terms of $r$ and $s$ and suppose that $y \leq N^{\eta}$.
\end{enumerate}
If $y \leq N^{\eta}$ is sufficiently large to ensure that $\Psi(N,y) > N/ \log_8 N$, then
\begin{align*}
 \sum_{\n \in \ZZ^s \cap N\mathfrak{K}} \prod_{j=1}^r g_{[y',y]} (\psi_j(\n)+a_j)
 = \vol(\mathfrak{K})N^{s}  \prod_{p < y'} \beta_p + o_{r,s,L}(N^s) ,
\end{align*}
as $N \to \infty$ and where
$$
\beta_p = \frac{1}{p^s} 
  \sum_{\mathbf{u} \in (\ZZ/p\ZZ)^s} \prod_{j=1}^r \frac{p}{p-1}
  \1_{\psi_j(\mathbf{u})+a_j \not\equiv 0 \Mod p}.
$$
\end{theorem}

We note as an aside that, as shown in Green and Tao \cite[\S 4]{GT-linearprimes}, this result implies an asymptotic count of the number of solutions in $[y',y]$-smooth numbers to systems of linear equations satisfying the non-degeneracy condition stated in \cite[Theorem~1.8]{GT-linearprimes}.
Theorem~\ref{thm:main-theorem} is the first instance of a result its kind for smooth numbers that applies in sparse situations where $\Psi(N,y) = o(N)$.
We mainly include this result for illustration and remark that we have not tried to optimise the lower bound on 
$\Psi(N,y)/N$.\footnote{As this application only involves fairly large values in $y$, the bounds in Theorem \ref{thm:main-theorem} could be improved by establishing better bounds in Theorem \ref{thm:main-uniformity} for large values of $y$. }
Theorem \ref{thm:main-theorem} follows by combining Theorem~\ref{thm:main-uniformity} with a `trivial' majorising function.
Once suitable majorising functions are available on the full range of $y$ on which Theorem \ref{thm:main-uniformity} applies, a version of Theorem~\ref{thm:main-theorem} will follow on that range of $y$, 
making Theorem~\ref{thm:main-theorem} redundant. For this reason, we chose not to optimise the bounds.

The following unweighted version of Theorem \ref{thm:main-theorem} is an easy consequence of an asymptotic lower bound on the weight factor that appears in the definition of $g_{[y',y]}(n)$.

\begin{corollary} \label{cor:main}
With the notation and under the assumptions of Theorem \ref{thm:main-theorem} the following holds.
If $y \leq N^{\eta}$ is sufficiently large to ensure that $\Psi(N,y) > N/ \log_8 N$
and $\frac{1}{2} \log_3 N \leq y' \leq (\log N)^{K'}$ for some fixed $K' \geq 1$, then 
the number $\mathcal{N}(N,\mathfrak{K})$ of $\n \in \ZZ^s \cap N\mathfrak{K}$ for which the given system of linear polynomials 
$(\psi_j(\n) + a_j)_{1 \leq j \leq r}$ takes simultaneous $[y',y]$-smooth values satisfies:
\begin{align*}
 \mathcal{N}(N,\mathfrak{K}):= 
 \sum_{\n \in \ZZ^s \cap N\mathfrak{K}} \prod_{j=1}^r \1_{S([y',y])} (\psi_j(\n)+a_j)
 \gg  \vol(\mathfrak{K}) N^{s-r} \Psi(N,[y',y])^r  \prod_{p} \beta_p
\end{align*}
for all sufficiently large $N$.
\end{corollary}

\paragraph{{\bf Previous work and discussion}}
Before turning towards Theorem~\ref{thm:main-uniformity}, we shall first discuss previous work in the direction of Theorem~\ref{thm:main-theorem}, partly because this work illustrates for which range of the smoothness parameter $y$ one can hope to prove such a result.

In the dense case where $y>N^{\eps}$ for any $\eps > 0$, Theorem~\ref{thm:main-theorem} has been proved by Lachand~\cite{Lachand}
with $S([y',y])$ replaced by $S(y)$ and with $W=1$, $A=0$. 
The special case of Lachand's result where $s = 2$ and $a_1 = \dots = a_r = 0$ also follows from work of Balog, Blomer, Dartyge and Tenenbaum \cite{BBDT}.
The latter paper studies smooth values of binary forms and can be applied to
$F(n_1,n_2) = \prod_{i=1}^r \psi_i(n_1,n_2)$.
In a similar spirit, Fouvry \cite{Fouvry} investigates smooth values of absolutely irreducible polynomials 
$F(\mathbf X) \in \ZZ[X_1, \dots, X_n]$ with $n \geq 2$ and $\deg F \geq 3$ and shows that
\begin{equation*} 
\sum_{\n \in \ZZ^n \cap [-N,N]} \1_{S(N^{d-\delta})} (F(\n)) \gg N^n 
\end{equation*}
for all $\delta < 4/3$ and all $N>N_0(F)$, which also corresponds to the dense setting.  

Concerning the sparse setting, i.e.\ smaller values of $y$, work of 
Lagarias and Soundararajan \cite{Lagarias-Sound-2011, Lagarias-Sound-2012} implies that the lower threshold on 
$y$ for which our counting function satisfies
$$\mathcal{N}(N,\mathfrak{K}) \to \infty, \quad \text{as } N \to \infty,$$
is $y> (\log N)^{\kappa}$ for some $\kappa \geq 1$.
More precisely, Lagarias and Soundararajan study primitive solutions to the equation $A+B=C$ in smooth numbers. 
Using our notation they prove in \cite[Theorem 1.5]{Lagarias-Sound-2012} that GRH implies
\begin{equation} \label{eq:lagarias-sound}
 \sum_{\substack{0<n_1+n_2 \leq N \\ \gcd(n_1,n_2)=1}} \1_{S(y)}(n_1)\1_{S(y)}(n_2)\1_{S(y)}(n_1+n_2)
\gg \frac{\Psi(N,y)^3}{N}
\end{equation}
for any $y\geq(\log N)^{\kappa}$ and $\kappa > 8$. 
On the other hand, they show in \cite[Theorem 1.1]{Lagarias-Sound-2011} that the $abc$ conjecture implies that the left hand side above is bounded independent of $N$ if $y=(\log N)^{\kappa}$ and $\kappa < 1$.
In view of this latter result, and with applications in the spirit of Theorem \ref{thm:main-theorem} in mind,
the range of $y$ in which we prove Theorem \ref{thm:main-uniformity} is optimal up to the size of the exponent $K$.
Indeed, the left hand side above is a special case of the expression $\mathcal{N}(N,\mathfrak{K})$ studied in Corollary \ref{cor:main} when ignoring the parameter $y'$.
When combined with a simple majorising function, Theorem \ref{thm:main-theorem} and Corollary \ref{cor:main} follow from Theorem \ref{thm:main-uniformity}, the Green--Tao--Ziegler inverse theorem \cite{GTZ} for the $U^k$-norms\footnote{To be precise, we require the quantitative version of the inverse theorem due to Manners \cite{Manners} in order to deduce explicit bounds from a very simple majorarising function.
A majorising function of the correct average order on a larger range of $y$ would allow us to use the original qualitative inverse theorem \cite{GTZ} instead.}
and the generalised von Neumann theorem \cite{GT-linearprimes}. 
The existence of a suitable majorising function for $g_{[y',y]}$ on a larger range of $y$ would imply an analogue of Theorem~\ref{thm:main-theorem} on the intersection of that range and the range on which Theorem~\ref{thm:main-uniformity} holds. 
Such an analogue can only hold provided that $y>(\log N)^K$ for some sufficiently large $K$.

Unconditional versions of the result \eqref{eq:lagarias-sound} of Lagarias and Soundararajan 
\cite[Theorem 1.5]{Lagarias-Sound-2012}, but with more restrictive ranges for $y$, 
have been established by a number of authors. 
In \cite[Corollary 1]{harper-minor-arcs}, Harper significantly improves those ranges and shows, almost matching the conditional range, that if $K>1$ is sufficiently large then the asymptotic formula
$$
\sum_{0<n_1+n_2 \leq N } \1_{S(y)}(n_1)\1_{S(y)}(n_2)\1_{S(y)}(n_1+n_2)
=  \frac{\Psi(N,y)^3}{2N} \Big(1 + O\Big(\frac{\log(u+1)}{\log y}\Big)\Big)
$$
holds for all $y\geq (\log N)^K$, where $u=(\log N)/\log y$.

Concerning smooth solutions to Diophantine equations in many variables, there is a vast literature on Waring's problem in smooth numbers. Building on the work of Harper \cite{harper-minor-arcs} mentioned before, Drappeau and Shao \cite{Drappeau-Shao} study Waring's problem in $y$-smooth numbers for $y\geq (\log N)^K$ and $K>1$ sufficiently large, which is the currently largest possible range of $y$. 
We refer to that paper for references to previous work on this question. 

Turning towards Theorem \ref{thm:main-uniformity}, the closest previous result is due to Lachand \cite{Lachand} who proves, starting out from a M\"obius inversion formula for $\1_{S(y)}$, that
$$
\sum_{n \leq N} \Big(\1_{S(y)} - \frac{\Psi(N,y)}{N}\Big) F(g(n)\Gamma) = o((1+\|F\|_{\mathrm{Lip}})\Psi(N,y))
$$
for any $\eps>0$ and uniformly for $x \geq y \geq \exp((\log N)/(\log \log N)^{1-\eps})$.
We shall explain at the end of Section \ref{sec:overview} why, when restricting to large values of $y$, neither a $W$-trick nor the restriction to integers with prime factors in the interval $[y',y]$ is necessary.

As exponential phases $e(\theta n^k)$ for $\theta \in \RR$, $k \in \NN$, form a special case of the nilsequences $F(g(n)\Gamma)$ in Theorem \ref{thm:main-uniformity}, previous work on this question includes all work on exponential sum (and Weyl sum) estimates over smooth numbers.
The latter is a well-studied subject.
Our proof indeed builds on some of this previous work, in particular on that of 
Harper~\cite{harper-minor-arcs}, Drappeau--Shao~\cite{Drappeau-Shao} and Wooley~\cite{Wooley}.

\begin{ack}
We would like to thank Adam Harper for very helpful comments on a previous version of this paper and the anonymous referee for their astute questions and suggestions, leading us, in particular, to rework our arguments and remove the previous restriction to large values of $y'$ ($\asymp (\log N)^C$).
While working on this paper, the authors were supported by the Swedish Research Council (VR) grant 2021-04526 and the G\"oran Gustafsson Foundation.
\end{ack}

\section{Outline of the proof and overview of the paper} \label{sec:overview}

The proof of Theorem \ref{thm:main-uniformity} occupies most of this paper.
After recalling some of the background on smooth numbers in Section \ref{sec:background}, we investigate in 
Section \ref{sec:shortintervals} the distribution of $[y',y]$-smooth numbers in short intervals and the results we prove here may be of independent interest.
The results of this section motivate the definition of the weighted function $g_{[y',y]}(n)$ as well as an auxiliary weighted function $h_{[y',y]}(n)$, and we prove that these functions are equidistributed in \emph{short intervals}.
After extending in Section~\ref{sec:harper} work of Harper on the distribution of $y$-smooth numbers in progressions, we deduce in Section~\ref{sec:W_trick_and_equid} that a $W$-tricked version of $g_{[y',y]}(n)$ is equidistributed in \emph{short arithmetic progressions}.
The work in Sections \ref{sec:shortintervals}-\ref{sec:shortAPs} (i.e.\ the fact that $g_{[y',y]}(n)$ is equidistributed in short arithmetic progressions) prepares the ground for the proof of the non-correlation estimate \eqref{eq:thm-non-corr-bound-g-intro}.
It allows us in Section~\ref{sec:nilsequences} to reduce the task of proving 
\eqref{eq:thm-non-corr-bound-g-intro} to the case in which the polynomial sequence $g$ is ``equidistributed''. 
Using a Montgomery--Vaughan-type decomposition, this task is further reduced in Section~\ref{subsec:MV}
to a Type II sum estimate. 
More precisely, we reduce our task to bounding bilinear sums of the (slightly simplified) form
$$
\sum_{\substack{m \in [y',y]}}
\sum_{\substack{N/m < n \\ \leq (N+N_1)/m}} \1_{S([y',y])}(n) \Lambda(m) F(g(mn)\Gamma).
$$
Swapping the order of summation reduces this problem to that of bounding the correlation of the ($W$-tricked) von Mangoldt functions with equidistributed nilsequences, \emph{provided} we can show an implication of the  form
$$(g(m) \Gamma)_{m \leq N} \text{ is equidistributed} 
\implies 
\begin{array}{c}
(g(mn) \Gamma)_{m \leq N/n} \text{ is equidistributed} \\ \text{for most } n \in S(N,[y',y]).
\end{array}
$$
A precise version of this implication will be established in Section~\ref{sec:linear-subsequences}.
The proof of this implication builds on the material from Sections \ref{sec:shortAPs}, \ref{sec:weyl-sums} and \ref{sec:bootstrap}. 
More precisely, it relies on a strong recurrence result for polynomial sequences over $[y',y]$-numbers that is proved in Section~\ref{sec:bootstrap}. The recurrence result in turn relies on Weyl sum estimates for $[y',y]$-smooth numbers that we establish in Section~\ref{sec:weyl-sums} by extending work of Drappeau and Shao \cite{Drappeau-Shao}.
As the Weyl sum estimates by themselves are in fact too weak for our purposes, we need to combine them with a bootstrapping argument, which requires good bounds on the distribution of smooth numbers in short arithmetic progressions (as established in Section \ref{sec:shortAPs}) as input.
The final section contains the proof of Theorem \ref{thm:main-theorem}.

The most difficult part of this work is to ensure that $y$ can be chosen as small as $(\log N)^{K}$ in our main result.
To achieve this, we need to run the bootstrapping argument in Section~\ref{sec:bootstrap} in four different stages of iterations, using the full scale of results proved in Sections \ref{sec:shortintervals} and \ref{sec:shortAPs}, and we have to choose our parameters in Sections \ref{sec:nilsequences} and \ref{sec:equid-nilsequences} very carefully.
Similarly, working with $y$ as small as $(\log N)^{K}$ required us to use a Montgomery--Vaughan-type reduction to a Type-II sum estimate in Section \ref{sec:equid-nilsequences} instead of working with the intrinsic decomposition that was used e.g.\ by Harper \cite[\S3]{harper-minor-arcs} to establish minor arc estimates for exponential sums over smooth numbers. (The intrinsic decomposition appears in the proofs of Lemmas \ref{lem:auxiliary-lemma} and \ref{lem:drappeau-shao}.)
Harper's approach, which in principle generalises to the nilsequences setting, involves an application of the Cauchy-Schwarz inequality and reduces the minor arc estimate to bounds on degree one Weyl sums 
$\sum_{n \leq N} e(n \theta)$ for irrational $\theta$.
The losses of the Cauchy-Schwarz application in this approach can be compensated by the strong bounds available for degree one Weyl sums for irrational $\theta$.
When working with an irrational (=equidistributed) nilsequence $n \mapsto F(g(n)\Gamma)$ instead of 
$n \mapsto e(\theta n)$ for irrational $\theta$, the savings on the corresponding sum 
$\sum_{n \leq N} F(g(n)\Gamma)$ are usually much weaker, and in our case too weak in order to compensate the loss of the Cauchy-Schwarz application when $y = (\log N)^K$.

\paragraph{{\bf Smooth numbers and the `$W$-trick'}}
A result of the form of Theorem \ref{thm:main-uniformity} requires as a necessary condition that the function it involves, here $g_{[y',y]}^{(W,A)}$, is equidistributed in arithmetic progressions to small moduli.
The reason for this is that additive characters $n \mapsto e(a n/q)$, where $a,q \in \NN$ and $e(x) = e^{2\pi i x}$, form a special case of the nilsequences $n \mapsto F(g(n)\Gamma)$ that appear in the statement.
The function $g_{[y',y]}^{(W,A)}$ is a weighted version of the characteristic function of $[y',y]$-smooth numbers, restricted to a reduced residue class $A$ modulo $W$.
Both the use of a $W$-trick, i.e.\ the restriction to integers of the form $n = Wm+A$, as well as the restriction to the subset of $y$-smooth numbers that are coprime to $P(y'):=\prod_{p < y'} p$ are a means to ensure equidistribution in progressions.
The weight is introduced in order to guarantees equidistribution in short intervals and progressions.

We proceed to explain why the $W$-trick is used.
Combining work of Soundararajan \cite{Sound} and Harper~\cite{harper-sound} on the distribution of smooth numbers in progressions with work of de la Bret\`eche and Tenenbaum~\cite{Breteche-Tenen-2005} one may deduce that
$$
\Psi(x,y;q,a) := \sum_{\substack{n \leq x \\ n \equiv a \Mod{q}}} \1_{S(y)}(n) 
~\sim \frac{1}{\phi(q)} \sum_{\substack{n \leq x \\ \gcd(n,q) =1 }} \1_{S(y)}(n)
\sim \frac{\Psi(x,y)}{q} \prod_{p|q} \frac{1-p^{-\alpha(x,y)}}{1-p^{-1}}
$$
for $(\log x)^2 < y \leq x$, $2 \leq q \leq y^2$, and $(a,q)=1$, as $\log x / \log q \to \infty$.

When $\alpha(x,y)$ is sufficiently close to $1$, which happens when $y$ is sufficiently close to $x$, then the final product over prime divisors $p \mid q$ will be approximately $1$ and the above asymptotic implies that $S(y)$ is equidistributed in \emph{all} residue classes (reduced and non-reduced) modulo $q$ for all small values of $q$. 
The necessary condition for the validity of Theorem \ref{thm:main-uniformity} is met in this situation and no $W$-trick is required. 

Once $\alpha(x,y)$ is no longer close to $1$, the product over prime factors $p \mid q$ in the asymptotic formula above genuinely depends on $q$.
In this case, $S(y)$ is seen to be equidistributed in the \emph{reduced} residue classes modulo a fixed integer $q$.
However, the density of $S(y)$ within a reduced class will differ from that in a non-reduced class modulo $q$, the latter being obtained by dividing out the common factor and applying the asymptotic formula with $q$ replaced by a suitable divisor of $q$.
To remove this discrepancy between reduced and non-reduced residue classes, one may use a $W$-trick with a slowly-growing function $w(N)$.
Fixing a reduced residue $A \Mod{W}$, one is then interested in the count of $y$-smooth integers of the form 
$n = W(qm+a)+A$, which satisfies
$$
\Psi(x,y; Wq, Wa + A) \sim \frac{\Psi(x,y)}{Wq} 
\bigg(\prod_{p\mid W} \frac{1-{p}^{-\alpha(x,y)}}{1-{p}^{-1}}\bigg) 
\prod_{p' \mid q, p' \nmid W } \frac{1-{p'}^{-\alpha(x,y)}}{1-{p'}^{-1}}.
$$
If $W = P(w(N))$ is the product of all primes $p < w(N)$, then all prime factors $p'$ that appear in the final product satisfy $p' \geq w(N)$, i.e.\ are large.
If $q$ is not too large, this allows one to show that the final product is asymptotically equal to $1$, with an error term that depends on $w(N)$.
This ensures, within any fixed reduced residue class $A \Mod{W}$, that the necessary condition for the validity of Theorem \ref{thm:main-uniformity} holds.

\section{Smooth numbers} \label{sec:background}
In this section we collect together general properties of smooth numbers that we will frequently make use of within this paper.
Recall the definitions of the sets $S(y)$, $S([y',y])$, $S(y,x)$ and $S([y',y],x)$ from the introduction.
The relative quantity
\begin{equation} \label{def:u}
u:=\frac{\log x}{\log y} 
\end{equation}
frequently appears when describing properties of $S(x,y)$ and $S(x,[y',y])$. 

\subsection{The summatory function of $\1_{S(y)}$}
Let 
$$\Psi(x,y) = |S(x,y)| =  \sum_{n \leq x} \1_{S(y)}(n)$$ 
denote the number of $y$-smooth numbers below $x$. 
As a function of $x$, $\Psi(x,y)$ may be viewed as the summatory function of the 
characteristic function $\1_{S(y)}$ of $y$-smooth numbers. 
The latter function is completely multiplicative and the associated Dirichlet series is given by
$$
\zeta(s,y)=\sum_{n\in S(y)}n^{-s}=\prod_{p\leq y}(1-p^{-s})^{-1}, \qquad (\Re(s)>0).
$$
Rankin's trick shows that $\Psi(x,y) \leq x^{\sigma} \zeta(\sigma,y)$ for all $\sigma >0$.
The (unique) saddle point of the function $\sigma \mapsto x^{\sigma} \zeta(\sigma,y)$ that appears here is usually denoted by $\alpha(x,y)$ and we have $\alpha(x,y) \in (0,1)$. 
Hildebrand \cite[Lemma 4]{Hildebrand-local-behaviour} (see also \cite[Lemma 2]{HilTen86}) showed that
\begin{align}\label{alpha-x-y}
\alpha(x,y)=1-\frac{\log\bigbrac{u\log (u+1)}}{\log y}+O\Bigbrac{\frac{1}{\log y}}, \qquad (\log x<y\leq x).
\end{align}
This saddle point can be used in order to asymptotically describe $\Psi(x,y)$.
More precisely, Hildebrand and Tenenbaum proved in \cite[Theorems 1-2]{HilTen86} that, uniformly for $y/\log x\to \infty$,
\begin{equation}\label{eq:Hil-Ten-Thm1}
 \Psi(x,y) 
 = \frac{x^{\alpha}\zeta(\alpha,y)}{\alpha \sqrt{2\pi\log x\log y }} 
   \Big(1+ O\Big(\frac{1}{\log (u+1)}+ \frac{1}{\log y}\Big)\Big),
\end{equation}
where $\alpha=\alpha(x,y)>0$ and $u=(\log x)/\log y$.
In applications, it is frequently necessary to understand the relation between $\Psi(cx,y)$ and $\Psi(x,y)$ as $c>0$ varies. 
In this direction, Theorem 3 of Hildebrand and Tenenbaum \cite{HilTen86} shows that
\begin{equation}\label{eq:Hil-Ten-Thm3}
 \Psi(cx,y) = \Psi(x,y) c^{\alpha(x,y)} \Big(1+ O\Big(\frac{1}{u}+ \frac{\log y}{u}\Big)\Big),
\end{equation}
uniformly for $x \geq y \geq 2$ and $1\leq c \leq y$.

The following related lemma is a consequence of \eqref{alpha-x-y}.
\begin{lemma} \label{lem:alpha-approx}
Suppose that $x >2$ is sufficiently large and that $\log x < y \leq x$. Then
 $$
 \alpha(cx,y) - \alpha(x,y) \ll 1/\log y
 $$
 uniformly for all $c \in [1,2]$.
\end{lemma}

\begin{proof}
We write $cx = x+x'$ and expand out the expression that \eqref{alpha-x-y} provides for $\alpha(x+x',y)$ when taking into account the definition \eqref{def:u} of $u$.
The lemma follows provided we can show that $\log_2(x+x') = \log_2 x + O(1)$ and
$\log_3(x+x') = \log_3 x + O(1)$.
Both of these estimates may be deduced by repeated application of the identity 
$\log (t+t') = \log t + \log (1 + t'/t)$
combined with the observation that for $t>2$ (in particular, for large $t$) the inequality
$t \geq t'$ is preserved in the sense that $T>T'$ if $T:= \log t$ and $T':=\log (1 + t'/t)$.
Finally, the bound $\log (1 + t'/t) \ll 1$ provides the shape of the error term.
\end{proof}

\subsection{The truncated Euler product}
On the positive real line, the Dirichlet series $\zeta(\sigma,y)$ can be estimated and we will frequently use the following lemma which is \cite[Lemma 7.5]{MV-book}.
\begin{lemma}\label{lem:MV}
Suppose that $y \geq 2$. If $\max\set{2/\log y,1-4/\log y}\leq\sigma\leq1$, then
\[
\prod_{p\leq y}(1-p^{-\sigma})^{-1}\asymp \log y.
\]
If $2/\log y\leq\sigma\leq 1-4/\log y$, then
\[
\prod_{p\leq y}(1-p^{-\sigma})^{-1}=\frac{1}{1-\sigma}\exp\left(\frac{y^{1-\sigma}}{(1-\sigma)\log y}
\left\{1+O\Big(\frac{1}{(1-\sigma)\log y}\Big)+O(y^{-\sigma})\right\}\right).
\]	
\end{lemma}

\subsection{The summatory function of $\1_{S([y',y])}$}

In \cite{Breteche-Tenen-2005}, de la Bret\`eche and Tenenbaum analyse the quantities 
$$\Psi_m(x,y) = \sum_{\substack{n \in S(x,y)\\ (n,m)=1}}1 \quad \text{ and } \quad
\frac{\Psi_m(x/d,y)}{\Psi(x,y)},
$$
and establish asymptotic estimates for these expressions under certain assumptions on $m$ and $d$.
Observe that
$$
\Psi_{P(y')}(x,y) = \Psi(x,[y',y])
$$
for $P(y') = \prod_{p < y'} p$.
We shall require versions of two of the results from \cite{Breteche-Tenen-2005} that are restricted to 
the case $m = P(y')$ but come with sufficiently good explicit error terms.
To state these results, define for any given integer $m \in \NN$ the restricted Euler product
\begin{align}\label{g-m}
g_m(s)=\prod_{p \mid m}(1-p^{-s}), \qquad(s\in\CC).
\end{align}
The following lemma is a consequence of de la Bret\`eche and Tenenbaum \cite[Th\'eor\`eme 2.1]{Breteche-Tenen-2005}.
\begin{lemma} \label{lem:BrTen-2.1}
Let $K'>0$ and $K>\max(2K',1)$ be constants, let $x \geq 2$ and $y' \leq (\log x)^{K'}$. 
Then 
$$
\Psi(x,[y',y]) = g_{P(y')}(\alpha(x,y)) \Psi(x,y)  \left(1 + O\left( \frac{\log_2 x}{\log x} + \frac{1}{\log y}\right)\right)
$$
uniformly for all $(\log x)^K < y \leq x$.
\end{lemma}
\begin{proof}
This result follows from Bret\`eche and Tenenbaum \cite[Th\'eor\`eme 2.1 and Corollaire 2.2]{Breteche-Tenen-2005}.
We need to verify that the error term is of the shape claimed in the statement above.
Let $m = P(y') = \prod_{p<y'}p$. Since $y' < y^{1/2}$ and $\pi(y') \ll y^{1/2}/(\log y)$, we are either in the situation $C_1$ or $C_2$ of \cite[Corollaire 2.2]{Breteche-Tenen-2005}.

Observe that in our case $\bar{u} := \min(u, y/\log y) = u = \log x/ \log y$ and 
$$W_{m} = W_{P(y')} = \log p_{\pi(y')} \asymp \log y' \ll \log \log x.$$
Hence, $\vartheta_{m} \ll (\log \log x)/\log y ~(\ll 1)$.

If $C_1$ holds, then the error term in the expression for $\Psi_{P(y')}(x,y)$ is bounded by
$$E^* := E^*_{m}(x,y) \ll \frac{\vartheta_{m} \log (u+2)}{u} \ll \frac{(\log_2 x)^2}{\log x} .$$
If condition $C_2$ holds, then $y^{1/\log(u+2)} \ll \omega(P(y')) <  y' \leq  (\log x)^{K'}$ and, hence, 
$$\log y \ll (\log_2 x) \log (u+2) \ll (\log_2 x)^2.$$
This shows that $\log(u+2) \asymp \log_2 x $ and
$$
E^* 
\ll \frac{\vartheta_{m_x}^2}{\log (u+2)} \ll \frac{(\log_2 x)^2}{(\log y)^2 \log (u+2)}
\ll \frac{1}{\log y},
$$
as required.
\end{proof}

By restricting de la Bret\`eche and Tenenbaum's \cite[Th\'eor\`eme 2.4]{Breteche-Tenen-2005} to the situation of the present paper, we obtain:
\begin{lemma} \label{lem:BrTen-2.4.i-ii}
(i) If $y' \leq (\log x)^{K'}$ for some constant $K' > 0$, then we have
$$
\Psi_{P(y')}(x/d,y) 
\ll d^{-\alpha} \Psi_{P(y')}(x,y)
$$
uniformly for all $\max(y'^2,(\log x)^2) < y \leq x$, $1 \leq d \leq x/y$.

(ii) If, moreover, $1 \leq d \leq \exp((\log x)^{1/3})$, the following more precise statement holds.
We have
\begin{align*}
\Psi_{P(y')}(x/d,y) 
&= d^{-\alpha} \Psi_{P(y')}(x,y) \Big(1 + O\Big(\frac{\log_3 x}{\log_2 x}\Big)\Big)
\end{align*}
uniformly for all $\max(y'^2,(\log x)^2) < y \leq x$ and $1 \leq d \leq \exp((\log x)^{1/3})$, 
where $y'\leq (\log x)^{K'}$ and $K'>0$ as before.
\end{lemma}

\begin{proof}
 Part (i) is a direct application of \cite[Th\'eor\`eme 2.4 (i)]{Breteche-Tenen-2005} combined with Lemma \ref{lem:BrTen-2.1}, when taking into account that 
 $\delta = 0$ according to the remark following (2.8) of that paper, since $y> \log^2 x$.
 
 Concerning part (ii), Lemma \ref{lem:BrTen-2.1} implies that the asymptotic formula given in the statement is equivalent to
 $$
 \Psi_{P(y')}(x/d,y) 
 = d^{-\alpha}g_{P(y')}(\alpha(x,y)) \Psi(x,y) \Big(1 + O\Big(\frac{\log_3 x}{\log_2 x}\Big)\Big).
 $$
 The statement thus follows from \cite[Th\'eor\`eme 2.4 (ii)]{Breteche-Tenen-2005} provided we can show that under our assumptions the error term in  \cite[Th\'eor\`eme 2.4 (ii)]{Breteche-Tenen-2005} is in fact of the above stated shape.
 In the notation of \cite[Th\'eor\`eme 2.4 (ii)]{Breteche-Tenen-2005} we therefore need to show that
 $h_m = o(1)$ and $(1-t^2/(2u^2))^{bu} = 1+o(1)$ as $x \to \infty$, where $o(1)$ needs to be made explicit in both cases.
 In the former case, we have
$$
 h_m 
  = \frac{1}{u_y} + t \frac{(1+E_m)}{u} + E^*_m(x,y) 
  \ll \frac{\log_2(2y)}{\log y} + \frac{1}{\log_2 x} + \frac{\log d}{\log x} + \frac{\log d}{\log y}\frac{E_m}{u},
$$
where we used that (cf.\ \cite[(2.18)]{Breteche-Tenen-2005}) 
$u_y = u + (\log y)/\log(u+2) \gg (\log y)/\log_2 (2y)$ and that $E^*_m(x,y) \ll 1/(\log_2 x)$ by the previous proof. For the term $(1+E_m)/u$, we obtain
$$
E_m = \frac{\vartheta_m (u \log (u+2))^{\vartheta_m}}{1 + \vartheta_m \log(u+2)}
$$
Note that $\vartheta_{m} = (\log p_{\omega(m)})/\log y = (\log y')/\log y \leq 1/2$ for $y\geq y'^2$. Further, $u \log (u+2) \geq 1$. Hence,
$$
\frac{E_m}{u} \ll \frac{\vartheta_m  \log^{1/2}(u+2)}{u^{1/2}(1 + \vartheta_m \log(u+2))}
$$
If $\log y > (\log_2 x)^3$, then
$$
\frac{\log d}{\log y}\frac{E_m}{u} \ll \frac{\log d}{\log y} \frac{\log^{1/2} (u+2)}{u^{1/2}(\log_2 x)^2} 
\ll  \frac{(\log d)(\log_2 x)^{1/2}}{(\log x)^{1/2} (\log y)^{1/2} \log_2 x} 
\ll \frac{\log d}{(\log x)^{1/2}(\log_2 x)^{2}}.
$$
If $\log y \leq (\log_2 x)^3$ and since $\log_2 x \ll \log y$, we have
\begin{align*}
t\frac{E_m}{u} 
&\ll t\frac{\vartheta_m  \log^{1/2} (u+2)}{u^{1/2}(1 + \vartheta_m \log(u+2))} 
\ll t \frac{\vartheta_m  (\log_2 x)^{1/2} (\log y)^{1/2}}{(\log x)^{1/2}(1 + \frac{1}{(\log_2 x)^2}\log(u+2))} \\
&\ll t \frac{(\log_2 x)^2}{(\log x)^{1/2}}
\ll \frac{\log d}{\log y} \frac{(\log_2 x)^2}{(\log x)^{1/2}} 
\ll  \frac{(\log d) \log_2 x}{(\log x)^{1/2}}.
\end{align*}

It remains to analyse $(1-t^2/(2u^2))^{bu}$. 
Since $t/u = (\log d) / \log x = o(1)$, we have

$$
 \Big(1-\frac{t^2}{2u^2}\Big)^{bu} = \exp\Big(-C\frac{t^2}{u}\Big)
 = 1 + O(1/\log y)
$$
for some constant $C>0$.
The final step above follows since 
$$
\frac{t^2}{u} 
= \frac{(\log d)^2}{(\log y) \log x}
\ll \frac{1}{\log y}
$$
by our assumption on $d$.
\end{proof}

\section{Smooth numbers in short intervals} \label{sec:shortintervals}

Our aim in this section is to show that a suitably weighted version of the function $\1_{S([y',y])}$ is equidistributed in short intervals.
More precisely, we will consider intervals contained in $[x,2x]$ that are of length at least $x/(\log x)^c$.
The error terms in the Lemmas \ref{lem:BrTen-2.1} and \ref{lem:BrTen-2.4.i-ii} are too weak in order to allow one to work with intervals as short as $x/(\log x)^c$. 
Better error terms are, however, available when considering the quantity 
$$
\frac{\Psi_{P(y')}(x/d,y)}{\Psi_{P(y')}(x,y)} = \frac{\Psi(x/d,[y',y])}{\Psi(x,[y',y])} \quad \text{ instead of } \quad \frac{\Psi_{P(y')}(x/d,y)}{\Psi(x,y)}~.
$$
In the first subsection below we prove asymptotic formulas for the former quotient.
In the second subsection we introduce a suitable smooth weight for $\1_{S([y',y])}$ and prove that the correspondingly weighted $[y',y]$-smooth numbers are equidistributed in short intervals of the above form. 

\subsection{The local behaviour of $\Psi(x,[y',y])$}
In order to analyse the distribution of $S([y',y])$ in short intervals, we need to compare
$\Psi(x,[y',y])$ to $\Psi(x(1+1/z),[y',y])$. 
In the notation of Lemma \ref{lem:BrTen-2.4.i-ii} this means that our $d$ is very close to $1$.
Recall that $u = (\log x)/ \log y$.
We split our analysis below into two case according to whether $u$ is large or small.

\begin{lemma}[Local behaviour of $\Psi_{P(y')}$ for large $u$]  \label{lem:Psi_m-approx}
Let $x > 2$ be a parameter, let $K' >0$ and $K \geq \max(2K',2)$ be constants and suppose that 
$y' \leq (\log x)^{K'}$ and $(\log x)^{K} < y < \exp( (\log x)^{1/3})$. 
Let $\alpha = \alpha(x,y)$ denote the saddle point associated to $S(x,y)$. 
Then
\begin{align*}
\left|\Psi(x,[y',y]) - \left(1 + z^{-1}\right)^{-\alpha} \Psi\left(\left(1 + z^{-1}\right)x,[y',y] \right) \right|
\ll  \Psi(x,[y',y]) \left(\frac{\sqrt{\log y}}{z \log^{1/6} x} + \frac{O_A(1)}{\log^A x}\right),
\end{align*}
for all $z > 1$.
\end{lemma}

\begin{proof}
 We will follow the strategy via Perron's formula employed in \cite[Lemma 8]{HilTen86} and \cite[\S4.2]{Breteche-Tenen-2005}, and start by bounding Euler factors from the relevant Dirichlet series.
 Suppose that $s=\alpha+i\tau$ and $(\log x)^{-1/4} \leq \tau \leq (\log y)^{-1}$. Then\footnote{Cf.\ the proof of \cite[Lemma 8]{HilTen86} for more details.} for any $y' \leq p \leq y$, we have
 \begin{align*}
  \frac{|1-p^{-\alpha}|}{|1-p^{-s}|}
  &= \left(1 + \frac{2(1-\cos(\tau \log p))}{p^{\alpha}(1-p^{-\alpha})^2}\right)^{-1/2}\\
  &\leq \exp \left(-c_1\frac{1-\cos(\tau \log p)}{p^{\alpha}(1-p^{-\alpha})^2}\right)
   \leq \exp \left(-c_2 \frac{(\tau \log p)^2}{p^{\alpha}}\right),
 \end{align*}
where we took advantage of the fact that $1-p^{-\alpha} \asymp 1$ for all $p \leq y$ since $\alpha > 1/2$.
Hence,
\begin{align} \label{eq:euler-prod-bd}
\nonumber
 \prod_{y' \leq p \leq y} \frac{|1-p^{-\alpha}|}{|1-p^{-s}|}
 &\leq \exp \left(-c_2 \frac{\tau^2}{y^\alpha} \sum_{y' \leq p \leq y} (\log p)^2\right)
 \leq \exp \left(-c_2 \frac{\tau^2}{y^\alpha} \sum_{y/2 \leq p \leq y} (\log p)^2\right) \\
\nonumber 
 &\leq \exp \left(-c_3 \tau^2 y^{1-\alpha} \log y \right)
  \leq \exp \left(-c_4 \tau^2 u \log(u+1) \log y \right)\\
 &\leq \exp \left(-c_4 \tau^2 \log x  \right).
\end{align}
By invoking bounds of the above Euler factors in different regimes of $\tau$, Bret\`eche and Tenenbaum \cite[(4.50)]{Breteche-Tenen-2005} deduce from Perron's formula that
\begin{equation} \label{eq:perron-approx}
\Psi(x,[y',y]) 
= \frac{1}{2 \pi i} \int_{\alpha - i/\log y}^{\alpha + i/\log y} \zeta_m(s,y)x^{s} \frac{ds}{s}
+ O(x^{\alpha} \zeta_m(\alpha,y)R),
\end{equation}
where $R = \exp(-c_5 u/(\log 2u)^2) + \exp(-(\log y)^{4/3})$, $m = P(y') = \prod_{p<y'}p$, and
$$\zeta_m(s,y) = \prod_{p\leq y, p \nmid m}(1-p^{-s})^{-1}.$$
The approximation \eqref{eq:perron-approx} applies in our situation since $u > (\log x)^{3/2} > (\log y)^{3/2}$ and 
$\pi(y') = \omega(P(y')) < (\log x)^{K'} < \sqrt{y}$, which ensures that the conditions \cite[(4.40)]{Breteche-Tenen-2005} are satisfied.
Under our assumption that $\log x < y < \exp( (\log x)^{1/3})$, we have 
$$
R \ll \exp(-c_5(\log x)^{1/2}) + \exp(-(\log \log x)^{4/3})
\ll_A (\log x)^{-A}.
$$
The bound \eqref{eq:euler-prod-bd} allows us to further truncate the integral in \eqref{eq:perron-approx}.
More precisely,
\begin{align*}
\left|\int_{\alpha \pm i/(\log x)^{1/3}}^{\alpha \pm i/\log y} \zeta_m(s,y) x^{s} \frac{ds}{s} \right|
&\leq \zeta_m(\alpha,y)\frac{x^{\alpha}}{\alpha} 
     \int_{1/(\log x)^{1/3}}^{1/\log y} \exp \left(-c_4 \tau^2 \log x  \right) d\tau \\
&\leq \frac{\zeta_m(\alpha,y) x^{\alpha}}{\alpha \log y } \exp \left(-c_4 (\log x)^{-1/3}  \right) \\
&\leq \Psi(x,[y',y]) \exp \left(-c_6 (\log x)^{-1/3}  \right),
\end{align*}
which implies that
\begin{equation} \label{eq:perron-approx-2}
\Psi(x,[y',y]) 
= \frac{1}{2 \pi i} \int_{\alpha - i/(\log x)^{1/3}}^{\alpha + i/(\log x)^{1/3}} \zeta_m(s,y)x^{s} \frac{ds}{s}
+ O_A\left(\frac{\Psi(x,[y',y])}{(\log x)^{A}}\right).
\end{equation}
The latter approximation then yields:
\begin{align*}
 &\Psi(x,[y',y]) - (1+z^{-1})^{-\alpha} \Psi\left(x(1+z^{-1}),[y',y]\right) \\
 &= \frac{1}{2 \pi i} \int_{\alpha - i/(\log x)^{1/3}}^{\alpha + i/(\log x)^{1/3}} \zeta_m(s,y)x^{s}
 \left(1-\left(1+z^{-1}\right)^{s-\alpha}\right) \frac{ds}{s}
+ O_A\left(\frac{\Psi(x,[y',y])}{(\log x)^{A}}\right)\\
 &\ll  \zeta_m(\alpha,y)x^{\alpha} \int_{-1/(\log x)^{1/3}}^{1/(\log x)^{1/3}} 
 \left|1-\left(1+z^{-1}\right)^{i\tau}\right| \frac{d\tau}{\alpha}
+ O_A\left(\frac{\Psi(x,[y',y])}{(\log x)^{A}}\right)\\
&\ll  \frac{\zeta_m(\alpha,y)x^{\alpha}}{z} \int_{-1/(\log x)^{1/3}}^{1/(\log x)^{1/3}} 
 \left|\tau \right| d\tau
+ O_A\left(\frac{\Psi(x,[y',y])}{(\log x)^{A}}\right)\\
&\ll \frac{\zeta_m(\alpha,y)x^{\alpha}}{z (\log x)^{2/3}} 
+ O_A\left(\frac{\Psi(x,[y',y])}{(\log x)^{A}}\right) \\
&\ll \Psi(x,[y',y]) \left(\frac{\sqrt{\log y}}{z (\log x)^{1/6}} + O_A\left( (\log x)^{-A} \right)\right),
\end{align*}
as claimed.
\end{proof}

The previous lemma applies to $y<\exp((\log x)^{1/3})$.
Below, we establish an analogues result in the complementary range where $\exp((\log x)^{1/4})\leq y\leq x$.
In this case, $1-\alpha(x',y)$ is very small as $x'$ ranges over $[x,2x]$, and the error terms in Lemma \ref{lem:BrTen-2.1} are also very good. 
This allows us to reduce the problem of asymptotically evaluating $\Psi(x(1+z^{-1}),[y',y])$ to that of bounding
$\Psi(x(1+z^{-1}),y)-\Psi(x,y)$. 
Bounds on the latter difference have been established by Hildebrand \cite[Theorem 3]{HilTen86} in the case where 
$u = (\log x)/\log y$ is small, which applies to our current situation.\footnote{Hildebrand's theorem has been extended by several authors and we refer to Granville's survey paper \cite[Section 4.1]{Granville-survey} for a discussion of these extensions and references.}

\begin{lemma}[Local behaviour of $\Psi_{P(y')}$ for small $u$] \label{lem:small-u}
Let $x > 2$ be a parameter, suppose that $y'\leq(\log x)^{K'}$ for some fixed $K'>0$ and that
$\exp\bigbrac{(\log x)^{1/4}}\leq y\leq x$, and let 
$\alpha=\alpha(x,y)$ denote the saddle point associated to $S(x,y)$. Then
\begin{equation} \label{eq:small-u}
\bigabs{\Psi(x(1+z^{-1}),[y',y])-(1+z^{-1})^\alpha\Psi(x,[y',y])}
\ll\Psi(x,[y',y])\frac{(\log_2x)^2}{\log y} 
\end{equation}
holds uniformly for $1\leq z\leq y^{5/12}$.
\end{lemma}

\begin{proof}
Under the assumptions on $y$ it follows from \eqref{alpha-x-y} that the saddle point $\alpha(x,y)$ is very close to $1$. More precisely, writing $u' = (\log x')/\log y$, we have 
\begin{equation} \label{eq:alpha-estimate}
\alpha(x',y)=1-\frac{\log (u'\log(u'+1))}{\log y}+O\Bigbrac{\frac{1}{\log y}}
=1+O\Bigbrac{\frac{\log_2x'}{\log y}}
=1+O\Bigbrac{\frac{\log_2 x}{(\log x)^{1/4} }}
\end{equation}
uniformly for all $x' \in [x,2x]$ and $y \geq \exp\bigbrac{(\log x)^{1/4}}$. 
This implies, in particular, that
\[
(1+z^{-1})^\alpha=(1+z^{-1})(1+z^{-1})^{O((\log_2x/\log y)}=(1+z^{-1})\Bigbrac{1+O\Bigbrac{\frac{\log_2x}{z\log y}}}
\]
for $\alpha = \alpha(x,y)$.
Substituting this formula into \eqref{eq:small-u} and rearranging the resulting expression reduces our task to that of establishing:
\begin{align*}\label{goal-short-interval}
\Psi\left(x(1+z^{-1}),[y',y]\right)-\Psi(x,[y',y])
=\frac{\Psi(x,[y',y])}{z}\Bigbrac{1+O\bigbrac{\frac{z(\log_2x)^2}{\log y}}}.
\end{align*}
By Lemma \ref{lem:BrTen-2.1}, the difference on the left hand side satisfies
\begin{equation}\label{apply-41}
\begin{aligned}
\Psi\left(x(1+z^{-1}),[y',y]\right)-\Psi(x,[y',y])&= g_{P(y')}(\alpha_z)\Psi(x(1+z^{-1}),y)-g_{P(y')}(\alpha)\Psi(x,y)\\
&\quad+O\Bigbrac{\Psi(x,[y',y])\Bigbrac{\frac{\log_2x}{\log x}+\frac{1}{\log y}}},
    \end{aligned}
\end{equation}
where $\alpha_z=\alpha(x(1+z^{-1}),y)$ is the saddle point associated to $S(x(1+z^{-1}),y)$. 
In order to simplify the main term above, we seek to relate $g_{P(y')}(\alpha_z)$ to $g_{P(y')}(\alpha)$. 
Following \cite[\S3.6]{Breteche-Tenen-2005}, define $\gamma_m(s):=\log \,g_m(s)$ and let $\gamma_m'(s)$ denote the derivative with respect to $s$. 
Since $\alpha_z\leq\alpha$, we then have
\[
\log \left(\frac{g_{P(y')}(\alpha_z)}{g_{P(y')}(\alpha)}\right)
= \int_\alpha ^{\alpha_z} \gamma_{P(y')}'(\sigma) \,\mathrm d\sigma
\leq (\alpha-\alpha_z) \sup_{\alpha_z \leq \sigma \leq \alpha} \gamma_{P(y')}'(\sigma),
\]
and shall estimate the latter two factors in turn.  
Since $1 < 1 + z^{-1} < 2$ for $z > 1$, it follows from \eqref{eq:alpha-estimate} that
$$
\alpha- \alpha_z \ll \frac{\log_2 x}{\log^{1/4} x}.
$$
Concerning the second factor, we have 
\[
\gamma_{P(y')}'(\sigma) = \sum_{p\leq y'} \frac{\log p}{ p^\sigma -1} 
\ll \sum_{n\leq y'} \frac{\Lambda (n)}{n^\sigma} 
\ll {y'}^{1-\sigma} \sum_{n\leq y'} \frac{\Lambda (n)}{n}
\ll \log y' \ll \log_2 x
\] 
since $\log y' \ll \log_2 x$ and $1-\sigma \ll_{\eps} \log^{-1/4 + \eps} x$ by \eqref{eq:alpha-estimate}.
Hence
\[
\frac{g_{P(y')}(\alpha_z)}{g_{P(y')}(\alpha)} 
= \exp \left( \log \left(\frac{g_{P(y')}(\alpha_z)}{g_{P(y')}(\alpha)}\right) \right)
= \exp \left( O \left(\frac{(\log_2 x)^2}{ \log^{1/4}x}\right) \right)
= 1+  O \left(\frac{(\log_2 x)^2}{ \log^{1/4}x}\right).
\]
Substituting this expression into (\ref{apply-41}), we obtain
\begin{align*}
&\Psi(x(1+z^{-1}),[y',y])-\Psi(x,[y',y]) \\
&=g_{P(y')}(\alpha)\bigset{\Psi(x(1+z^{-1}),y)-\Psi(x,y)}+O\Bigbrac{\Psi(x,[y',y])\frac{(\log_2x)^2}{\log y}},
\end{align*}
where we used Lemma \ref{lem:BrTen-2.1} and \eqref{eq:Hil-Ten-Thm1} to simplify the error term.
To the right hand side, we may now apply Hildebrand's \cite[Theorems 1 and 3]{Hildebrand} which assert that 
\[
\Psi(x(1+z^{-1}),y)-\Psi(x,y)
=\frac{\Psi(x,y)}{z}\left\{1+O_{\eps}\left(\frac{\log(u+1)}{\log y}\right)\right\}
\]
holds uniformly for $\exp\bigbrac{(\log_2x)^{5/3+\eps}}\leq y\leq x$ and $1\leq z\leq y^{5/12}$.
Thus, under our assumptions on $z$ and $y$,
\begin{align*}
\Psi(x(1+z^{-1}),[y',y])-\Psi(x,[y',y])
=\frac{g_{P(y')}(\alpha)\Psi(x,y)}{z}\left\{1+O\left(\frac{\log(u+1)}{\log y}\right)\right\}
\\+O\Bigbrac{\Psi(x,[y',y])\frac{(\log_2x)^2}{\log y}}.
\end{align*}
Another application of Lemma \ref{lem:BrTen-2.1} to the right-hand side yields the desired result.
\end{proof}

\subsection{Equidistribution of weighted $[y',y]$-smooth numbers in short intervals} \label{subsec:shortintervals}

The correction factors in Lemmas \ref{lem:Psi_m-approx} and \ref{lem:small-u} suggest to consider the smoothly weighted version 
$$
n \mapsto n^{1-\alpha(x,y)} \1_{S([y',y])}(n)
$$
of the characteristic function of $[y',y]$-smooth numbers.
This choice of weight allows us to deduce from these lemmas that the new weighted function is equidistributed 
in all short intervals of length at least $x(\log x)^{-c}$, a property that is required in 
Section \ref{sec:bootstrap} and, indirectly, in Section  \ref{sec:nilsequences}.
For simplicity, we further normalise our weighted function to mean value $1$ in the following lemma.

\begin{lemma}[Equidistribution in short intervals] \label{lem:short-int}
Let $N>2$ be a parameter and define the weighted and normalised function
\begin{equation} \label{def:h_y}
 h_{[y',y]}(n) := \frac{N^{\alpha}}{\Psi(N,[y',y])} \frac{n^{1-\alpha}}{\alpha} \1_{S([y',y])}(n), 
 \qquad (N\leq n\leq 2N),
\end{equation}
where $\alpha = \alpha(N,y)$. Let $y' \leq (\log N)^{K'}$ for some fixed $K'>0$, let $K = \max(2, 2K')$, and suppose that 
$(\log N)^K <y \leq N$.
Then 
\begin{align*}
 \sum_{N_0 < n \leq N_0 + N_1} h_{[y',y]}(n) 
 = N_1 \biggset{1+O\left(\frac{\log_3N}{\log_2N}\right)} 
 + O\left(\frac{N}{\log^{1/24} N}\right).
\end{align*}
uniformly for all  
$N\leq N_0 < N_0 + N_1 \leq 2N$ such that $N_1 \gg N \exp(-(\log N)^{1/4}/4)$.
\end{lemma}

\begin{proof}
By partial summation,
\begin{align} \label{eq:partialsum}
\nonumber
 \sum_{\substack{n \in S([y',y])\\N_0 < n \leq N_0 + N_1}} n^{1-\alpha}
&= (N_0 + N_1)^{1-\alpha} \Psi(N_0 + N_1,[y',y])- N_0^{1-\alpha}\Psi(N_0,[y',y]) \\
& \qquad - (1-\alpha) \int_{N_0}^{N_0 + N_1} \frac{\Psi(t,[y',y])}{t^{\alpha}}~\mathrm dt .
\end{align}
We seek to bound the right hand side with the help of Lemmas \ref{lem:Psi_m-approx} and \ref{lem:small-u}.
To start with, suppose that $y < \exp ((\log N)^{1/4})$.
For $N\leq N_0<t \leq N_0+N_1\leq 2N$, we have $t= N(1 + 1/z)$ where $z>1$ and 
$(t/N)^{-\alpha} \asymp 1$. Hence,
\begin{align*}
\Psi(t,[y',y])
&=(t/N)^{\alpha} \Psi(N,[y',y]) \left(1+ O\left(\frac{\sqrt{\log y}}{z (\log N)^{1/6}} + O_A\left((\log N)^{-A}\right) \right)\right)\\
&=(t/N)^{\alpha} \Psi(N,[y',y]) \Big(1+ O\Big((\log N)^{-1/24}\Big)\Big)
\end{align*}
for $y < \exp ((\log N)^{1/4})$
by Lemma \ref{lem:Psi_m-approx}.
Thus,
\begin{align*}
\frac{N^{\alpha}}{\Psi(N,[y',y])} 
\sum_{\substack{n \in S([y',y])\\N_0 < n \leq N_0 + N_1}} n^{1-\alpha} 
&= (N_0 + N_1)  - N_0  
 - (1-\alpha) \int_{N_0}^{N_0 + N_1} 1 ~dt + O\left(\frac{N}{(\log N)^{1/24}}\right)\\
&= \alpha N_1 + O\left(\frac{N}{(\log N)^{1/24}}\right),
\end{align*}
which establishes the lemma when $y < \exp ((\log N)^{1/4})$.

Suppose next that $\exp ((\log N)^{1/4})\leq y \leq N$. 
In view of the restriction on the size of $z$ in Lemma \ref{lem:small-u}, we cannot apply this lemma in the current situation with $x = N$ (which would correspond to the choice in the first part of the proof),
but need to choose $x = N_0$ instead.
For this purpose, let $\alpha'=\alpha(N_0,y)$ denote the saddle point associated to $S(N_0,y)$ and note that
$$|\alpha'- \alpha| = |\alpha(N_0,y) - \alpha(N,y)| \ll (\log y)^{-1} \ll (\log N)^{-1/4}$$
by Lemma \ref{lem:alpha-approx} and the lower bound on $y$.

For $N\leq N_0<t \leq N_0+N_1$ with $N \exp(- (\log N)^{1/4}/4) \leq N_1 \leq N$, we have 
$$t= N_0(1 + 1/z), \quad \text{where} \quad 
1 < z \ll \exp((\log N)^{1/4}/4) = o(y^{5/12}),
$$ 
and $(t/N_0)^{-\alpha'} \asymp 1$. 
Hence, Lemma \ref{lem:small-u} applies and yields
\begin{align} \label{eq:aux1}
\Psi(t,[y',y]) =(t/N_0)^{\alpha'} \Psi(N_0,[y',y]) \Big(1+ O\Big(\frac{(\log_2 N)^2}{\log^{1/4} N}\Big)\Big).
\end{align}
Each of the three terms arising in the partial summation expression \eqref{eq:partialsum} involves a weighted count of the form
$\Psi(t,[y',y]) / t^{\alpha}$.
With this and \eqref{eq:aux1} in mind, observe that
\begin{align*}
 (t/N_0)^{\alpha'}t^{-\alpha}
&= (t/N_0)^{\alpha'-\alpha} N_0^{-\alpha}
= N_0^{-\alpha} (1 + 1/z)^{\alpha' - \alpha}
= N_0^{-\alpha} (1 + O(|\alpha-\alpha'|)) \\
&= N_0^{-\alpha} \Big(1 + O(\log N)^{-1/4}\Big),
\end{align*}
since 
$$
|(1 + 1/z)^{a} - 1^{a}|
\leq \int_1^{1 + 1/z} |a| t^{a-1} dt
\leq \frac{|a|}{z} \leq |a|
$$
for $0<|a|<1$, which we applied with $a = \alpha'-\alpha \ll (\log N)^{-1/4}$.

Hence, it follows from \eqref{eq:partialsum} and \eqref{eq:aux1} that
\begin{align*}
\frac{N_0^{\alpha}}{\Psi(N_0,[y',y])} 
\sum_{\substack{n \in S([y',y])\\N_0 < n \leq N_0 + N_1}} n^{1-\alpha} 
&= (N_0 + N_1)  - N_0  
 - (1-\alpha) \int_{N_0}^{N_0 + N_1} 1 ~dt + O\left(\frac{N (\log_2 N)^2}{(\log N)^{1/4}}\right)\\
&= \alpha(N,y) N_1 + O\left(\frac{N (\log_2 N)^2}{(\log N)^{1/4}}\right). 
\end{align*}
The lemma then follows since
$$
\frac{(N_0/N)^{-\alpha}\Psi(N_0,[y',y])}{\Psi(N,[y',y])} 
= 1 + O\left(\frac{\log_3 N}{\log_2 N}\right)
$$
by Lemma \ref{lem:BrTen-2.1} and part (ii) of Lemma \ref{lem:BrTen-2.4.i-ii}.
\end{proof}

Working with functions $h_{[y',y]}: [N,2N] \to \RR$ whose support is restricted to a dyadic interval is too restrictive for the purposes of our main theorem.
The following lemma provides an intrinsically defined function $g_{[y',y]}: \NN \to \RR$ 
that approximates any function $h_{[y',y]}$ on the interval $[N,2N]$ where the latter function is defined.

\begin{lemma} \label{lem:g-h-approximation}
Define
\begin{equation} \label{eq:g-def}
  g_{[y',y]}(n) = \frac{n}{\alpha(n,y)\Psi(n,[y',y])} \1_{S([y',y])}, \qquad (n \in \NN),
\end{equation}
and suppose that $N \leq n < 2N$ and that $h_{[y',y]}$ is define on $[N,2N]$.
Then
$$
g_{[y',y]}(n) = h_{[y',y]}(n) \left(1+O\left(\frac{\log_3 N}{\log_2 N}\right)\right),
$$
provided that $N$ is sufficiently large.
\end{lemma}
\begin{proof}
 This follows immediately from Lemmas \ref{lem:alpha-approx} and \ref{lem:BrTen-2.4.i-ii} (ii).
\end{proof}

\section{Major arc analysis: Smooth numbers in arithmetic progressions and a $W$-trick} \label{sec:shortAPs}
Extending previous work of Soundararajan \cite{Sound}, Harper \cite{harper-sound} proved
that for $y \leq x$, $\eps>0$, $2 \leq q \leq y^{4\sqrt{e}-\eps}$, and $(a,q)=1$,
$$
\Psi(x,y;q,a) := \sum_{\substack{n \leq x \\ n \equiv a \Mod{q}}} \1_{S(y)}(n) ~\sim \frac{1}{\phi(q)} \Psi_q(x,y)
$$
as $\log x / \log q \to \infty$ provided that $y$ is sufficiently large in terms of $\eps$.
This shows that, for $y \in (\log x, x]$, the $y$-smooth numbers up to $x$ are equidistributed within the reduced residue classes to a given modulus $q$ provided $x$ is sufficiently large.
Our aim in this section is to construct a subset of the $y$-smooth numbers with the property that its elements are equidistributed in the reduced residue classes of \emph{any} modulus $q \leq (\log x)^c$. 
More precisely, we will show that, after applying a $W$-trick, the $[y',y]$-smooth numbers for any $y' \leq (\log N)^{K'}$, $K'>0$, have this property.
In addition, we will construct a suitable weight for this subset which allows us to retain equidistribution when we restrict any of the above progressions to a shorter interval of suitable length.
These equidistribution properties will be required in order to reduce in Section \ref{sec:nilsequences} the task of establishing a non-correlation estimate with nilsequences to the case where the nilsequence is highly equidistributed.

\subsection{Smooth numbers in arithmetic progressions} \label{sec:harper}
In this subsection we extend Harper's result \cite[Theorem 1]{harper-sound} and show that numbers without small and large prime factors are equidistributed in progressions. More precisely, defining
$$
\Psi(x,[y',y];q,a) = \sum_{\substack{n \leq x \\ n \equiv a \Mod{q}}} \1_{S([y',y])}(n),
$$
we show the following:

\begin{theorem} \label{thm:equid-in-APs}
Let $K', K'' >0$ and $K > \max(2K',2)$ be constants, let $x \geq 2$ be a parameter and let
$1 \leq y' \leq (\log x)^{K'}$ and $(\log x)^{K} < y \leq x$. 
If $q < (\log x)^{K''}$, $p \mid q \Rightarrow p<y'$ and $(a,q)=1$, then
$$
\Psi(x,[y',y];q,a) 
= \frac{\Psi(x,[y',y])}{\phi(q)}  
\Big(1 + O\Big( \log^{-1/5} x \Big) \Big),
$$
provided that $K$ and $K/K''$ are sufficiently large.
\end{theorem}

\begin{rem*}
 The theorem generalises to the case where $q \leq (\log x)^{K''}$ but $p \mid q \not\Rightarrow p<y'$.
 In this case, $\Psi(x,[y',y])$ needs to be replaced by $\Psi_q(x,[y',y])$.
\end{rem*}

As we require explicit error terms, we do not follow Harper's proof strategy from \cite{harper-sound}, but instead start by establishing a version of the Perron-type bound given in Harper \cite[Proposition 1]{harper-BV} that applies to 
$$\Psi(x,[y',y],\chi):= \sum_{n \leq x} \chi(n) \1_{S([y',y])}~.$$

\begin{proposition} \label{prop:smooth-perron}
There exist a small absolute constant $\rho \in (0,1)$, and a large absolute
constant $C > 1$, such that the following is true.
Let $K'>0$ and $K > \max(2K, 2)$ be constants, let $1 \leq y' \leq (\log x)^{K'}$ and $(\log x)^{K} < y \leq x$, and suppose that $x$ is large. Suppose that $\chi$ is a non-principal Dirichlet character
with conductor $r \leq x^{\rho}$, and to modulus $q \leq x$, such that $L(s, \chi)=\sum_n \chi(n) n^{-s}$ has no zeros
in the region
$$\Re(s) > 1 - \eps, \quad |\Im(s)| \leq H,$$
where $C/ \log y < \eps \leq \min(\alpha/2,1/(2K'))$ and $y^{0.9\eps} \log^2 x \leq H \leq \min(x^{\rho},y^{5/6})$. 
Suppose, moreover, that at least one of the following holds:
\begin{itemize}
\item[(i)] $y \geq (Hr)^C$;
\item[(ii)] $\eps \geq 40 \log \log(qyH) /\log y$.
\end{itemize}
Then we have the bound
$$
\Psi(x,\chi;[y',y])
\ll\Psi(x,[y',y]) \bigg( (\log x)^{-1/5} + \sqrt{\log x\log y}\Big(H^{-1/2}+x^{-0.4\eps}\log H \Big)\bigg)~.
$$
\end{proposition}
\begin{rem}
 The $(\log x)^{-1/5}$ term in the bound can be omitted when $y<\exp(\log^{1/4}x)$, i.e.\ when the proof involves an application of Lemma \ref{lem:Psi_m-approx} and avoids Lemma \ref{lem:small-u}.
\end{rem}

The proof of Proposition \ref{prop:smooth-perron} presented below follows the proof of Harper's original result \cite[Proposition 1]{harper-BV} very closely.
It involves an application of Perron's formula to relate $\Psi(x,[y',y];\chi)$ to its Dirichlet $L$-function, defined as
\[
L(s,\chi;[y',y])
=\sum_{n\in S([y',y])}\frac{\chi(n)}{n^s}
=\prod_{y' \leq p \leq y} (1-\chi(p)p^{-s})^{-1},
\qquad (\Re s>0),
\] 
which is followed by a contour shift in the resulting contour integral.
The following lemma will be used in order to estimate the integrals on the new contour.
\begin{lemma} \label{lem:contour-shift}
 Suppose that $\alpha=\alpha(x,y)$ is the saddle point associated to $S(x,y)$. 
 Then under the assumptions of the proposition, we have
\[
\bigabs{\log L(\sigma+it,\chi;[y',y])-\log L(\alpha+it,\chi;[y',y])}\leq\frac{(\alpha-\sigma)\log x}{2}
\]
for all $\alpha-0.8\eps\leq\sigma\leq\alpha$ and $|t|\leq H/2$.
\end{lemma}
Before proving this lemma, we complete the proof of the proposition. 
\begin{proof}[Proof of Proposition \ref{prop:smooth-perron} assuming Lemma \ref{lem:contour-shift}]
The truncated version of Perron's formula as given in \cite[Theorems 5.2 and 5.3]{MV-book} yields
\begin{align} \label{eq:truncated-Perron}
\Psi(x,[y',y];\chi)&=\frac{1}{2\pi i}\int_{\alpha-\frac{iH}{2}}^{\alpha+\frac{iH}{2}}L(s,\chi;[y',y])\frac{x^s}{s}\mathrm ds\\
\nonumber
&\quad+O\Bigg(\frac{x^\alpha}{H}\sum_{n\in S([y',y])}\frac{\chi_0(n)}{n^\alpha}+\twosum{x/2<n<2x}{n\in S([y',y])}\chi_0(n)\min\left\{1,\frac{x}{H|x-n|}\right\}\Bigg),
\end{align}
where $\alpha=\alpha(x,y)$. 
Since $x/(H|x-n|)\leq H^{-1/2}$ whenever $|x-n|>xH^{-1/2}$, we have
$$
\min\left\{1,\frac{x}{H|x-n|}\right\}
\leq 
\begin{cases}
 1 & \text{if } |x-n|\leq xH^{-1/2}; \cr
 H^{-1/2} & \text{if } |x-n|>xH^{-1/2}.
\end{cases}
$$
Thus, the second error term in \eqref{eq:truncated-Perron} is bounded above by
\begin{align*} 
&\twosum{n\in S([y',y])}{x/2<n<2x}H^{-1/2}\chi_0(n)+\twosum{|x-n|\leq xH^{-1/2}}{n\in S([y',y])}1 \\
\nonumber
&\ll\frac{x^\alpha}{H^{1/2}}L(\alpha,\chi_0;[y',y])+\bigabs{\Psi(x+xH^{-1/2},[y',y])-\Psi(x-xH^{-1/2},[y',y])}. 
\end{align*}
Since the first error term in \eqref{eq:truncated-Perron} is smaller than the first term in the preceding line, 
it suffices to estimate the two terms in the bound above in order to bound the error in \eqref{eq:truncated-Perron}.

Concerning the first term, Lemma \ref{lem:BrTen-2.1}, combined with formulas (\ref{eq:Hil-Ten-Thm1}) and (\ref{g-m}), shows that
\begin{align} \label{eq:L-bound}
\nonumber
x^{\alpha} L(\alpha,\chi_0;[y',y])
&\ll
x^{\alpha}\sum_{n\in S([y',y])}  n^{-\alpha}  
\ll x^{\alpha} \zeta(\alpha,y) \prod_{p < y'}(1-p^{-\alpha}) \\
&\ll  \Psi(x,[y',y])\sqrt{\log x\log y}.
\end{align}
Concerning the second term, applying Lemma \ref{lem:Psi_m-approx} if $y < \exp(\log^{1/4}x)$ and Lemma \ref{lem:small-u} if $\exp(\log^{1/4}x) \leq y \leq x$ shows that
\begin{align*}
&\Psi(x(1+H^{-1/2}),[y',y])-\Psi(x(1- H^{-1/2}),[y',y])\\
&\ll \Bigbrac{ (1+H^{-1/2})^\alpha - (1-H^{-1/2})^\alpha + H^{-1/2} + (\log x)^{-1/5}} \Psi(x,[y',y])\\
&\ll (H^{-1/2} + (\log x)^{-1/5})\Psi(x,[y',y]),
\end{align*}
provided that $H \leq y^{5/6}$ so that $z = H^{1/2} \leq y^{5/12}$ in the application of Lemma \ref{lem:small-u}, and where we used an integral to bound the difference.
This shows that
\begin{align} \label{eq:truncated-Perron-with-error}
\Psi(x,[y',y];\chi)
&=\frac{1}{2\pi i}\int_{\alpha-\frac{iH}{2}}^{\alpha+\frac{iH}{2}}L(s,\chi;[y',y])\frac{x^s}{s}\mathrm ds \\
\nonumber
&\quad +O\Bigbrac{\Psi(x,[y',y])\left\{(\log x)^{-1/5}+H^{-\frac{1}{2}}\sqrt{\log x\log y}\right\}}.
\end{align}
Since $L(s,\chi;[y',y])$ is defined and has no singularities in $\Re s > 0$, shifting the line of integration to 
$\Re s=\alpha-0.8\eps$, shows that the integral above is equal to
\[
\int_{-\frac{H}{2}}^{\frac{H}{2}}L(\alpha-0.8\eps+it,\chi;[y',y])\frac{x^{\alpha-0.8\eps+it}}{\alpha-0.8\eps+it}\mathrm dt\pm\int_{\alpha-0.8\eps}^\alpha L(\sigma\pm\frac{iH}{2},\chi;[y',y])\frac{x^{\sigma\pm\frac{iH}{2}}}{\sigma\pm\frac{iH}{2}}\mathrm d\sigma,
\]
where the final expression indicates the sum over the two horizontal pieces.
Multiplying each of their integrands by a trivial factor of the form 
$L(\alpha\pm\frac{iH}{2},\chi;[y',y])^{+1-1}$ and applying Lemma \ref{lem:contour-shift}, each of the horizontal integrals is seen to be bounded by
\[
\frac{L(\alpha,\chi_0;[y',y])}{H}\int_{\alpha-0.8\eps}^\alpha x^{\frac{\alpha-\sigma}{2}}x^\sigma \mathrm d\sigma\ll\frac{x^\alpha L(\alpha,\chi_0;[y',y])}{H}.
\]
A similar argument shows that the vertical integral is bounded by
\begin{align*}
L(\alpha,\chi_0;[y',y])x^{\alpha-0.8\eps + 0.4 \eps}\int_{-H/2}^{H/2}\frac{\mathrm d t}{\alpha+|t|}
&\ll x^\alpha L(\alpha,\chi_0;[y',y])x^{-0.4\eps}\Big(\frac{1}{\alpha}+\log H\Big) \\
&\ll\sqrt{\log x\log y}\Psi(x,[y',y])x^{-0.4\eps}\log H,
\end{align*}
where we applied \eqref{eq:L-bound}. 
Putting everything together,
\[
\Psi(x,\chi;[y',y])\ll
\Psi(x,[y',y])\bigg( (\log x)^{-1/5} + \sqrt{\log x\log y}\Big(H^{-1/2}+x^{-0.4\eps}\log H \Big)\bigg),
\]
as required. It remains to prove the lemma.
\end{proof} 
\begin{proof}[Proof of Lemma \ref{lem:contour-shift}]
Reinterpreting the given difference as an integral shows that it is bounded above by
\[
(\alpha-\sigma)\sup_{\sigma\leq\sigma'\leq\alpha}\Bigabs{\frac{L'(\sigma'+it,\chi;[y',y])}{L(\sigma'+it,\chi;[y',y])}}=(\alpha-\sigma)\sup_{\sigma\leq\sigma'\leq\alpha}\Bigabs{\sum_{n\in S([y',y])}\frac{\Lambda(n)\chi(n)}{n^{\sigma'+it}}}.
\]
We may replace the summation condition in the final sum by $n \in [y',y]$ when bounding the contribution 
from proper prime powers $p^k \in [y',y]$ with $p < y'$ and $p^k>y$ with $p\leq y$ separately.
Bounding this contribution from proper prime powers trivially, the expression above is seen to be 
\begin{multline*}
\ll (\alpha-\sigma)\sup_{\sigma\leq\sigma'\leq\alpha}\Bigabs{\sum_{y'\leq p\leq y}\frac{\log p\,\chi(n)}{p^{\sigma'+it}}}+(\alpha-\sigma)\sum_{p\leq y}\frac{\log p}{p^{2(\alpha-0.8\eps)}}\\	
\ll (\alpha-\sigma)\biggbrac{\sup_{\alpha-0.8\eps\leq\sigma'\leq\alpha}\Bigabs{\sum_{ n\leq y}\frac{\Lambda(n)\chi(n)}{n^{\sigma'+it}}}+\sup_{\alpha-0.8\eps\leq\sigma'\leq\alpha}\Bigabs{\sum_{ n\leq y'}\frac{\Lambda(n)\chi(n)}{n^{\sigma'+it}}}+\frac{y^{1-\alpha-0.1\eps}}{1-\alpha}+\frac{1}{\eps}},
\end{multline*}
as $2(\alpha-0.8\eps)\geq\alpha+0.4\eps$ if $\eps\leq \alpha/2$. 
In the case that condition (i) of Proposition \ref{prop:smooth-perron} holds, that is $y\geq(Hr)^C$, it follows from Harper \cite[Lemma 1]{harper-BV} that 
\[
\sup_{\alpha-0.8\eps\leq\sigma'\leq\alpha}\Bigabs{ \sum_{ n\leq y} \frac{ \Lambda(n)\chi(n) }{ n^{\sigma'+it} } }\ll  \frac{ y^{1-\alpha-0.1\eps }}{ 1-\alpha } +\log(rH)+\log^{0.9}q+\frac{1}{\eps}
\]
since $\log^2(qyH)/H \ll y^{-0.9\eps}$ whenever $y^{0.9\eps}\log^2x \leq H$. 
Otherwise, condition (ii) of Proposition \ref{prop:smooth-perron} holds by assumption, and we have
$\log^2(qyH) < y^{2\eps/40} = y^{0.05 \eps}$.
Thus it follows from \cite[Lemma 2]{harper-BV} that the previous estimate holds in this case as well. 

Concerning the sum over $n \leq y'$, we shall show that for all sufficiently large $x$,
$$
\sup_{\alpha-0.8\eps\leq\sigma'\leq\alpha}\Bigabs{ \sum_{ n\leq y'} \frac{ \Lambda(n)\chi(n) }{ n^{\sigma'+it}}}
< \frac{\log x}{4}.
$$
To see this, note that if $1-1/\log y' < \sigma' \leq 1$, then
$$
\Bigabs{ \sum_{ n\leq y'} \frac{ \Lambda(n)\chi(n) }{ n^{\sigma'+it}}} \leq
\sum_{ n\leq y'} \frac{ \Lambda(n)}{n^{\sigma'}} \ll \log y' \ll K' \log_2 x = o(\log x).
$$
If $\alpha-0.8\eps \leq \sigma' < 1-1/\log y'$, then partial summation and the prime number theorem imply that
$$
\Bigabs{ \sum_{ n\leq y'} \frac{ \Lambda(n)\chi(n) }{ n^{\sigma'+it}}}
\ll  \frac{{y'}^{1-\sigma'}}{ 1-\sigma'} 
\ll  {y'}^{1-\sigma'}\log y'
\ll  {y'}^{1-\alpha} y'^{0.8\eps} \log y'
\ll  {y}^{(1-\alpha)/2} (\log x)^{0.4} \log y',
$$
where we also used that $\eps \leq 1/(2K')$ and $y' < y^{1/2}$.
On recalling the estimate \eqref{alpha-x-y}, that is
$1-\alpha = (\log (u \log (u+1))+O(1))/\log y$, and noting that $\log (u+1) \ll \log_2 x$, the above is seen to be bounded by
\begin{align*}
\ll  (u \log (u+1))^{1/2} (\log x)^{0.4} \log y'
\ll  \frac{ K' (\log x)^{9/10} (\log_2 x)^{3/2}}{(\log y)^{1/2}} = o(\log x),
\end{align*}
as required.

Collecting everything together, we conclude that
\begin{align*}
 \bigabs{\log L(\sigma+it,\chi;[y',y])&-\log L(\alpha+it,\chi;[y',y])}\\
&\ll(\alpha-\sigma)\left(
\Big(y^{-0.1\eps}+ \frac{1}{4}\Big)\log x +\log(rH)+\log^{0.9}q+\frac{1}{\eps}\right).
\end{align*}
Since $\eps> C/ \log y$, $q< x$ and $r,H\leq x^{\rho}$,  we obtain
\begin{align*}
\bigabs{\log L(\sigma+it,\chi;[y',y])&-\log L(\alpha+it,\chi;[y',y])} \\
&\ll(\alpha-\sigma)\left(\bigbrac{e^{-0.1C} +4^{-1}+2\rho+C^{-1}}\log x+\log^{0.9} x\right).
\end{align*}
The claimed bound thus follows provided $C>1$ is sufficiently large, $0<\rho<1$ sufficiently small and $x$ is sufficiently large.
\end{proof}

\begin{proof}[Proof of Theorem \ref{thm:equid-in-APs}]
Since $(a,q)= 1$, character orthogonality implies that
\begin{align*}
\Psi(x,[y',y];q,a) &= \frac{1}{\phi(q)} \sum_{\chi\Mod q} \bar\chi(a) \sum_{n\in S(x,[y',y])} \chi(n)\\
&=\frac{1}{\phi(q)}\Psi(x,[y',y];\chi_0)+\frac{1}{\phi(q)} \twosum{\chi\Mod q}{\chi\neq\chi_0} \bar\chi(a)\Psi(x,[y',y];\chi).
\end{align*}
If $p\mid q \Rightarrow p<y'$, then 
 $$
 \Psi(x,[y',y],\chi_0) 
 = \sum_{\substack{n\leq x,\, (n,q)=1}} 1_{S([y',y])}(n)
 =  \Psi(x,[y',y]).
 $$ 
To complete the proof, it thus suffices to show that
\[
\Psi(x,[y',y];\chi)\ll \Psi(x,[y',y]) \log^{-1/5} x
\]
uniformly for all non-principal characters $\chi\Mod q$.

The classical zero-free region (cf.\ \cite[Theorem 11.3]{MV-book}) implies that there is an absolute constant 
$0<\kappa \leq 1$ such that $\prod_{\chi\neq\chi_0} L(\sigma+it,\chi)$ has at most one, necessarily simple, zero in the region 
$$\{s: \sigma > 1-\kappa /\log (qH), |t|< H\}.$$
If such an exceptional zero exists, let $\chi_{\mathrm{Siegel}}$ denote the corresponding character. 
We consider the cases $\chi \neq \chi_{\mathrm{Siegel}}$ and $\chi=\chi_{\mathrm{Siegel}}$ separately.

Suppose first that $\chi \neq \chi_{\mathrm{Siegel}}$ and let $r\leq q \leq (\log x)^{K''}$ denote the conductor of 
$\chi$. 
We claim that the conditions of Proposition \ref{prop:smooth-perron} are satisfied if we choose
$$H = \min\set{ y^{\kappa/(2C)}, \exp(\log^{2/3} x), y^{5/6}, x^{\rho}} \quad \text{ and } \quad
\eps= \kappa/\log(qH).$$
Suppose that $x$ is sufficiently large to ensure that $\exp(\log^{2/3} x) < x^{\rho}$ and 
suppose that $C>1 \geq \kappa$.
If $K>2CK''/\kappa$ (or, in other words, if $K/K''$ is sufficiently large), then
$$q \leq (\log x)^{K''}<(\log x)^{\kappa K/(2C)}<y^{\kappa/(2C)}.$$
We thus have $qH<y^{\kappa/C}$ and, hence, $\eps= \kappa/\log(qH) > C / \log y$.
The upper bound $\eps< \min(\alpha/ 2, 1/(2K''))$ holds as soon as $x$, and hence $H$, is sufficiently large.
Concerning the conditions on $H$, suppose first that $H=\exp(\log^{2/3} x)$.
In this case, $\eps \asymp \log^{-2/3} x$ and $y^{0.9 \eps} \ll \exp(\log^{1/3}x)$. 
Hence, $H>y^{0.9\eps} \log^2 x$ holds as soon as $x$ is sufficiently large.
Next, suppose that $H \neq \exp(\log^{2/3} x)$. 
In this case, $\min(y^{\kappa/(2C)}, y^{5/6}) \leq \exp(\log^{2/3} x)$ and 
$\eps \asymp 1/\log y \gg 1/\log^{2/3} x$.
From $\eps \asymp 1/\log y$ we obtain $y^{0.9 \eps} \ll 1$ and, thus, $H>y^{0.9\eps} \log^2 x$ holds provided that the exponent $K$ in $y > (\log x)^K$ satisfies $\min(5K/6,K \kappa/(2C))>2$ and provided that $x$ is sufficiently large. 
With the below application of Proposition \ref{prop:smooth-perron} in mind, observe that the lower bound
$\eps \gg 1/\log^{2/3} x$, which holds in either of the above two cases, implies that
$$
x^{-0.4 \eps} \ll \exp (- c\log^{1/3} x), \quad (c>0).
$$
Finally, observe that condition (i) of Proposition \ref{prop:smooth-perron} is satisfied since
$rH \leq qH \leq y^{\kappa/C} \leq y^{1/C}$ in view of $\kappa \leq 1$.

This shows that Proposition \ref{prop:smooth-perron} applies to all $\chi \Mod q$, $\chi \neq \chi_{\mathrm{Siegel}}$ 
and yields
\begin{align*}
\Psi(x,[y',y];\chi) 
&\ll \Psi(x,[y',y]) \bigg( (\log x)^{-1/5} + \sqrt{\log x\log y}\Big(H^{-1/2}+x^{-0.4\eps}\log H \Big)\bigg) \\
&\ll \Psi(x,[y',y]) (\log x)^{-1/5}.
\end{align*}
To see that the latter bound holds, we shall show that the second and third term in the bound above make a negligible contribution.
Indeed, the third term is bounded as follows:
\begin{align*}
\sqrt{\log x\log y}~ x^{-0.4\eps}\log H 
&\ll \sqrt{\log x\log y} ~\exp(-c\log^{1/3} x)(\log y + \log^{2/3} x) \\
&\ll \exp(-c'\log^{1/3} x),
\end{align*}
where $c,c'> 0$ are positive constants. 
Further, if $H = \exp(\log^{2/3} x)$, then the second term is bounded by:
\begin{align*}
 \sqrt{\log x\log y} ~H^{-1/2} 
 \ll \exp(- (\log x)^{2/3}/2 + \log_2 x),
\end{align*}
which is negligible.
If $H \neq \exp(\log^{2/3} x)$, then $\log y \ll \log^{2/3} x$ and $\eps \asymp 1/(\log y)$, and the second term
satisfies:
\begin{align*}
 \sqrt{\log x\log y} ~H^{-1/2} 
 \ll \sqrt{\log x\log y} ~\frac{y^{-9 \eps/20}}{\log x}
 \ll \left(\frac{\log y}{\log x}\right)^{1/2} \ll (\log x)^{-1/3},
\end{align*}
which is also negligible. This concludes the case of unexceptional characters.

We now consider the contribution of the potential exceptional character $\chi=\chi_{\mathrm{Siegel}}$. 
Following Harper \cite[\S3.3]{harper-BV}, we split the analysis into two cases according to the size of $y$.
Suppose first that $y'^2 \leq y \leq x^{1/ (\log \log x)^2}$.
Applying the truncated Perron formula \eqref{eq:truncated-Perron-with-error} with $H=y^{5/6}$,
we obtain
\begin{align} \label{eq:truncated-perron2}
\Psi(x,[y',y];\chi) 
&= \frac{1}{2\pi i} \int_{\alpha-iy^{5/6}/2}^{\alpha+iy^{5/6}/2}
L(s,\chi;[y',y]) x^s \frac{\mathrm{d}s}{s} \\ \nonumber
&\quad +O\Bigbrac{\Psi(x,[y',y])\left\{(\log x)^{-1/5}+y^{-\frac{5}{12}}\sqrt{\log x\log y}\right\}}.
\end{align}
As in \cite{harper-BV}, we proceed by bounding the integrand in absolute value from above.
The argument used in \cite[\S3.3]{harper-BV} is partly based on ideas from Soundararajan \cite{Sound}. 
In our case, only small modifications are required.
To start with, we have
\begin{align*}
&\Bigabs{ \frac{ L(\alpha+it,\chi;[y',y]) }{ L(\alpha,\chi_0;[y',y]) } }
=\prod_{y'\leq p\leq y} 
\left| \frac{1-\chi(p)p^{-\alpha-it}}{ 1-\chi_0(p)p^{-\alpha}} \right|^{-1} 
\leq \prod_{\substack{y'\leq p\leq y \\ p\nmid q}} \left| 1+\sum_{k\geq1} \frac{ 1-\Re(\chi(p)p^{-it}) }{ p^{k\alpha} } \right|^{-1}\\
&\leq \exp \biggset{-\twosum{y'\leq p\leq y}{p\nmid q} \frac{ 1-\Re(\chi(p)p^{-it}) }{ p^{\alpha} } }
\leq \exp \biggset{-\twosum{y^{1/2}\leq p\leq y}{p\nmid q} \frac{ 1-\Re(\chi(p)p^{-it}) }{ p^{\alpha} } },
\end{align*}
since $y'^2<y$.
The next step is to show that in the final expression the argument of the exponential function satisfies 
\[
\twosum{y^{1/2}\leq p\leq y}{p\nmid q} \frac{ 1-\Re\chi(p)p^{-it} }{ p^{\alpha} } 
\gg\frac{u}{\log^2(u+1)},
\]
provided that $|t|\leq y/2$ and $y \leq x^{1/(\log \log x)^2}$.
The proof of this bound proceeds by splitting the range of $t$ into $|t| \leq 1/(2 \log y)$ and 
$1/(2 \log y) \leq |t| \leq y/2$.
On the latter range, the proof of the second part of Soundararajan's \cite[Lemma 5.2]{Sound} is employed.
In that proof, all sums over primes can be restricted to the range $y^{1/2} \leq p \leq y$, $p\nmid q$.
On the former range, Harper's argument (cf.\ \cite[p.17]{harper-BV}) applies directly.
For sake of completeness, we summarise this argument here.
Since $p \leq y$ and $t \leq 1/(2 \log y)$, we have $|t \log p| \leq 1/2$ and, by analysing the cases 
$\chi(p) = 1$ and $\chi(p) = -1$ for the quadratic character $\chi = \chi_{\mathrm{Siegel}}$ separately,
one obtains
\begin{align*}
 \sum_{y^{1/2}\leq p\leq y, p\nmid q} \frac{ 1-\Re\chi(p)p^{-it} }{ p^{\alpha} } 
 &\gg \sum_{y^{1/2}\leq p\leq y, p\nmid q} \frac{ 1-\chi(p) }{ p^{\alpha} } \\
 &\geq \frac{1}{\log y} \left( \sum_{y^{1/2}\leq p\leq y, p\nmid q} \frac{\log p}{p^{\alpha}} 
 - \sum_{y^{1/2}\leq p\leq y, p\nmid q} \frac{\chi_{\mathrm{Siegel}}^*(p)\log p}{p^{\alpha}}  \right),
\end{align*}
where $\chi_{\mathrm{Siegel}}^*$ is the primitive character that induces $\chi_{\mathrm{Siegel}}$.
Harper's argument is based on a combination of partial summation and the asymptotic evaluation of 
$\sum_{n \leq x} \Lambda(n) \chi(n)$ together with the observation that a Siegel zero has a negative contribution in that asymptotic evaluation.
The error terms in the asymptotic evaluation in \cite[Theorem 11.16 and Exercise 11.3.1.2]{MV-book} are small when 
$q < (\log x)^{K''}$.
Thanks to the minus sign in the final expression above, the contribution from the Siegel zero to this expression is positive and can therefore be ignored when seeking a lower bound.
See \cite[\S3.3]{harper-BV} for the remaining details.

Returning to \eqref{eq:truncated-perron2}, the above bounds and an application of \eqref{eq:L-bound} yield
\begin{align*}
\Psi(x,[y',y];\chi)&\ll L(\alpha,\chi_0;[y',y]) e^{-c_3 u(\log u)^{-2}} x^\alpha \log y + 
\frac{\Psi(x,[y',y])}{\log^{1/5} x}\\
&\ll \Psi(x,[y',y])\bigg(\sqrt{\log x} (\log y)^{3/2} e^{-c_3 u(\log u)^{-2}} + \log^{-1/5} x \bigg).
\end{align*}
Since $u = (\log x)/\log y  \gg \log \log^2 x$ when $y\leq \exp\bigbrac{ (\log x)/\log \log^2 x}$, the first 
term in the bound is $O_A(\Psi(x,[y',y])(\log x)^{-A})$ and we thus have
\[
\Psi(x,[y',y];\chi) \ll (\log x)^{-1/5} \Psi(x,[y',y]) 
\]
if $\chi=\chi_{\mathrm{Siegel}}$ and $y \leq x^{ 1/(\log \log x)^2 }$.

It remains to analyse the range $x^{ 1/(\log \log x)^2 }\leq y\leq x$ of $y$. 
In this case too, we follow the overall strategy used in \cite{harper-BV} but need to make a number of technical changes arising from our slightly different situation.
Observe that
\[
\Psi(x,[y',y];\chi) = \sum_{n\in S(x,[y',y])} \chi(n) 
= \twosum{ n\in S(x,y) }{ (n,P(y'))=1}\chi(n)
\leq \sum_{\substack{d\leq x \\ d \mid P(y')}} \Bigg| \sum_{n \in S(x/d,y)} \chi(n) \Bigg|,
\]
where we applied M\"obius inversion in order to remove the co-primality condition.
Our next aim is to show that, at the expense of an acceptable error term, the sum over $d$ can be further truncated 
in such a way that $\log (x/d) \asymp \log x$ uniformly for all values of $d$ that remain in the sum.
Recall that $\alpha(x,y)>1-1/K + o(1)$ and that $\Psi(x,y') \leq x^{1-1/K' + o(1)}$. 
Let $K^*>K$ be sufficiently large so that $\alpha(x,y)(1-1/K^*)>1-1/K'$.
Then the $m=1$ case of \cite[Theorem 2.4(i)]{Breteche-Tenen-2005}, which applies uniformly for $1 \leq d \leq x$, 
implies that
\begin{align} \label{eq:harper-truncation}
\nonumber
\sum_{\substack{x^{1-1/K^*} \leq d \leq x\\ d \mid P(y')}} \Psi(x/d,y)
&\ll \Psi(x,y) x^{-(1-1/K^*)\alpha} \sum_{x^{1-1/K^*} \leq d \leq x} \1_{S(y')}(d) \\
&\ll \Psi(x,y) x^{-(1-1/K^*)\alpha} x^{1-1/K' + o(1)}
 \ll \Psi(x,y) x^{-c_4}
\end{align}
for some positive constant $c_4>0$.
Assuming for the moment that this bound is sufficiently good, it remains to analyse the expression
$$
\sum_{\substack{d\leq x^{1-1/K^*} \\ d \mid P(y')}} 
\Bigg| \sum_{\substack{n \in S(x/d,y)}} \chi(n) \Bigg|,
$$
where $\log (x/d) \asymp \log x$. 
Fouvry and Tenenbaum study bounds on character sums such as the inner sum above in 
\cite[Lemme 2.1]{fouvry-tenenbaum-1996} (see also their Lemme 2.2 and its deduction) in a larger range for $y$ than the one under current consideration.
In view of the truncation of the sum over $d$ and the bound $q \leq (\log x)^{K''}$ on the modulus of 
$\chi$, \cite[Lemma 2.1 (i)]{fouvry-tenenbaum-1996} applies uniformly to the above character sum for each $d$.
By invoking \cite[Theorem 2.4(i)]{Breteche-Tenen-2005} in the case $m=1$, we obtain
\begin{align} \label{eq:harper-aux}
\nonumber
\sum_{\substack{d\leq x^{1-1/K^*} \\ d \mid P(y')}} 
\Bigg| \sum_{\substack{n \in S(x/d,y)}} \chi(n) \Bigg|
&\ll \sum_{\substack{d\leq x^{1-1/K^*} \\ d \mid P(y')}} d^{-\alpha} \Psi(x,y) \exp \bigbrac{-c\sqrt{\log y}} \\
& \ll  \exp \bigbrac{-c\sqrt{\log y}} \Psi(x,y) \prod_{p\leq y'} (1-p^{-\alpha})^{-1} .
\end{align}

It remains to express the bounds \eqref{eq:harper-aux} and \eqref{eq:harper-truncation} in terms of 
$\Psi(x,[y',y])$.
By Lemma \ref{lem:BrTen-2.1} and the first part of Lemma \ref{lem:MV}, we have 
$$\Psi(x,y)\ll \Psi(x,[y',y]) \prod_{p\leq y'} (1-p^{-\alpha})^{-1} .$$
Since $\log y' \asymp \log_2 x$ and $1- \alpha \ll (\log_2 x)^{2+\eps}/\log x$ 
for $x^{ 1/(\log \log x)^2 } \leq y \leq x$, it follows from the first part of Lemma \ref{lem:MV} (or direct computation) that
$\prod_{p\leq y'} (1-p^{-\alpha})^{-1} \ll \log y'$.
Hence, it follows from \eqref{eq:harper-aux} and \eqref{eq:harper-truncation} that
$$
\Psi(x,[y',y];\chi) \ll_A (\log x)^{-A} \Psi(x,[y',y]).
$$
if $\chi=\chi_{\mathrm{Siegel}}$ and $x^{ 1/(\log \log x)^2 } \leq y \leq x$.
Taking $A = 1/5$ completes the proof.
\end{proof}

\subsection{$W$-trick and equidistribution of weighted smooth numbers in short APs}
\label{sec:W_trick_and_equid}
The choice of our weight factor $n^{1-\alpha}$ ensures that the weighted version $h_{[y',y]}(n)$ of $\1_{S([y',y])}$ is equidistributed in short intervals and, by invoking Theorem \ref{thm:equid-in-APs}, also in short progressions 
$n \equiv a \Mod{q}$, provided the residue class $a$ is coprime to the modulus $q$ of the progression and $q < y'$.
Establishing non-correlation with nilsequences later requires equidistribution in almost all residue classes
$a \Mod{q}$.
Given any $Q \in \NN$ and $1 \leq A < Q$, $\gcd(Q,A)=1$, we consider the following renormalised restrictions  
\begin{equation} \label{def:h_y--W,A}
h_{[y',y]}^{(Q,A)}(m) = \frac{\phi(Q)}{Q}h_{[y',y]}(Q m + A),\qquad (m  \sim (N-A)/Q),
\end{equation}
and 
\begin{equation} \label{def:g_y--W,A}
g_{[y',y]}^{(Q,A)}(m) = \frac{\phi(Q)}{Q}g_{[y',y]}(Qm + A), \qquad (m \in \NN).
\end{equation}
Let $w(N) = \frac{1}{2}\log_3 N$, let $W=W(N) = \prod_{p\leq w(N)}p$ and observe for later use that 
$w(N) = (\log_3 N')/2 + o(1)$ for any $N' \in (N/\log N, N]$. 
We will mainly be interested in the case where 
\begin{itemize}
 \item $Q = W(N)$, or
 \item $Q = \W$ for some $\W \leq (\log N)^{K''}$ such that $W(N)|\W$ and $p|\W \Rightarrow p\leq w(N)$
\end{itemize}
in \eqref{def:h_y--W,A} and \eqref{def:g_y--W,A}.

The following lemma is an immediate consequence of Theorem \ref{thm:equid-in-APs} and Lemma \ref{lem:short-int}.

\begin{lemma}[Equidistribution in short progression]\label{lem:equid-short-progressions}
Let $K', K'' >0$ and $K > \max(2K',2)$ be constants and let $N \geq 2$ be a parameter. 
Suppose that $y' \leq (\log N)^{K'}$, and that $(\log N)^{K} < y \leq N$. 
Then, provided that $K$ and $K/K''$ are sufficiently large, the following estimates hold uniformly for all $N_0$, $N_1$ such that $$N < N_0 < N_0 + N_1 \leq 2N$$ and 
$N_1 \gg N \exp(-(\log^{1/4} N)/4)$, for all $q \leq (\log N)^{K''}$ such that $p\mid q \Rightarrow p<y'$ 
and all $a \Mod{q}$ such that $\gcd(a,q)=1$.
Then
\begin{align*}
\sum_{\substack{N_0<m\leq N_0+N_1 \\ m\equiv a \Mod{q}}} h_{[y',y]}(m) -  \frac{N_1}{\phi(q)}
&\ll \frac{N_1}{\phi(q)} \frac{\log_3N}{\log_2 N} + \frac{N}{\phi(q) \log^{1/24}N} + \frac{N}{\log^{1/5}N}.
\end{align*}
\end{lemma}

\begin{proof}
Recall the definition \eqref{def:h_y} of $h_{[y',y]}(n)$. 
On omitting the weight factor $N^\alpha/ (\alpha \Psi(N,[y',y]))$ for the moment, partial summation yields
\begin{align*}
\sum_{N_0<n\leq N_0+N_1} &n^{1-\alpha} \1_{n\equiv a\Mod{q}}1_{S([y',y])}(n)\\
&=\bigbrac{N_0+N_1}^{1-\alpha}\Psi(N_0+N_1,[y',y];q,a)
-(N_0)^{1-\alpha}\Psi(N_0,[y',y];q,a)\\
&\quad -(1-\alpha)\int_{N_0}^{N_0+N_1}\frac{\Psi(t,[y',y];q,a)}{t^\alpha}\mathrm dt.	
\end{align*}
Theorem \ref{thm:equid-in-APs} implies that
\begin{equation} \label{eq:aux2}
\Psi(t,[y',y];q,a)=\frac{\Psi(t,[y',y])}{\phi(q)}\Big(1+O(\log^{-1/5} N)\Big)
\end{equation}
whenever $N\leq t\leq 2N$ and provided that $K$ is sufficiently large.
Inserting this expansion into the previous expression and recalling the partial summation application
\eqref{eq:partialsum}, we see that the contribution from the main term in \eqref{eq:aux2} can be reinterpreted
as a sum over $n^{1-\alpha} \1_{S([y',y])}$ that is not restricted to a progression.
More precisely,
\begin{align} \label{eq:unweighted}
 \nonumber
&\phi(q)\sum_{N_0<n\leq N_0+N_1}n^{1-\alpha} \1_{n\equiv a\Mod{q}} \1_{S([y',y])}(n)\\
&= \sum_{N_0<n\leq N_0+N_1} n^{1-\alpha} \1_{S([y',y])}(n)
+ O \left(\frac{ N^{1-\alpha} \Psi(N,[y',y])}{\log^{1/5}N}\right).
\end{align}
The shape of the error term above follows by several applications of Lemmas \ref{lem:BrTen-2.4.i-ii} (i) and \ref{lem:BrTen-2.1}.
On reinserting the weights into \eqref{eq:unweighted} and invoking Lemma \ref{lem:short-int}, we obtain
\begin{align*}
&\sum_{\substack{ N_0\leq m\leq N_0+N_1 \\ m\equiv a \Mod q}} h_{[y',y]}(m) 
= \frac{1}{\phi(q)} \sum_{N_0\leq n\leq N_0+N_1} h_{[y',y]} (n) + O \left(\frac{N}{\log^{1/5}N} \right)\\
&= \frac{N_1}{\phi(q)} \left(1+ O\Bigbrac{ \frac{\log_3 N}{\log_2N}}\right) 
   + O\left(\frac{N}{\phi(q) (\log N)^{1/24}}\right)
   + O \left(\frac{N}{\log^{1/5}N} \right) 
\end{align*}
provided that $N_1 \gg N \exp(-(\log N)^{1/4}/4)$.
\end{proof}

By analysing the asymptotic order of the weight factor in $h_{[y',y]}$, we may deduce from the previous result an upper bound on the unweighted count of $[y',y]$-smooth numbers in short \emph{intervals}.
This bound will be integral in the application of an bootstrapping argument in Section \ref{sec:bootstrap}.

\begin{corollary}[{$[y',y]$}-smooth numbers in short intervals] \label{cor:W-tricked-equid-short-APs}
Let $K', K'' >0$ and $K > \max(2K',2)$ be constants and let $N \geq 2$ be a parameter.
Suppose that $1 \leq y' \leq (\log N)^{K'}$, and that $(\log N)^{K} < y \leq N$. 
Let $P \subseteq [N,2N]$ be any progression of length 
$|P| \geq N \exp(-(\log^{1/4} N)/4)$ and common difference $q \leq (\log N)^{K''}$ such that 
$p\mid q \Rightarrow p<y'$.

Then the following assertions hold provided that $K$ and $K/K''$ are sufficiently large.

 \begin{enumerate}
  \item Suppose that $|P| \geq L:=N/(\log N)^{\ell}$ for some constant $\ell>0$. Then
 $$
 |S(2N,[y',y]) \cap P| \ll \Delta \frac{\Psi(2N,[y',y]) \cdot |P|}{N}~, 
 $$
 where $\Delta = \log_3 N + (\log N)^{\ell - 1/24}$.
 \item  If, instead, $|P| \geq L:=N \exp(-C \varpi(N))$ and $q \leq N/L \leq \exp(C \varpi(N)) = o(\log N)$ for some 
 $\varpi(N) = o(\log_2 N)$, then 
 $$
 |S(2N,[y',y]) \cap P| \ll \Delta \frac{\Psi(2N,[y',y]) \cdot |P|}{N}~, 
 $$
 where $\Delta = \log \varpi(N)$.
 \end{enumerate}
\end{corollary}

\begin{proof}
Note that if $n \in [N,2N]$ is an element of $S([y',y])$, 
then $h_{[y',y]}(n) \neq 0$ and $N^{1-\alpha} \leq n^{1-\alpha} \leq (2N)^{1-\alpha}$, so that
\begin{equation} \label{eq:constant-order}
h_{[y',y]}(n) = \frac{N^{\alpha}}{\Psi(N,[y',y])} \frac{n^{1-\alpha}}{\alpha} 
\asymp \frac{N}{\Psi(N,[y',y])}, \qquad (n \in [N,2N] \cap S([y',y])).
\end{equation}
We shall show that in either of the two cases in the statement
$$
\frac{q}{\phi(q)} \ll \Delta,
$$
which then leaves us with the task of deducing from Lemma \ref{lem:equid-short-progressions} that 
$$
\sum_{n \in P} h_{[y',y]}(n) \ll \Delta |P|.
$$

In order to bound $q/\phi(q)$, consider the decomposition $q=q_1q_2$, where $q_1$ is composed of primes $p\leq \varpi$ and $q_2$ is composed of primes $p>\varpi$.
Then 
$$
\prod_{p|q_1}(1-p^{-1})^{-1} \ll \log \varpi.
$$
Concerning the product over prime divisors of $q_2$, we obtain 
\begin{equation} \label{eq:phi-and-W-trick}
\prod_{p \mid q, p>\varpi} (1 - p^{-1})^{-1}
\ll \exp\Big(\sum_{p \mid q, p>\varpi} p^{-1} \Big)
\ll 1 + O\Big(\frac{\log q}{\varpi \log \varpi}\Big),
\end{equation}
where we used the bound $\omega(q_2) \leq (\log q)/ \log \varpi$.
Thus,
$$
\frac{q}{\phi(q)} \ll \log \varpi + \frac{\log q}{\varpi} \ll \log_2 q,
$$
where we chose $\varpi \log \varpi = \log q$.
In case (2), we use this bound directly and obtain $\frac{q}{\phi(q)} \ll \log \varpi(N) = \Delta$. 
In case (1), we have $q / \phi(q) \ll \log_2 q \ll \log_3 N \ll \Delta$.

Turning towards the application of Lemma \ref{lem:equid-short-progressions}, note that for $|P| = N_1 / q$ the bound in that lemma satisfies
\begin{align*}
 \frac{N_1}{\phi(q)} \frac{\log_3N}{\log_2 N} + \frac{N}{\phi(q) \log^{1/24}N} + \frac{N}{\log^{1/5}N}
 \ll |P| \Big( \frac{q}{\phi(q)} \frac{\log_3N}{\log_2 N} + (N/|P|)(\log N)^{- 1/24}\Big).
\end{align*}
If $|P|>N (\log N)^{-\ell}$, we obtain 
\begin{align*}
 \frac{N_1}{\phi(q)} \frac{\log_3N}{\log_2 N} + \frac{N}{\phi(q) \log^{1/24}N} + \frac{N}{\log^{1/5}N}
 \ll |P| \Big( \frac{q}{\phi(q)} \frac{\log_3N}{\log_2 N} + (\log N)^{\ell - 1/24}\Big)
 \ll |P| \Delta,
\end{align*}
provided $\Delta \gg (\log N)^{\ell - 1/24}$, which holds under the assumptions of (1).

If, instead, $\log (N/|P|) \ll \varpi(N) = o(\log_2 N)$, then
\begin{align*}
 \frac{N_1}{\phi(q)} \frac{\log_3N}{\log_2 N} + \frac{N}{\phi(q) \log^{1/24}N} + \frac{N}{\log^{1/5}N}
 \ll |P| \Big( \frac{q}{\phi(q)} \frac{\log_3N}{\log_2 N} + o(1) \Big) = o(|P|),
\end{align*}
which concludes the proof.
\end{proof}

In addition to the above result, we shall need an analogous bound that is valid on much shorter intervals but may in return have a larger value of $\Delta$.
For $y$-smooth number, such a result appears as `Smooth Numbers Result 3' in Harper \cite{harper-minor-arcs}, see also \cite[Lemma 3.3]{Drappeau-Shao} for a slight extension.
The lemma below will be used in Section \ref{sec:weyl-sums}, where we extend some of the results of Drappeau and Shao \cite{Drappeau-Shao}, as well as in Section \ref{sec:bootstrap}, where it will find application in those situations where the progressions are too short for the previous corollary to be used.

\begin{lemma} \label{lem:short-progressions-Harper}
Let $K'>0$, $K \geq \max(2K',2)$, $x>2$, $1 \leq y' \leq (\log x)^{K'}$ and $(\log x)^K < y\leq x$. 
Suppose that $P\subseteq[x,2x]$ is an arithmetic progression.
Then, provided that $x$ is sufficiently large, we have
\[
\#\set{n\in P:n\in S([y',y])}\ll \left(x/|P|\right)^{1-\alpha}\frac{\Psi(x,[y',y]) |P|}{x} \log x.	
\]
\end{lemma}

\begin{proof}
Let $X\geq x$ denote the smallest element of $P$. 
This lemma is a generalization of \cite[Smooth number result 3]{harper-minor-arcs}, which bounds
$$\sum_{\substack{X \leq n \leq X+Z \\ n \equiv a \Mod{q}}} \1_{S(y)}(n)$$
under the assumptions that $q \geq 1$ and $qy \leq Z \leq X$, i.e. concerns progressions of length $|P|\geq y$.
The proof of \cite[Smooth number result 3]{harper-minor-arcs} carries over directly to the situation where 
$S(y)$ is replaced by $S([y',y])$, with one small exception: 
The application of Bret\`eche and Tenenbaum \cite[Th\'eor\`eme 2.4 (i)]{Breteche-Tenen-2005} with $m=1$ needs to be replaced by an application of Lemma \ref{lem:BrTen-2.4.i-ii} (i), which is 
\cite[Th\'eor\`eme 2.4 (i)]{Breteche-Tenen-2005} with $m=P(y')$.
While in the former case, the bound holds uniformly for $1 \leq d \leq x$, it is only available for 
$1 \leq d \leq x/y$ in our case.

Replacing the assumption that $qy \leq Z$ by $qy^2 \leq 2Z$ ensures that $d < X/y$ in all applications of 
\cite[Th\'eor\`eme 2.4 (i)]{Breteche-Tenen-2005}.
More precisely, we have 
$$
 \frac{Xqy}{Z2^{j+1}} \leq \frac{X}{y}
$$
for all admissible values of $j$ in Harper's proof.
 
It remains to consider the case where $|P| < y^2/2$, which corresponds to the case $|P|<y$ addressed in (the proof of) 
\cite[Lemma 3.3]{Drappeau-Shao}.
By bounding the left hand side in our statement trivially by $|P|$, and observing that
$|P|^{-1 + \alpha} \gg y^{(-1+\alpha)/2}$, it suffices to show that
$$
y^{\alpha-1} \frac{\Psi(x,[y',y])}{x^\alpha} \log x \gg 1.
$$
Note that $y^{\alpha - 1} = (u \log (u+1) + O(1))^{-1}$ in view of \eqref{alpha-x-y} and observe that  
\[
y^{\alpha-1} \frac{\Psi(x,[y',y])}{x^\alpha} \log x 
\gg y^{\alpha-1}\frac{\prod_{y'\leq p\leq y}(1-p^{-\alpha})^{-1}}{\sqrt{\log x\log y}} \log x
\gg y^{\alpha-1} \prod_{y'\leq p\leq y}(1-p^{-\alpha})^{-1}.
\]
by Lemma \ref{lem:BrTen-2.1} and \eqref{eq:Hil-Ten-Thm1}.
We split the analysis of the product over primes into two cases according to the size of $y$.
When $\exp(\log x/\log_2x)\leq y\leq x$, noting that $\alpha\leq 1$, we have
\[
\prod_{y'\leq p\leq y}(1-p^{-\alpha})^{-1} \gg \frac{ \prod_{p\leq y} (1-p^{-1})^{-1}}{\prod_{p\leq y'} (1-p^{-1})^{-1}} \gg \frac{\log y}{1+\log y'} \gg \frac{\log x}{(\log_2x)^2}.
\]
Hence,
\begin{align*}
y^{\alpha-1} \prod_{y'\leq p\leq y}(1-p^{-\alpha})^{-1}
\gg
\frac{\log y}{\log x \log(u+1)} \frac{\log x}{(\log_2x)^2} 
\gg \frac{\log x}{(\log_2 x)^3\log_3 x}
\gg 1.
\end{align*}
When $y'^2\leq y\leq \exp(\log x/\log_2x)$, then $u \geq \log_2 x$ and 
$\alpha \leq  1 - (\log_3 x)/\log y + O(1/\log y)$.
Hence, the second conclusion of Lemma \ref{lem:MV} shows that
\begin{align*}
y^{-1 +\alpha}\prod_{y'\leq p\leq y}(1-p^{-\alpha})^{-1}
& \gg y^{-1+\alpha}\prod_{\sqrt{y}\leq p\leq y}(1-p^{-\alpha})^{-1} \\
& \gg y^{-1+\alpha} 
   \exp\left(\frac{y^{1-\alpha}(1+ o(1))- 2 y^{(1-\alpha)/2}(1+ o(1)) }{(1-\alpha)\log y} \right) \\
&\gg  
   \exp\left(\frac{y^{1-\alpha}(1+ o(1)) - (\log y^{1-\alpha})^2}{\log y^{1-\alpha}}  \right) \gg 1,
\end{align*}
where we used that $y^{1-\alpha} \asymp u \log u$, where $u \to \infty$ as $x \to \infty$, implies that
$y^{(1-\alpha)/2} = o(y^{1-\alpha})$.
\end{proof}

\section{Weyl sums for numbers without small and large prime factors}
\label{sec:weyl-sums}

As technical input in the proof of the orthogonality of $(h_{[y',y]}^{(W,A)}(n)-1)$ to nilsequences, we shall require bounds on Weyl sums where the summation variable is restricted to integers in $S([y',y])$.
In \cite[\S5]{Drappeau-Shao}, Drappeau and Shao establish bounds on Weyl sums for smooth numbers $S(x,y)$ by extending the methods from Harper's minor arc analysis in \cite{harper-minor-arcs}.
Our aim in this section is to obtain analogues bounds for Weyl sums over $S(x,[y',y])$, that is for Weyl sums over numbers free from small and large prime factors.

Following the notation in \cite{Drappeau-Shao}, define for any given parameters $Q,x>0$ and co-prime integers 
$0< a \leq q \leq Q$ the major arcs
$$\mathfrak{M}(q,a;Q,x) = \{\theta \in [0,1): |q\theta - a| \leq Qx^{-k}\}$$
and
\begin{equation} \label{def:majorarc}
\mathfrak{M}(Q,x) 
= \bigcup_{\substack{0\leq a<q \leq Q \\ (a,q)=1}} \mathfrak{M}(q,a;Q,x)
\end{equation}
Given any positive integer $k$, define the following smooth Weyl sum of degree $k$: 
\[
E_k(x,[y',y];\theta) := \sum_{n \in S(x,[y',y])} e(\theta n^k).
\]
The main objective of this section is to establish the theorem below.

\begin{theorem} [Weyl sums, $x^{\eta}$-smooth case]\label{thm:x-eps-smooth}
Let $k>0$ be a fixed positive integer and suppose that $\eta \in (0,1/(4k)]$.  
Let $K>2K'>2$ and suppose that $y' \leq (\log x)^{K'} < (\log x)^{K} < y \leq x^{\eta}$
and that $K$ is sufficiently large.
Let $\alpha = \alpha(x,y)$ denote the saddle point associated to $S(x,y)$, and
let $\theta \in \T:=\RR/\ZZ$ be a frequency. Then: 
\begin{enumerate}
 \item If $x$ is sufficiently large and $\theta\not\in\mathfrak M(x^{1/12},x)$, we have
$$
E_k(x,[y',y];\theta)  \ll x^{1-c}
$$
for some $c=c(k,\eps)>0$.
 \item If $x$ is sufficiently large, $\theta\in\mathfrak M(x^{1/12},x)$ and if $0<a<q\leq x^{0.1}$ are co-prime integers such that $|q\theta - a| \leq q^{-1}$, then
\[
E_k(x,[y',y];\theta)\ll\Q^{-c+2(1-\alpha)}(\log x)^5\Psi(x,[y',y])
\]
for some $c=c(k,\eps)>0$ and $\Q=q+x^k\norm{q\theta}$.
\end{enumerate}
\end{theorem}

We prove the two parts of this lemma in turn. 
The former case, in which $\theta$ is `highly irrational', will be handled by extending results of Wooley \cite{Wooley}.
For this purpose, we first establish an auxiliary lemma generalising \cite[Lemma 2.3]{Wooley}, which in turn builds on an argument of Vaughan \cite{Vaughan}.

\begin{lemma}\label{lem:auxiliary-lemma}
Suppose that $\theta\in\T$ is a frequency and $r\in\NN$. If $1 \leq y'< y\leq M<x$, then we have
\[
\twosum{n\in S(x,[y',y])}{(n,r)=1}e(n^k\theta)
\ll y (\log x) \max_{y'\leq p\leq y}\twosum{v\in\mathfrak B(M,p,y)}{(v,r)=1}
\sup_{\beta\in\T}\Bigabs{\twosum{u\in S(x/M,[y',p])}{(u,r)=1}e(u^kv^k\theta+u\beta)}+M,
\]
where 
$\mathfrak B(M,p,y)=\set{M<v\leq Mp:p\mid v,v\in S([p,y])}$.
\end{lemma}

\begin{proof}
We start by decomposing the elements $n \in S(x,[y',y])$ in a similar fashion as in Vaughan \cite[Lemma 10.1]{Vaughan}.
By considering the prime factorisation $n=p_1^{k_1}\cdots p_s^{k_s}$, where the prime factors are ordered in \emph{decreasing} order $y\geq p_1>\dots> p_s\geq y'$, it is immediate that, for every given $M \in [y,x)$ and every
$n \in S(x,[y',y])$ with $n>M$ there is a unique triple $(u,v,p)$ such that $n=uv$ and such that $v$ is the smallest initial factor in the factorisation above that exceeds $M$. More precisely,
\begin{enumerate}
\item $y'\leq p\leq y$ and $p\mid v$;
\item $u\in S(x/v,[y',p])$;
\item $M<v\leq Mp$ and $v\in S([p,y])$.
\end{enumerate}
Using this factorization, the exponential sum can be decomposed as
\[
\twosum{n\in S(x,[y',y])}{(n,r)=1}e(n^k\theta)=\sum_{y'\leq p\leq y}\twosum{v\in\mathfrak B(M,p,y)}{(v,r)=1}\twosum{u\in S(x/v,[y',p])}{(u,r)=1}e(u^kv^k\theta)+O(M),
\]
where the error term $O(M)$ bounds the contribution from all $n\in S(x,[y',y])$ with $n\leq M$.
As in Vaughan \cite[(10.9)]{Vaughan}, we can use the orthogonality principle to remove the dependence on $v$ in the restriction $u\in S(x/v,[y',p])$. Since $v>M$, the inner sum above is equal to
\begin{align*}
\int_\T\Bigbrac{\twosum{u\in S(x/M,[y',p])}{(u,r)=1}&e(u^kv^k\theta+u\beta)}\bigbrac{\sum_{m\leq x/v}e(-m\beta)}\,\mathrm d\beta\\
&\ll\sup_{\beta\in\T}\Bigabs{\twosum{u\in S(x/M,[y',p])}{(u,r)=1}e(u^kv^k\theta+u\beta)}\int_\T\min\set{x/v,\norm{\beta}^{-1}}\,\mathrm d\beta\\
&\ll\log x\sup_{\beta\in\T}\Bigabs{\twosum{u\in S(x/M,[y',p])}{(u,r)=1}e(u^kv^k\theta+u\beta)}.\qquad\qquad\qquad
\end{align*}
The lemma follows by combining this bound with the previous expression.
\end{proof}

The following lemma is a generalization of the Weyl sum estimate for $y$-smooth numbers given by Wooley \cite{Wooley}, and it relies on the observation that the conclusion of Wooley \cite[Lemma 3.1]{Wooley} 
continues to hold when we restrict the $y$-smooth numbers to $[y',y]$-smooth numbers.

\begin{lemma}[Weyl sum, $\theta$ is highly irrational]\label{lem:wooley}
Let $k\in\NN$ be a fixed positive integer and suppose that $\sigma \in (0, 1/2)$ and $\eta \in (0,\sigma/(4k)]$.
Then there exists $c = c(k, \sigma)>0$ such that the following statement is true. 
Let $1 \leq y'< y\leq x^\eta$ and $\theta \in (0,1)$. 
Then, if $\theta\not\in\mathfrak M(x^\sigma,x)$, we have
\[
E_k(x,[y',y];\theta)\ll x^{1-c}.
\]   
\end{lemma}

\begin{proof}
We seek to apply a modified version of Wooley \cite[Lemma 3.1]{Wooley} in a similar way as at the start of the proof of \cite[Theorem~4.2]{Wooley}. 
Wooley's lemma produces bounds on 
$$
f(\theta;x,y) := E_k(x,[1,y];\theta) 
$$
in terms of $M$ and $q$, where $M = x^{\lambda}>y$ is a parameter such that $\lambda \in (1/2,1)$ and $q$ is any natural number such that $\|q\theta \| = |q \theta -a |<q^{-1}$ for some $a \in \ZZ$ with $(a,q)=1$.
As we shall explain below, one can deduce from the proof of the lemma that the same bounds hold for 
$E_k(x,[y',y];\theta)$ in place of $f(\theta;x,y) = E_k(x,[1,y];\theta)$, i.e.\ for the exponential sum relevant to us.
In order for those bounds to be useful, additional assumptions on $M$ and $q$ are necessary, and our first aim is to show that, under the assumption of our lemma, we can choose $M$ and $q$ in such a way that  
\begin{equation} \label{eq:wooley-conditions}
M>y, \quad q \leq 2(yM)^k, \quad |q \theta -a |< (yM)^{-k}/2 \quad \text{ and } \quad q > (x/M)^k. 
\end{equation}
These conditions corresponds to those in place at the start of the proof of \cite[Theorem~4.2]{Wooley}.
When combined\footnote{See the proof of \cite[Theorem~4.2]{Wooley} for details.} with the upper bound produced by \cite[Lemma 3.1]{Wooley}  
it follows that
\begin{equation} \label{eq:wooley-conclusion}
E_k(x,[y',y];\theta) \ll x^{1 - c(k, \sigma)}, 
\end{equation}
once we proved that the lemma can be extended to the case of $[y',y]$-smooth numbers.

Concerning the conditions \eqref{eq:wooley-conditions}, let $0<\sigma^* < \sigma/2$. 
Then $\sigma^* + 2k\eta < \sigma < 1/2$.
By the Dirichlet approximation theorem, there is a positive integer $q \leq x^{k-\sigma^*}$ and some $a \in \ZZ$ with 
$(a,q)=1$ such that $$\|q\theta\| = |q\theta-a| \leq x^{\sigma^*-k} < x^{\sigma-k}.$$
Since $\theta\not\in\mathfrak M(x^\sigma,x)$, this implies that $q > x^{\sigma}$.
Let $M$ be defined by the equation $$x^{k-\sigma^*} = 2(yM)^k.$$
Then, since $\sigma^* + 2\eta < 1/2$,  we have
$$
y<x^{\eta} < x^{1-\sigma^*/k-\eta} 2^{-1/k} \leq M = 2^{-1/k} x^{1-\sigma^*/k} y^{-1} \leq x^{1-\sigma^*/k}.
$$
as soon as $x$ is sufficiently large. Hence, $y<M$.
If we write $M = x^{\lambda}$, then the above line of inequalities also shows that 
$\lambda \in (1/2,1)$ since $\sigma^*/k+\eta< 1/2$.
Observe further that $\sigma^* + k\eta < \sigma$ implies that
$$
q > x^{\sigma} > 2 x^{\sigma^* + k\eta} \geq (x/M)^k
$$
as soon as $x$ is sufficiently large. 
Hence, all conditions from \eqref{eq:wooley-conditions} are satisfied and \eqref{eq:wooley-conclusion} therefore follows under the assumptions of our lemma.

It remains to show that \cite[Lemma 3.1]{Wooley} can be extended to the case of exponential sums over $[y',y]$-smooth numbers.
The proof strategy is to follow the original argument, applied to the exponential sum $E_k(x,[y',y];\theta)$
instead of $E_k(x,[1,y];\theta)$, and explain how all dependencies on $y'$ that appear in any of the bounds on the way can be removed, so as to eventually arrive at an intermediate bound that agrees with the corresponding bound in the original proof. 
Following the original proof through to the end shows then that the upper bound that \cite[Lemma 3.1]{Wooley} provides for $f(\theta;x,y) = E_k(x,[1,y];\theta)$ is also an upper bound for
$E_k(x,[y',y];\theta)$.

Turning to the details, let $M = x^{\lambda}$ for any $\lambda\in (1/2,1)$ and let $q \in \NN$ be such that 
$|q\theta-a|<q^{-1}$ for some $a \in \ZZ$ with $(a,q)=1$.
The $r=1$ case of Lemma \ref{lem:auxiliary-lemma} (which replaces \cite[Lemma 2.3]{Wooley}) produces a prime $y'< p\leq y$ and some frequency $\gamma\in\T$ such that
\[
E_k(x,[y',y];\theta)\ll (\log x) y\sum_{v\in S(My,y)}\bigabs{\sum_{u\in S(x/M,[y',p])}e(u^kv^k\theta+u\gamma)}+M.
\]
Note carefully that the sum over $v$ only involves $y$-smooth numbers. 
Following the argumentation of the start of the proof of \cite[Lemma~3.1]{Wooley} 
(which involves expressing the absolute value of the inner sum as 
$\eps(v, \theta) \sum_{u\in S(x/M,[y',p])}e(u^kv^k\theta+u\gamma)$ for some $\eps(v, \theta) \in \CC$ 
of unit modulus, reinterpreting certain exponential sums, as well as two applications of H\"older's inequality), shows that for all positive integers $t,w\in\NN$,
\[
\biggbrac{\sum_{v\in S(My,y)}\bigabs{\sum_{u\in S(x/M,[y',p])}e(u^kv^k\theta+u\gamma)}}^{2tw} \ll (My)^{2w(t-1)} \Big(\sum_cn_c\Big)^{2w-2} \Big(\sum_cn_c^2\Big) J_w(\theta),
\]
where $J_w(\theta)$ is as in \cite[(3.4)]{Wooley} and where $n_c$ denotes the number of solutions to the equation $u_1^k+\cdots+u_t^k=c$ with $u_i \in S(x/M,[y',p])$. 
Since $p\leq y$, we have $S(x/M,[y',p]) \subset S(x/M,y)$, which implies that $\sum_cn_c^2$ is trivially bounded above by 
$\int_\T\bigabs{\sum_{u\in S(x/M,y)}e(u^k\theta)}^{2t}\,\mathrm d\theta$.
Using the trivial bound  $\sum_cn_c\ll(x/M)^t$ for the remaining sum over $c$ completes our task of removing all dependencies on $y'$ from the upper bound on the given exponential sum $E_k(x,[y',y];\theta)$.
To summaries, we obtain
\begin{multline*}
\biggbrac{\sum_{v\in S(My,y)}\bigabs{\sum_{u\in S(x/M,[y',p])}e(u^kv^k\theta+u\gamma)}}^{2tw} \\
 \ll (My)^{2w(t-1)} \bigbrac{x/M}^{t(2w-2)} J_w(\theta) \int_{\T} \bigabs{ \sum_{u\in S(x/M,y)} e(u^k\theta) }^{2t}\,\mathrm d\theta,
\end{multline*}
corresponding to the bound in \cite[(3.5)]{Wooley}. 
\end{proof}

The following final lemma of this section is an easy generalization of Drappeau and Shao \cite[Proposition~5.7]{Drappeau-Shao}, which itself generalises the $k=1$ case established by Harper \cite[Theorem 1]{harper-minor-arcs}.

\begin{lemma}[Weyl sum, $\theta$ is irrational] \label{lem:drappeau-shao}
Suppose that $\theta=a/q+\delta$ for some $q \in \NN$, $a \in \ZZ$ and $\delta \in \RR$ such that $(a,q)=1$ and 
$|\delta|\leq 1/(2q)$. 
Further, let $x > 2$, $1 \leq y' \leq (\log x)^{K'}$ for some $K'\geq 1$ and let $y'^2 \leq y \leq x$. 
Write $\Q=q+x^k\norm{q\theta}$, and assume that $4y^3\Q^3\leq x$.
Then there exists some constant $c=c(k)>0$ such that
\[
E_k(x,[y',y];\theta)\ll \Q^{-c+2(1-\alpha)}(\log x)^5\Psi(x,[y',y]).
\]
\end{lemma}

\begin{proof}
The proof of \cite[Proposition 5.7]{Drappeau-Shao} carries over almost directly when replacing any application of 
\cite[Lemma 3.2]{Drappeau-Shao} by one of Lemma \ref{lem:BrTen-2.4.i-ii} (i) combined with Lemma \ref{lem:BrTen-2.1}.
In doing so, we however have to ensure that the stronger condition that $1\leq d \leq x/y$ of 
Lemma~\ref{lem:BrTen-2.4.i-ii}~(i) is satisfied.
For this reason we make the stronger assumption that $4y^3\Q^3\leq x$ instead of $4y^2\Q^3\leq x$.
As only fairly straightforward changes are required, we explain below how to deduce our lemma from the proof of \cite[Proposition 5.7]{Drappeau-Shao}, instead of repeating the entire proof.

Extracting, as before, the greatest common divisor of $n\in S(x,[y',y])$ and $q^\infty$, we obtain
\[
E_k(x,[y',y];\theta)
=\sum_{\substack{d\mid q^\infty \\ d \leq x}}\1_{S([y',y])}(d)
 \twosum{n\leq x/d}{(n,q)=1}\1_{S([y',y])}(n)e(n^kd^k\theta).
\]
We claim that the contribution from those terms with $d>\Q^{1/2}$ and those terms with $n\leq x/\Q$ is negligible, in the sense that
\begin{equation} \label{eq:DS-aux1}
E_k(x,[y',y];\theta)
=\sum_{ \substack{d\mid q^\infty \\ d\leq\Q^{1/2}}} \1_{S([y',y])}(d)
 \sum_{ \substack{ x/\Q\leq n\leq x/d \\ n \in S([y',y]) \\ (n,q)=1 }} e(n^kd^k\theta)
+O_{\eps}\left(\frac{\Psi(x,[y',y])}{\Q^{\alpha/2-\eps}}\right).
\end{equation}
We start with the contribution from $\max(x/y, \Q^{1/2}) < d \leq x$.
In this case $x/d < y < x^{1/3}$. 
Note that $q \leq \Q \leq x^{1/3}$ has $\omega(q) \leq \log x$ prime factors.
Hence, the contribution to $E_k(x,[y',y];\theta)$ is bounded by
$$
\sum_{\substack{d\mid q^{\infty}\\ x/y < d \leq x}} y
\leq x^{1/3} \Psi(x, \log x) \leq x^{1/3 +o(1)} 
$$
by Erd{\H{o}}s's bound stated in \cite[Section 7.1.1, Exercise 12]{MV-book}.
Since $\alpha > 1/2$ and $\Q < x^{1/3}$, it follows that
$$
x^{1/3 +o(1)} \ll \Q^{-\alpha/2} \Psi(x,[y',y]).
$$

If $\Q^{1/2} < d < x/y$, then Lemma \ref{lem:BrTen-2.4.i-ii} (i) applies and we obtain
\begin{align*}
 &\sum_{\substack{d \mid q^\infty \\ \Q^{1/2} < d \leq x/y}} \Psi(x/d,[y',y]) \1_{S([y',y])}(d)
 \ll \Psi(x,[y',y]) \sum_{\substack{ \Q^{1/2} < d \leq x/y \\ d\mid q^{\infty}}} d^{-\alpha} \1_{S([y',y])}(d) \\
 &\qquad\ll \Psi(x,[y',y]) \Q^{-\alpha/2} \sum_{\substack{d\mid q^{\infty}}} d^{-\alpha/2} \1_{S([y',y])}(d) 
 \ll \Psi(x,[y',y]) \Q^{-\alpha/2} \exp\bigg(\sum_{\substack{p > y'\\ p\mid q}} p^{-\alpha/2} + O(1)\bigg) \\
 &\qquad\ll \Psi(x,[y',y]) \Q^{-\alpha/2} \exp\Big(O \Big( {y'}^{-\alpha/2} \frac{\log \Q}{\log y'}\Big)\Big) 
  \ll_{\eps} \Psi(x,[y',y]) \Q^{-\alpha/2 +\eps},
 \end{align*}
since $\alpha > 1/2$.

To remove the contribution from $n \leq x / \Q$ we may now assume that $d < \Q^{1/2}$.
This contribution is trivially bounded above by
$$
\sum_{d < \Q^{1/2}} \sum_{\substack{n \leq x/\Q \\ (n,d) = 1}} \1_{S([y',y])}(dn)
\leq \Psi(x/\Q^{1/2};[y',y]) \ll \Q^{\alpha/2} \Psi(x;[y',y]),
$$
where we used Lemma \ref{lem:BrTen-2.4.i-ii} (i) together with the bound $\Q^{1/2} < x/y$ which follows from the assumption that $4 y^3\Q^3 \leq x$.
This shows that \eqref{eq:DS-aux1} holds.

Following \cite{Drappeau-Shao}, set $L=4y\Q$ and decompose any element of $[x/\Q,x/d]\cap S([y',y])$ as the unique product $mn$ for which $m\in[L,yL]$ and $P^+(m)\leq P^-(n)$, which is possible since
$y<Ly < x/\Q$ by the assumption that $4y^3 \Q^3 \leq x$.
Decomposing the range of $m$ dyadically and extracting the largest prime factor of $m$ implies that 
\[
E_k(x,[y',y];\theta)\ll\log x\sup_{L\leq M\leq yL}\mathcal E(M)+\frac{\Psi(x,[y',y])}{\Q^{\alpha/4}},
\]
where 
\[
\mathcal E(M)
 = \threesum{d\mid q^\infty}{d\leq\Q^{1/2}}{d\in S([y',y])}
  \sum_{y' \leq p\leq y}
  \sum_{\substack{m \in S([y',p]) \\ M/p \leq m\leq\min\set{2M,yL}/p}}
  \sum_{\substack{n \in S([p,y]) \\ (pmn,q)=1 \\ x/(pm\Q)\leq n\leq x/(pmd) }}
  e\bigbrac{(pmnd)^k\theta}.
\]
The dependence of the summation condition of the inner sum on $m$ can be removed with the help of the same trick as in the proof of Lemma \ref{lem:auxiliary-lemma}, which leads to
$$
\mathcal E(M)
\ll (\log x) \sup_{\beta \in [0,1)}
  \sum_{\substack{d \mid q^\infty \\ d\leq\Q^{1/2} \\ d\in S([y',y])}}
  \sum_{y' \leq p\leq y}
  \sum_{\substack{m \in S([y',p]): \\ M/p \leq m \leq \\ \min\set{2M,yL}/p}}
  \Big|
  \sum_{\substack{n \in S([p,y]) \\ (pmn,q)=1 \\ x/(2M\Q)\leq n\leq x/(Md) }}
  e\bigbrac{(pmnd)^k\theta + \beta n} \Big|.
$$
An application of Cauchy-Schwarz, combined with the bound
$$\sum_{\substack{d \mid q^\infty \\ d \leq \Q^{1/2}\\d \in S([y',y])}} 1 
\leq \Q^{\eps/2} \sum_{\substack{d \mid q^\infty \\ d \leq \Q^{1/2}\\d \in S([y',y])}} d^{-\eps} 
\leq \Q^{\eps/2} \exp(O(1)\sum_{p\mid q} p^{-\eps}) \ll \Q^{\eps/2} \Q^{o(1)} \ll_{\eps} \Q^{\eps},$$
which follows from $\omega (d) \leq (\log \Q)/ \log y'$, we obtain
$$
\mathcal E(M)
\ll_{\eps} (\log x) \Q^{\eps} M^{1/2} \mathcal{S}_1(M)^{1/2},
$$
where $\mathcal{S}_1(M)$ is defined as in \cite[display just below (5.3)]{Drappeau-Shao} except that the sum over primes runs over $y' \leq p \leq y$:
$$
\mathcal{S}_1(M) = 
\sum_{\substack{d \mid q^{\infty} \\ d \leq \Q}}
\sum_{y'\leq p \leq y} 
\sum_{M/p <m \leq 2M/p} 
\bigg| 
  \sum_{\substack{n \in S([p,y]) \\ (pmn,q)=1 \\ x/(2M\Q)\leq n\leq x/(Md) }}  e\bigbrac{(pmnd)^k\theta + \beta n}
\bigg|^2.
$$
Here we relaxed the summation conditions on $d$ (to match the corresponding sum in \cite{Drappeau-Shao}) and, following \cite{harper-minor-arcs, Drappeau-Shao}, on $m$, which allows one to later use a standard exponential sum estimate where the summation variable runs over a full interval.
Expanding out the square, swapping the order of summation, applying the triangle inequality and 
relaxing some of the summation conditions in the outer sum leads to
$$
\mathcal{S}_1(M) \ll 
\sum_{\substack{d \mid q^{\infty} \\ d \leq \Q}} \sum_{y'\leq p \leq y} 
\sum_{\substack{n_1, n_2 \in S([y',y]) \\ (n_1n_2,q)=1 \\ x/(2M\Q)< n_1 \leq n_2 \leq x/(Md)}}
\bigg|
\sum_{M/p < m \leq 2M/p} e\Big( (pmd)^k \theta (n_1^k-n_2^k)\Big)
\bigg|.
$$
Note carefully that instead of relaxing the condition $n_1, n_2 \in S([p, y])$ to $n_1, n_2 \in S(y)$, as in 
\cite{Drappeau-Shao}, we replaced it by $n_1, n_2 \in S([y', y])$.
We now follow the analysis of $\mathcal{S}_1(M)$ from \cite{Drappeau-Shao}, with the only change that in the definitions of $\mathcal{S}_j$, $j \in \{2,\dots,5\}$ the variables $n_1$ and $n_2$ continue to be restricted to 
$S([y',y])$.
Instead of bounding the quantity $\mathcal{S}_5(M;r',d;n_1,b;T)$ with the help of \cite[Lemma 3.3]{Drappeau-Shao}, we use Lemma \ref{lem:short-progressions-Harper}, which produces the analogues bound with
$\Psi(x/(Md),y)$ replaced by $\Psi(x/(Md),[y',y])$.
In the definition of $\mathcal{S'}_3$, we include the restriction to $n_1 \in S([y',y])$.
Adapting the analysis of $\mathcal{S}_3$ requires the lower bound 
\begin{align*}
\Psi\left(x/(Md),y\right)
&\asymp 
\left(\frac{x}{Md}\right)^{\alpha'} (\log x\log y)^{-1/2}
\prod_{y' \leq  p  \leq y} (1-p^{-\alpha'})^{-1} \\
&\gg \left(\frac{x}{Md}\right)^{\alpha'} (\log x\log y)^{-1/2} 
\gg \left(\frac{x}{Md}\right)^{\alpha' + o(1)} \gg \left(\frac{x}{Md}\right)^{\alpha},
\end{align*}
where $\alpha' = \alpha(x/(Md), y)$,
which follows from Lemma \ref{lem:BrTen-2.1}, the asymptotic expansion \eqref{eq:Hil-Ten-Thm1} 
as well as the bound $\alpha' = \alpha(x/(Md), y) > \alpha(x,y)=\alpha$ implied by the lower bound 
$M \leq L = 4 y \Q > \log x$ and the estimate \eqref{alpha-x-y}.

The remaining analysis in the proof of \cite[Proposition 5.7]{Drappeau-Shao} involves two applications of 
\cite[Lemma 3.2]{Drappeau-Shao}, which may be replaced by applications of 
Lemma \ref{lem:BrTen-2.4.i-ii} (i) combined with Lemma \ref{lem:BrTen-2.1} since for 
$d \leq \Q$ and $M \leq yL$ we have
$$
Md \leq y L \Q \leq 4y^2 \Q^2 \leq x/(y \Q) \leq x/y,
$$
thanks to our stronger assumption that $4 y^3 \Q^3 \leq x$. 
This ensures that $d \leq x/(yM)$ in the first application and that $M < x/y$ in the second application, i.e.\ that the conditions of Lemma \ref{lem:BrTen-2.4.i-ii} (i) are satisfied.
Our lemma thus follows from the proof of \cite[Proposition 5.7]{Drappeau-Shao}.
\end{proof}

\begin{proof}[Proof of Theorem \ref{thm:x-eps-smooth}]
The first part of the result follows directly from Lemma \ref{lem:wooley} applied with $\sigma=1/12$.
Concerning the second part, suppose that $\theta \in \mathfrak{M}(x^{1/12},x)$ and recall that by definition \eqref{def:majorarc} there exists a positive integer $q \leq x^{0.1}$ such that 
$\norm{q\theta}\leq x^{1/12} x^{-k}$. 
For any such value of $q$, we have $\Q = q + \|q\theta\| x^k \leq 2x^{1/12}$ and 
$4y^3\Q^3  \leq x$ since $y< x^{\eta}$ and $\eta < 1/4$. 
It then follows from Lemma \ref{lem:drappeau-shao} that
\[
E_k(x,[y',y];\theta)\ll \Q^{-c+2(1-\alpha)}(\log x)^5\Psi(x,[y',y]).
\]

\end{proof}

\section{Strongly recurrent polynomial sequences over smooth numbers} \label{sec:bootstrap}

Our aim in this section is to show that if a sequence $\{\|\beta n^k\|\}_{n \in S(x,[y',y])}$ is strongly recurrent, then $\beta$ is very close to a rational with small denominator $q$.
More precisely:

\begin{theorem} \label{thm:strong-recurrence}
Let $N>2$ be a parameter, let $\theta\in\RR$, and let $k\geq1$ be a fixed integer. 
Let $K' > 0$, $K>\max (2, 2K')$, $1 \leq y' \leq (\log N)^{K'}$, and suppose that 
$(\log N)^K \leq y \leq N^{\eta}$, for some small constant 
$\eta = \eta(k) \in (0,1)$. 
Let $\delta: \RR_{>0} \to \RR_{>0}$ be a function of $N$ that satisfies 
$$\log_5 N \ll \delta(N)^{-1} < \log_2 N $$
for all sufficiently large $N$.
Finally suppose that $0<\eps\leq \delta/2$ and that
$\eps < \delta(N)^{-O(1)} (\log N)^{-C_1}$ for some fixed constant $C_1 \geq 1$. 
Then the following assertion holds provided that $K$ and $C_1$ are sufficiently large depending on $k$.

If the bound $\|n^k \theta\| \leq \eps$ holds for at least $\delta\Psi(N,[y',y])$ elements 
$n\in S(N,[y',y])$, then there exists a positive integer $0<q\ll\delta^{-O(1)}$
such that
$$
\norm{q\theta}\ll \eps \delta^{-O_{C_1}(1)} /N^k.
$$
\end{theorem}

The corresponding problem for strongly recurrent \emph{unrestricted} polynomial sequences 
$\{\|\beta n^k\|\}_{n \leq x}$ has been treated by Green and Tao in \cite[\S3]{GT-nilmobius}
where the polynomial case was reduced to the linear case via an application of bounds 
in Waring's problem.
While a suitable Waring-type result was proved by Drappeau and Shao in \cite[Theorem 2.4]{Drappeau-Shao} for the set $S(x,y)$, no such result is currently available for the sparse subset $S(x,[y',y])$, although we expect such an analogue to hold.

For this reason, our approach proceeds instead via first reducing the problem  via Fourier analysis to bounds on Weyl sums over $S(x,[y',y])$, which can then be combined with the results of Section \ref{sec:weyl-sums}.
The Fourier analysis reduction is standard (cf.\ \cite[Proposition 3.1 and Lemma 3.2]{GT-polyorbits}).
The bounds on the relevant Weyl sums obtained in Section \ref{sec:weyl-sums} by themselves are, however, not strong enough in order to deduce the theorem above; they only provide part (1) of Lemma \ref{lem:strong-recurrence} below.
In particular, those bounds on $\|q\theta\|$ are not strong enough in order to later analyse the correlation of $g_{[y',y]}(n)$ with nilsequence for very small values of $y$, that is when $\log y \ll \log \log x$.
To work around this problem, we will employ a bootstrapping argument in order to improve the bounds on 
$\|q \theta\|$. 
This argument is based on the combination of the following lemma due to Drappeau and Shao \cite[Lemma 3.7]{Drappeau-Shao} (which is a higher-dimensional version of the bootstrapping argument used in the proof of \cite[Lemma 3.2]{GT-polyorbits}) 
and the bounds we obtained in Section~\ref{sec:W_trick_and_equid} on equidistribution of $S([y',y])$ in short progressions.

\begin{lemma}[Bootstrapping lemma, Drappeau--Shao \cite{Drappeau-Shao}]
\label{lem:bootstrapping}
 Let $k>0$ be a fixed integer and let $\eps', \delta \in (0,1)$. 
 Let $1 \leq L \leq x$ be parameters and let $\mathcal{A} \subset [x,2x]$ be a non-empty subset with the property that
 $$
 |\mathcal{A} \cap P| \leq \Delta \frac{|\mathcal{A}||P|}{x}
 $$
 for some $\Delta \geq 1$ and 
 for any arithmetic progression $P \subseteq [x,2x]$ of length at least $L$ and common difference $q=1$.
 Suppose that for some $\vartheta \in \RR$ with $\| \vartheta \| \leq \eps' / (Lx^{k-1})$, there are at least 
 $\delta |\mathcal{A}|$ elements $m \in \mathcal{A}$ satisfying $\|m^k \vartheta\| \leq \eps'$.
 Then either $\eps' \gg \delta / \Delta$ or 
 $$
 \vartheta \ll \Delta \delta^{-1} \eps'/x^k.
 $$
\end{lemma}

\begin{rem}
 The original statement of \cite[Lemma 3.7]{Drappeau-Shao} does not restrict the progressions $P$ to have common difference $q=1$. An inspection of the proof reveals, however, that the corresponding assumption is only required for progressions with common difference $q=1$, i.e.\ for discrete intervals.
\end{rem}

Theorem \ref{thm:strong-recurrence} is an immediate consequence of the following lemma, which will help us structure our proof.

\begin{lemma} \label{lem:strong-recurrence}
Let $\theta\in\RR$, $k\geq1$ a fixed integer, and let $K'>0$ and $K>\max(2K',2)$ be constants. 
Let $1 \leq y'\leq (\log x)^{K'}$ and suppose that $(\log x)^K\leq y\leq x^{\eta}$ for some small constant $\eta \in (0,1)$. 
Let $\delta = \delta(x)$ be such that $\delta(x)^{-B} \ll_B x$ for all $B>0$ and suppose that 
$0<\eps\leq\frac{\delta}{2}$. 
Suppose that there are at least $\delta\Psi(x,[y',y])$ elements $n\in S(x,[y',y])$ for which 
$\|n^k \theta\| \leq \eps$. Then:
\begin{enumerate}
 \item \label{item:SR1}
 There is some integer $0<q\ll\delta^{-O(1)}$ and a constant $c>0$ depending on $k$ such that
$$
\norm{q\theta}\ll \delta^{-O(1)} (\log x)^{10/c}/x^k
$$
 provided $K$ is sufficiently large in terms of $c$.
\item \label{item:SR2} If, furthermore, $\delta(x)>(\log_2 x)^{-1}$ and $\eps < \delta(x)^{-O(1)} (\log x)^{-C_1}$ for some fixed constant $C_1 \geq 5 + 20/c$, then the integer $0<q\ll\delta^{-O(1)}$ from part (\ref{item:SR1}) satisfies
$$
\norm{q\theta}\ll\eps\delta^{-O(1)} (\log x)^{C_1}/x^k,
$$
provided that $K$ is sufficiently large in terms of $c$.
\item \label{item:SR3} Under the assumptions of part (\ref{item:SR2}) 
the integer $0<q\ll\delta^{-O(1)}$ from part (\ref{item:SR1}) satisfies the bound
$$
\norm{q\theta}\ll\eps\delta^{-O_{C_1}(1)} (\log_3 x)/x^k,
$$
provided that $K$ is sufficiently large in terms of $c$.
\item \label{item:SR4} If, in addition to all previous assumptions,
$\delta^{-1} \gg \log_5 x$, 
then the integer $0<q\ll\delta^{-O(1)}$ from part (\ref{item:SR1}) is such that
$$
\norm{q\theta} \ll \eps\delta^{-O_{C_1}(1)}/x^k,
$$
provided that $K$ is sufficiently large in terms of $c$.
\end{enumerate}

\end{lemma}

\begin{proof}[Proof of Lemma \ref{lem:strong-recurrence}]
Suppose that $I\subseteq[0,1]$ is an interval of length $|I|=\eps$, and that there are at least 
$\delta \Psi(x,[y',y])$ elements $n\in S(x,[y',y])$ such that $n^k\theta\Mod\ZZ\in I$.
By approximating the characteristic function of the interval $I$ by a suitable smooth Lipschitz function $F$,
we obtain
\begin{equation} \label{eq:Lip-approx}
 \Bigabs{\sum_{n\in S(x,[y',y])}F(n^k\theta\Mod \ZZ)}\geq\delta'\Psi(x,[y',y]),
\end{equation}
for some $\delta' \in [\delta/2,\delta]$. 
We may assume that $F$ is supported on a set of measure $\delta$. Hence, on rescaling $F$ and redefining
$\delta'$ in the bound above as $\delta'=\delta^{C}$ for some $C \geq 1$, 
we may assume that $\norm{F}_{\mathrm{Lip}}=1$, which implies that, for every $k \in \ZZ$,
\begin{equation} \label{eq:F-hat-bound}
 |\hat F (k)| = \int_{\RR/\ZZ} F(\theta) e(-\theta k)\,\mathrm d\theta \leq \|F\|_{\infty} \leq \|F\|_{\mathrm{Lip}} = 1.
\end{equation}

Following the proof of \cite[Lemma 3.1]{GT-polyorbits}, our next step is to approximate $F$ by a function whose Fourier transform is finitely supported, with a support defined in terms of $\delta'$. 
For this purpose, consider the Fej\'er kernel $K(\theta) = \chi_Q*\chi_Q(\theta)$, 
where $\chi_Q(\theta) = \frac{\delta'}{16} \1_Q(\theta)$ is the normalised characteristic function of the short interval $Q=[-\delta'/16, \delta'/16]$.
Since $\int_{\RR/\ZZ} K = 1$ and by \eqref{eq:F-hat-bound}, the convolution $F_1=F*K$ satisfies
$$
\|F-F_1\|_{\infty} \leq\frac{\delta'}{4} \qquad \text{ and } \qquad |\hat F_1(k)| \leq 1 \text{ for all } k \in \ZZ.
$$
The rapid decay of the Fourier coefficients of $K$ allows one to approximate $F_1$ by a finite truncation of its Fourier series as follows. If
\[
F_2(\theta):=\sum_{0<|q|\ll\delta'^{-3}}\hat F_1(q)e(q\theta),\qquad\text{then}\qquad \norm{F_2-F_1}_\infty\leq\frac{\delta'}{4}.
\]
When applying both of these approximations to \eqref{eq:Lip-approx} and then swapping the order of summation,
it follows from the triangle inequality that
\[
\frac{\delta'}{2}\Psi(x,[y',y])
\leq\sum_{0<|q|\ll\delta'^{-3}}|\hat F_1(q)|\Bigabs{\sum_{n\in S(x,[y',y])}e(qn^k\theta)}
\leq\sum_{0<|q|\ll\delta'^{-3}}\Bigabs{\sum_{n\in S(x,[y',y])}e(qn^k\theta)}.
\]
By the pigeonhole principle there therefore is some integer $0<|q|\ll\delta'^{-3}$ such that
\begin{equation}
\label{eq:Fourier-bd}
\delta'^4\Psi(x,[y',y])\ll\Bigabs{\sum_{n\in S(x,[y',y])}e(qn^k\theta)} = |E_k(x,[y',y];q\theta)|.
\end{equation}

Following the above reduction via Fourier analysis, we are now in the position to invoke the bounds on Weyl sums over smooth numbers from Section \ref{sec:weyl-sums}.
To start with, the first part of Theorem~\ref{thm:x-eps-smooth} shows that, if $q\theta\Mod\ZZ$ does not belong to 
$\mathfrak M(x^{1/12},x)$, then the right hand side of \eqref{eq:Fourier-bd} is bounded by 
\[
E_k(x,[y',y];q\theta) 
\ll x^{1-c} 
\ll x^\alpha x^{1-\alpha - c} \ll x^{-c'}\Psi(x,[y',y]),
\]
where $c' := c-(1-\alpha)$ and where the lower bound $\Psi(x,[y',y]) \gg x^{\alpha+o(1)}$ follows from \eqref{eq:Hil-Ten-Thm1} and Lemma \ref{lem:MV}. 
Observe that $c' > 0$ provided that $K$ is sufficiently large in terms of $c$ since 
$1-\alpha(x,y) \leq \frac{1}{K} + o(1)$ by \eqref{alpha-x-y}. 
Hence, $\delta'^{4} \ll x^{-c'}$ and, thus, $x \ll {\delta'}^{-4/c'} = \delta^{-4C/c'}$, contradicting our assumptions on $\delta$, which implies $\delta^{-4C/c'} \ll_B x^{1/B}$ for all $B>0$.

It follows that $q\theta\Mod\ZZ\in\mathfrak M(x^{1/12},x)$.
In this case, the second part of Theorem \ref{thm:x-eps-smooth}, applied with $\theta$ replaced by $q\theta$ 
(and applied in the special case where, in the statement of Theorem \ref{thm:x-eps-smooth}, $q=a=1$), shows that
\[
\sum_{n\in S(x,[y',y])}e(qn^k\theta)\ll\bigbrac{1+x^k\norm{q\theta}}^{-c+2(1-\alpha)}(\log x)^5\Psi(x,[y',y]).
\]
Suppose that $K$ is sufficiently large to ensure that $c-2(1-\alpha)>c/2$.
Then it follows from the Fourier analysis bound \eqref{eq:Fourier-bd} that
\[
1+x^k\norm{q \theta} \ll \delta^{-8C/c} (\log x)^{10/c},
\]
and hence
\[
\norm{q\theta}\ll\delta^{-8C/c} (\log x)^{10/c} x^{-k},
\]
which proves {\em part (\ref{item:SR1})}.
For later use in the proofs of all remaining parts, we record the following consequence of part (1) and our assumptions.

\begin{rc}
Let $0 < q \ll \delta^{-O(1)}$ denote the integer produced by part (1). Then we have 
\begin{equation} \label{eq:recurrence-condition}
 \|n^k q \theta\| \leq q\|n^k \theta\| < \eps' 
\end{equation}
for some $\eps' \asymp q \eps \ll \delta^{-O(1)} \eps$ and for at least $\delta \Psi(x,[y',y])$ elements $n$ of 
$S(x,[y',y])$.
\end{rc}

To establish {\em part (\ref{item:SR2})}, observe that in view of Lemma \ref{lem:short-progressions-Harper}, the conditions of Lemma~\ref{lem:bootstrapping} are satisfied for the set $\mathcal{A}= S(2x,[y',y])$ and any lower bound $L \geq 1$  with a correction factor of the form $\Delta = (x/L)^{1-\alpha}\log x$.
Moreover, we have
$$\|q \theta\| \ll \delta^{-8C/c} (\log x)^{10/c} x^{-k} = \eps' /(Lx^{k-1})$$ 
if we set $L = \eps' \delta^{8C/c}(\log x)^{-10/c} x$ and with $\eps'$ as in \eqref{eq:recurrence-condition}. 
In this case, 
\begin{align*}
\eps'\Delta &\leq \eps' (\log x) 
(\eps' \delta^{8C/c}(\log x)^{-10/c})^{\alpha-1}
\ll (q \eps)^{\alpha} \delta^{-O(1)}(\log x)^{1 + 5/c} \\
&\ll (\log x)^{- \alpha C_1} \delta^{-O(1)}(\log x)^{1 + 5/c}
= o(\delta(x))
\end{align*}
provided that $C_1$ is sufficiently large depending on $c$. 
Here, we used that $\alpha > 1/2$ if $K>2$.
To find a simple bound for $\Delta$, note that
$$x/L \ll (\eps q \delta^{8C/c}(\log x)^{-10/c})^{-1} 
\ll (\log x)^{C_1} \delta^{-O(1)}(\log x)^{10/c} \ll (\log x)^{3C_1/2},$$
provided that $C_1$ is sufficiently large (e.g. $C_1 \geq 5 + 20/c$), which implies that 
$$
\Delta = (x/L)^{1-\alpha}\log x \ll (\log x)^{1 + 3C_1(1-\alpha)/2} \ll (\log x)^{3C_1/4 +1} \leq (\log x)^{C_1}.
$$
Thus, in view of the Recurrence Condition above, the conclusion of part (\ref{item:SR2}) follows from Lemma \ref{lem:bootstrapping}, applied with the given value of $\eps'$  and $\vartheta := q \theta$, provided that
$C_1$ is sufficiently large in terms of $c$.

To prove {\em part (\ref{item:SR3})}, we use the information from part (\ref{item:SR2}), i.e.\ that the positive integer $q \ll \delta^{-O(1)}$
produced by part (\ref{item:SR1}) satisfies
\begin{equation} \label{eq:bootstrap-start}
\|q\theta\| \ll \eps \delta^{-O(1)} (\log x)^{C_1}/x^k. 
\end{equation}
We shall now apply the short intervals case of Corollary \ref{cor:W-tricked-equid-short-APs} (1).
This result and the Recurrence Condition imply that the conditions of the bootstrapping lemma are satisfied for the set 
$\mathcal{A} = S(x,[y',y])$, the lower bound $L = x/(\log x)^{\ell}$ and 
$$\Delta = \log_3 x + (\log x)^{\ell - 1/24}$$ 
for any constant $\ell > 0$. 
Observe that, since $\eps < \delta^{-O(1)} (\log x)^{-C_1}$ and $\delta(x)^{-1}<\log_2 x$, 
a bound of the form $\ell \leq C_1$ implies
\begin{align*}
 \eps' \delta(x)^{-O(1)} \Delta 
&<  \delta(x)^{-O(1)} (\log_3 x + (\log x)^{\ell-1/24}) (\log x)^{-C_1} \\
&< (\log x)^{-C_1 + o(1)} + (\log x)^{-1/24 + o(1)} 
= o(\delta(x)).
\end{align*}
This ensures that the bootstrapping lemma produces bounds on $\|q\theta\|$ in all applications below.

Let $0 \leq j < 24 C_1$ be an integer, and suppose inductively that
$$
\|q\theta\| \ll \eps \delta^{-O_{C_1}(1)} (\log x)^{C_1 -j/24}/x^k,
$$
the case $j=0$ being \eqref{eq:bootstrap-start}. Let $\ell := C_1 -j/24$, and $L = x/(\log x)^{\ell}$.
Then the bootstrapping lemma, applied with $\eps'$ replaced by $\eps \delta^{-O_{C_1}(1)}$, implies that
$$
\|q\theta\| \ll \eps \delta^{-O_{C_1}(1)} (\log x)^{C_1 -(j+1)/24}/x^k
$$
if $j+1< 24 C_1$. Recalling the shape of $\Delta$, we pick up a $(\log_3 x)$-factor when $j+1\geq24 C_1$ and thus 
$\ell - 1/24 \leq 0$, and obtain
$$
\|q\theta\| 
\ll \eps \delta^{-O_{C_1}(1)} (\log_3 x) /x^k
$$
as claimed.

It remains to establish {\em part (\ref{item:SR4})}. In view of part (\ref{item:SR3}), we have 
$$\|q\theta\|  
\ll \eps \delta^{-O_{C_1}(1)} (\log_3 x) /x^k$$
and there are $\delta \Psi(x,[y',y])$ elements $n \in  S(x,[y',y])$ for which $\|n^k q\theta\| < \eps'$
by \eqref{eq:recurrence-condition}.
We shall now appeal to Corollary \ref{cor:W-tricked-equid-short-APs} (2) in order to remove the $(\log_3 x)$-factor from the bound.
For this purpose let $\varpi(x) = \log_4 x$ and recall the Recurrence Condition.
Then Corollary~\ref{cor:W-tricked-equid-short-APs}~(2) shows that the conditions of Lemma \ref{lem:bootstrapping} are satisfied for the set $\mathcal{A} = S(x,[y',y])$, the lower bound $L = x/\log_3 x$ and 
$\Delta = \log \varpi(x) \ll \log_5 x \ll \delta^{-1}$.
By the assumptions on $\delta$ and $\varpi(x)$ it follows easily that
$$
\Delta \eps'\delta^{-O_{C_1}(1)} \ll \delta^{-O_{C_1}(1)}(\log x)^{-C_1} = o(\delta).
$$ 
Hence, Lemma \ref{lem:bootstrapping} applied with $\eps'$ replaced by $\eps \delta^{-O_{C_1}(1)}$ 
finally leads to the bound
$$
\|q\theta\| \ll  \eps \delta^{-O_{C_1}(1)} /x^k.
$$
\end{proof}

\section{Non-correlation with nilsequences. A reduction and general lemmas} 
\label{sec:nilsequences}

The aim of this section is to first establish an initial reduction of non-correlation estimate stated in Theorem \ref{thm:main-uniformity} to the case where the nilsequence is equidistributed.
In preparation for the proof of the reduced version we then show that most sequences in certain sparse families of subsequences of an equidistributed polynomial sequence are equidistributed.
We start by recalling some notation around nilsequences.

\begin{definition}[Filtered nilmanifold]
Let $d, m_G \geq 0$ be integers and let $M>0$. 
We define a {\em filtered nilmanifold $G/\Gamma$ of degree $d$, dimension $m_G$ and complexity at most $M$} to be an $s$-step nilmanifold $G/\Gamma$, for some $1\leq s \leq d$, of dimension $m_G$ in the sense of 
\cite[Definition 1.1]{GT-polyorbits} such that $G$ is equipped with a filtration $G_{\bullet}$ of degree $d \geq s$ in the sense of \cite[Definition 1.1]{GT-polyorbits} and such that the Lie algebra $\mathfrak{g} = \log G$ is equipped with an $M$-rational Mal'cev basis adapted to $G_{\bullet}$ in the sense of \cite[Definition 2.1]{GT-polyorbits}. 
\end{definition}

A Mal'cev basis $\mathcal{X}$ gives rise to a metric $d_{\mathcal{X}}$ on $G/\Gamma$ 
(see \cite[Definition 2.2]{GT-polyorbits}).
With respect to this metric, Lipschitz functions can be defined. 
More precisely, if $F: G/\Gamma \to \CC$, we define (cf.\ \cite[Definition 1.2]{GT-polyorbits}) the Lipschitz norm
$$
\|F\|_{\mathrm{Lip}} := 
\|F\|_{\infty} + \sup_{x,y \in G/\Gamma,\, x \neq y} \frac{|F(x)-F(y)|}{ d_{\mathcal{X}}(x,y)}
$$
and call $F$ a Lipschitz function if $\|F\|_{\mathrm{Lip}} < \infty$.

\begin{definition}[Polynomial sequence and nilsequence]
Given a nilpotent Lie group $G$ and a filtration 
$$G_{\bullet}: \quad G= G_0 = G_1 \supseteq G_2 \supseteq \dots \supseteq G_d \supseteq G_{d+1} = \{ \id_G\},$$ 
the set $\mathrm{poly}(\ZZ, G_{\bullet})$ of \emph{polynomial sequences} is defined as the set of all maps 
$g: \ZZ \to G$ 
such that, if $\partial_h g(n) := g(n+h) g(n)^{-1}$, 
the $j$-th discrete derivative $\partial_{h_j} \dots \partial_{h_1} g$ takes values in $G_j$ 
for all $j \in \{1, \dots, d+1\}$ and all $h_1 \dots h_j \in \ZZ$.
If $g \in \mathrm{poly}(\ZZ, G_{\bullet})$ and $F: G/\Gamma \to \CC$ is a Lipschitz function, then the sequences $\ZZ \to F(g(\cdot) \Gamma)$ is called a \emph{nilsequence}.
\end{definition}

With this notation in place, we restate the main result of our paper, which shows that
$n \mapsto (g_{[y',y]}^{(W,A)}(n) - 1)$ is orthogonal to nilsequences.

\begin{theorem}[Non-correlation with nilsequences]
\label{thm:non-corr}
Let $N$ be a large positive parameter and let 
$K' \geq 1$, $K > 2K'$ and $d \geq 0$ be integers. 
Let $\frac{1}{2}\log_3 N \leq  y' \leq (\log N)^{K'}$ and suppose that $(\log N)^K < y_0 < y < N^{\eta}$ 
for some sufficiently small $\eta \in (0,1)$ depending the value of $d$.
Let $(G/\Gamma, G_{\bullet})$ be a filtered nilmanifold of complexity $Q_0$ and degree $d$.
Finally, let $w(N) = \frac{1}{2}\log_3 N + o(1)$, $W=P(w(N))$ and define 
$\delta(N) = \exp(-\sqrt{\log_4 N})$.

If $K$ is sufficiently large depending on the degree $d$ of $G_{\bullet}$, then the estimate 
\begin{align} \label{eq:thm-non-corr-bound-g}
\bigg| \frac{\W}{N}
\sum_{n \leq (N-A)/\W} 
(g_{[y',y]}^{(\W,A)}(n)-1)
F(g(n)\Gamma)
\bigg| 
\ll_{d,C} (1 +\|F\|_{\mathrm{Lip}}) \delta(N) Q_0 + \frac{1}{\log w(N)}
\end{align}
holds uniformly for all $\W = Wq$, where $q \leq (\log y_0)^C$ satisfies $p \mid q \Rightarrow p < w(N)$ and $C \geq 1$ is a fixed constant, for all $1\leq A\leq \W$ with $\gcd(A,W)=1$, all polynomial sequences $g \in \mathrm{poly}(\ZZ,G_{\bullet})$ and all $1$-bounded Lipschitz functions $F:G/\Gamma \to \CC$.
\end{theorem}

By decomposing the summation range $1 \leq n \leq (N-A)/\W$ into dyadic intervals 
$n \sim (N'-A)/\W$ for $N' \in (N/\log N, N]$ and the initial segment $1 \leq n \ll N/(\W\log N)$, 
and approximating $g_{[y',y]}$ by $h_{[y',y]}$ on each dyadic interval,
it follows from Lemma \ref{lem:g-h-approximation} that the task of proving Theorem \ref{thm:non-corr} may be reduced to proving that the estimate
\begin{align} \label{eq:thm-non-corr-bound}
\bigg| \frac{\W}{N}
\sum_{n \sim (N-A)/\W} 
(h_{[y',y]}^{(\W,A)}(n)- 1)
F(g(n)\Gamma)
\bigg| \ll_{d,C} (1 +\|F\|_{\mathrm{Lip}}) \delta(N) Q_0 + \frac{1}{\log w(N)}
\end{align}
holds under the assumptions of Theorem \ref{thm:non-corr} and for the function $h_{[y',y]}^{(\W,A)}$ that is defined on the interval $n \sim (N-A)/\W$.

\subsection{Reduction to non-correlation with equidistributed nilsequences}

Observe that the bound in the statement of Theorem \ref{thm:non-corr} holds trivially unless 
$$Q_0 \leq \delta(N)^{-1}.$$
This information can be used in combination with Green and Tao's factorisation theorem for polynomial sequences 
(which states that every polynomial sequence is the product of a slowly varying, a highly equidistributed and a periodic polynomial sequences) together with our results on the distribution of $n \mapsto h_{[y',y]}$ in short arithmetic progressions in order to reduce the statement to one in which the sequence $g$ can be assumed to be equidistributed.
Before stating the reduced version and proving the reduction, we recall the relevant definitions around equidistribution as well as the factorisation theorem.

\begin{definition}[$\delta$-equidistributed and totally $\delta$-equidistributed sequence]\cite[Definition 1.2]{GT-polyorbits}\label{almost-equidistribution}
  Let $G/\Gamma$ be a nilmanifold.
	\begin{enumerate}
		\item
		Given a length $N > 0$ and an error tolerance $\delta > 0$, a finite sequence $(g(n)\Gamma)_{n \in [N]}$ is said to be \emph{$\delta$-equidistributed} if we have
		$$ \left|\EE_{n \in [N]} F(g(n) \Gamma) - \int_{G/\Gamma} F\right| \leq \delta \|F\|_{\operatorname{Lip}}$$
		for all Lipschitz functions $F: G/\Gamma \to \CC$.
		\item A finite sequence $(g(n)\Gamma)_{n \in [N]}$ is said to be \emph{totally $\delta$-equidistributed} if we have
		$$ \left|\EE_{n \in P} F(g(n) \Gamma) - \int_{G/\Gamma} F\right| \leq \delta \|F\|_{\operatorname{Lip}}$$
		for all Lipschitz functions $F: G/\Gamma \to \CC$ and all arithmetic progressions $P \subset [N]$ of length at least $\delta N$.
	\end{enumerate}
\end{definition}

\begin{definition}[Rational sequence]\cite[Definition 1.17]{GT-polyorbits}\label{rat-def-quant}
	Let $G/\Gamma$ be a nilmanifold and let $Q > 0$ be a parameter. We say that $\gamma \in G$ is \emph{$Q$-rational} if $\gamma^r \in \Gamma$ for some integer $r$, $0 < r \leq Q$. A \emph{$Q$-rational point} is any point in $G/\Gamma$ of the form $\gamma\Gamma$ for some $Q$-rational group element $\gamma$. A sequence $(\gamma(n))_{n \in \ZZ}$ is \emph{$Q$-rational} if every element $\gamma(n)\Gamma$ in the sequence is a $Q$-rational point.
\end{definition}

\begin{definition}[Smooth sequences]\cite[Definition 1.18]{GT-polyorbits}\label{smooth-seq-def}  Let $G/\Gamma$ be a nilmanifold with a Mal'cev basis $\X$.  Let $(\eps(n))_{n \in \ZZ}$ be a sequence in $G$, and let $M, N \geq 1$.  We say that $(\eps(n))_{n \in \ZZ}$ is \emph{$(M,N)$-smooth} if we have $d(\varepsilon(n),\operatorname{id}_G) \leq M$ and $d(\varepsilon(n),\varepsilon(n-1)) \leq M/N$ for all $n \in [N]$, where $\id_G$ denotes the identity element of $G$.
\end{definition}

The following proposition is Green and Tao's factorisation theorem \cite[Theorem 1.19]{GT-polyorbits}
for polynomial sequences, which asserts that any polynomial sequence can be decomposed into a product of a smooth, a highly equidistributed and a rational polynomial sequence.

\begin{proposition}[Green--Tao factorization theorem \cite{GT-polyorbits}, \cite{TT}]\label{factorization}
	Let $m,d \geq 1$, and let $M_0, N \geq 1$ and $A > 0$ be real numbers.
	Suppose that $G/\Gamma$ is an $m$-dimensional nilmanifold together with a filtration $G_{\bullet}$ of degree $d$. Suppose that $\X$ is an $M_0$-rational Mal'cev basis adapted to $G_{\bullet}$ and that $g \in \mathrm{poly}(\ZZ,G_{\bullet})$. Then there is an integer $M$ with 
	$M_0 \leq M \ll M_0^{O_{A,m,d}(1)}$, 
	a rational subgroup $G' \subseteq G$, a Mal'cev basis $\X'$ for $G'/\Gamma'$ in which each element is an $M$-rational combination of the elements of $\X$, and a decomposition $g = \eps g' \gamma$ into polynomial sequences $\eps, g', \gamma \in \mathrm{poly}(\ZZ,G_{\bullet})$ with the following properties:
	\begin{enumerate}
		\item $\eps : \ZZ \rightarrow G$ is  $(M,N)$-smooth;
		\item $g' : \ZZ \rightarrow G'$ takes values in $G'$, and the finite sequence $(g'(n)\Gamma')_{n \in [N]}$ is totally $1/M^A$-equidistributed in $G'/\Gamma'$, using the metric $d_{\mathcal{X}'}$ on $G'/\Gamma'$;
		\item $\gamma: \ZZ \rightarrow G$ is $M$-rational, and $(\gamma(n)\Gamma)_{n \in \ZZ}$ is periodic with period at most $M$.
	\end{enumerate}
\end{proposition}

With the above notation in place, we may now state the equidistributed version of Theorem \ref{thm:non-corr} and 
deduce this theorem from it.

\begin{proposition}[Orthogonality to equidistributed nilsequences] 
 \label{prop:equid-non-corr}
Let $N$ be a large positive parameter, let $K' \geq 1$, $K>2K'$ and $d \geq 0$ be integers.
Let $\frac{1}{2}\log_3 N \leq y' \leq (\log N)^{K'}$ and suppose that $(\log N)^K < y_0 < y \leq N^{\eta}$ for some $\eta \in (0,1)$ that is sufficiently small depending on $d$.
Let $ w(N) = \frac{1}{2}\log_3 N$, $W=P(w(N))$ and 
$\delta(N) = \exp(-C_0 \sqrt{\log_4 N}))$
with $C_0 \in [1,(\log N)^{1/4}]$.

Let $G/\Gamma$ be a nilmanifold together with a filtration $G_{\bullet}$ of degree $d$, and suppose that $\mathcal X$ is a 
$\delta(N)^{-1}$-rational Mal'cev basis adapted to $G_{\bullet}$.
Let $g \in \mathrm{poly}(\ZZ,G_{\bullet})$ be any polynomial sequence 
such that the finite sequence 
$(g(n)\Gamma)_{n\leq 2N/W}$ is totally $\delta(N)^{E_0}$-equidistributed for some
$E_0 > 1$.
Let $F:G/\Gamma \to \CC$ be any $1$-bounded Lipschitz function such that
$\int_{G/\Gamma}F=0$.

If $1 \leq q \leq (\log y_0)^C$ is $(w(N)-1)$-smooth, where $1\leq C \ll1$, if $0 \leq a < q$ and $0<A<W$ are integers such that $(W,A)=1$ (and thus $(Wa+A,Wq)=1$), and if $0<N_1\leq N$, then we have
\begin{align} \label{eq:prop-bound}
\bigg|
\frac{Wq}{N}
\sum_{\substack{ m \in \NN: \\ N < Wqm < N+N_1}} 
h_{[y',y]}^{(W,A)}\left(qm+a\right)
F(g(m)\Gamma)
\bigg| 
\ll_{d, \dim G, \|F\|_{\mathrm{Lip}}, E_1}  \delta(N)^{E_1} 
\end{align}
for any given $E_1\geq 1$, provided that $E_0$ is sufficiently large with respect to $d$, $\dim G$ and $E_1$, provided that $K$ is sufficiently large depending on the degree $d$ of $G_{\bullet}$, and provided that $N$ is sufficiently large in terms of $\dim G$, $d$ and $E_0$.
\end{proposition}

\begin{proof}[Proposition \ref{prop:equid-non-corr} implies Theorem \ref{thm:non-corr}]
We shall prove that Proposition \ref{prop:equid-non-corr} implies \eqref{eq:thm-non-corr-bound}, from which Theorem \ref{thm:non-corr} follows.
We may assume that $Q_0 \leq \delta(N)^{-1}$ as \eqref{eq:thm-non-corr-bound} is trivially true otherwise.
Let $B>1$ be a parameter.
Then, by Proposition \ref{factorization}, applied with $N$ replaced by $2N/\W$, there exists 
$Q_0 \leq Q \ll Q_0^{O_{B,\dim G, d}(1)}$ 
and a factorisation of the polynomial sequences $g$ as 
$\eps g' \gamma$ that satisfies properties (1)--(3) of that proposition.
In particular, the polynomial sequence $\gamma: \ZZ \to G$ gives rise to a $\tilde q$-periodic function 
$\gamma(\cdot) \Gamma: \ZZ \to G/\Gamma$ for some period $1 \leq \tilde q \leq Q$,
and the sequence $\eps: \ZZ \to G$ is $(Q,2N/\W)$-smooth. 
The sequence $g': \ZZ \to G'$ takes values in a $Q$-rational subgroup $G'$ of $G$, it is a polynomial sequence with respect to the filtration $G'_{\bullet}:= G_{\bullet} \cap G'$ and the finite sequence 
$(g'(n)\Gamma')_{n\leq 2N/\W}$ is $Q^{-B}$-equidistributed in $G'/\Gamma'$, where $\Gamma' = \Gamma \cap G'$ and where equidistribution is defined with respect to the metric $d_{\mathcal{X'}}$ arising from a Mal'cev basis $\X'$ adapted to $G'_{\bullet}$. 
The existence of the Mal'cev basis $\X'$ is guaranteed by \cite[Lemma A.10]{GT-polyorbits}, which also allows us to assume that each of its basis elements is a $Q$-rational combination of basis elements from $\mathcal{X}$.

In order to reduce the non-correlation estimate to the case where the polynomial sequence is highly equidistributed, we seek to decompose the summation range of $n$ in \eqref{eq:thm-non-corr-bound} into (short) subprogressions on which $\gamma$ and $\eps$ are almost constant.
Splitting the interval $(N/\W,2N/\W]$ into arithmetic progressions with common difference $\tilde q$, we obtain
$$
\sum_{n\sim N/\W}\bigbrac{h_{[y',y]}^{(\W,A)}(n)-1}F(g(n)\Gamma)
=\sum_{0\leq a<\tilde q}\sum_{\substack{n\sim N/\W \\ n\equiv a \Mod{\tilde q}}}\bigbrac{h_{[y',y]}^{(\W,A)}(n)-1}F(\eps(n)g'(n)\gamma_a\Gamma),
$$
where $\gamma_a \in G$ is such that $\gamma(n)\Gamma=\gamma_a\Gamma$ whenever $n\equiv a\Mod {\tilde q}$. 

Since $F$ is a Lipschitz function and since $d_{\mathcal{X}}$ is right-invariant 
(cf.\ \cite[Appendix A]{GT-polyorbits}), we deduce that for any $n_0, n \in \ZZ$,
\begin{align*}
 \bigabs{F(\eps(n_0)g'(n)\gamma_a\Gamma)-F(\eps(n)g'(n)\gamma_a\Gamma)} 
 &\leq \|F\|_{\mathrm{Lip}} ~d_\X\bigbrac{\eps(n_0)g'(n)\gamma_a,\eps(n)g'(n)\gamma_a}\\
 &=\|F\|_{\mathrm{Lip}} ~d_\X\bigbrac{\eps(n_0),\eps(n)}.
\end{align*}
It then follows from the assumption $\eps$ is $(Q,2N/\W)$-smooth that
\[
d_\X\bigbrac{\eps(n_0),\eps(n)}\leq\frac{Q|n_0-n|}{N/\W}.
\]
whenever $n,n_0 \leq 2N/\W$, and therefore
\begin{align}\label{difference-eps}
\bigabs{F(\eps(n_0)g'(n)\gamma_a\Gamma) - F(\eps(n)g'(n)\gamma_a\Gamma)}\ll_{\|F\|_{\mathrm{Lip}}} \log^{-1} Q.
\end{align}
if, in addition, $|n_0-n| \ll N/(\W Q \log Q)$.
With this in mind, we refine the partition of our summation range and consider a sub-partition 
$$(N/\W, 2N/\W] = \bigcup_{j} P_j,$$
where each $P_j$ is a `short' progression of common difference $\tilde q$, as before, but with diameter bounded by 
$O(N/(Q \W \log Q))$. 
The bound on the diameter ensures that $\eps$ is almost constant on each $P_j$.
Note that the total number of short progressions $P_j$ is $O(\tilde q Q \log Q)$. 
By fixing an element $\eps_j\in \eps(P_j)$ in the image of $P_j$ under $\eps$ for each progression $P_j$, we obtain
\begin{multline}\label{remove-eps}
    \sum_{n\sim N/\W}\bigbrac{h_{[y',y]}^{(\W,A)}(n)-1}F(g(n)\Gamma)=\sum_{j}\sum_{n\in P_j}\bigbrac{h_{[y',y]}^{(\W,A)}(n)-1}F(\eps_jg'(n)\gamma_a\Gamma)\\
    +O\bigg\{\sup_{P_j}\sup_{n\in P_j}\bigabs{F(\eps_jg'(n)\gamma_a\Gamma)-F(\eps(n)g'(n)\gamma_a\Gamma)}\Bigbrac{\sum_{n\sim N/\W}h_{[y',y]}^{(\W,A)}(n)+N/\W}\bigg\}.
\end{multline}
By construction or, more precisely, by Lemma \ref{lem:equid-short-progressions}, we have 
\[
\sum_{n\sim N/\W}h_{[y',y]}^{(\W,A)}(n)\ll N/\W.
\]
Bounding the error in \eqref{remove-eps} with the help of this bound and \eqref{difference-eps}, we obtain
\[
\sum_{n\sim N/\W}\bigbrac{h_{[y',y]}^{(\W,A)}(n)-1}F(g(n)\Gamma)
=\sum_{j}\sum_{n\in P_j}\bigbrac{h_{[y',y]}^{(\W,A)}(n)-1}F(\eps_jg'(n)\gamma_a\Gamma)+O\Big(\frac{N}{\W\log Q}\Big).
\]

Observe that in the argument of $F$, apart from two constant factors, only the sequence $g'$ occurs, which has the property that $(g'(n)\Gamma)_{n \leq 2N/\W}$ is $Q^{-B}$-equidistributed.
We aim to use this equidistribution property in combination with Proposition \ref{prop:equid-non-corr} in order to bound the correlations on the right hand side above.
For this purpose, we shall now first 
show that the sequence $n \mapsto \eps_jg'(n)\gamma_a$ can be reinterpret as a polynomial sequence $n \mapsto g^*(n)$ that is equidistributed on some filtered nilmanifold $H/\Lambda$. 
At the same time, we show that $F(\eps_jg'(n)\gamma_a \Gamma) = \tilde F(g^*(n)\Lambda)$ for a Lipschitz function 
$\tilde F: H/\Lambda \to \CC$.
As a second step,
we carry out a reduction that allows us to assume that $\int_{H/\Lambda} \tilde F = 0$.
And, thirdly and finally, we will apply Proposition \ref{prop:equid-non-corr} and complete the proof.
We denote these steps (1)--(3) below.

(1) Define $g^*(n) := \gamma_a^{-1}g'(n)\gamma_a$. 
For an application of Proposition \ref{prop:equid-non-corr}, it is necessary to verify that $g^*$ is a polynomial sequence and that it inherits the equidistribution properties of $g'$.
These questions have been addressed by Green and Tao in \cite[\S2]{GT-nilmobius} and we follow their argument here.  
Let $H=\gamma_{a}^{-1} G' \gamma_{a}$ and define
$H_{\bullet} = \gamma_{a}^{-1} (G')_{\bullet} \gamma_{a}$.
Let $\Lambda = \Gamma \cap H$ and define
$\tilde{F} = \tilde{F}_{a,j}: H/\Lambda \to \RR$ via
$$
\tilde{F}(x \Lambda) 
= F(\eps_j \gamma_{a} x \Gamma).
$$
Then $g^* \in \mathrm{poly}(\ZZ,H_{\bullet})$, we have 
$\tilde{F}(g^*(n) \Lambda) =  F(\eps_j g'(n)\gamma_{a} \Gamma)$,
and the correlation  that we seek to bound takes the form
\begin{align} \label{eq:F-tilde}
\sum_{n \in P_j } 
\Big(h_{[y',y]}^{(\W,A)}(n)-1\Big)\tilde{F}(g^*(n) \Lambda).
\end{align}
The `Claim' from the end of \cite[\S2]{GT-nilmobius} guarantees the existence of  
a Mal'cev basis $\mathcal{Y}$ for $H/\Lambda$ adapted to $H_{\bullet}$ 
such that each basis element $Y_i$ is a $Q^{O(1)}$-rational combination of basis
elements $X_i$ from $\X$. 
Thus, there is $C'=O(1)$ such that $\mathcal{Y}$ is $Q^{C'}$-rational.
Furthermore, the `Claim' implies that there is $c'>0$, depending only on $\dim G$ and on the degree $d$ of $G_{\bullet}$, such that whenever $B$ is sufficiently large 
the sequence 
\begin{equation}\label{eq:h-seq}
 (g^*(n) \Lambda)_{n \leq 2N/\W}
\end{equation}
is totally $Q^{-c'B + O(1)}$-equidistributed in $H/\Lambda$, equipped with 
the metric $d_{\mathcal{Y}}$ induced by $\mathcal{Y}$.
Taking $B$ sufficiently large, we may assume that the sequence \eqref{eq:h-seq} is
totally $M^{-c'B/2}$-equidistributed.
Finally, the `Claim' also provides the bound 
$\|\tilde{F}\|_{\mathrm{Lip}} \leq Q^{C''} \|F\|_{\mathrm{Lip}}$ for some 
$C''=O(1)$.
This shows that all conditions of Proportion \ref{prop:equid-non-corr} are satisfied 
except for $\int_{H/\Lambda}\tilde{F}=0$.

(2) Let $\mu(\tilde{F}) = \int_{H/\Gamma} \tilde{F}$ denote the mean value of $\tilde{F}$ and observe that 
$\mu(\tilde{F}) \ll 1$ since $\tilde{F}$ is $1$-bounded. 
Then $\int_{H/\Gamma} \bar{F} = 0$ if $\bar{F} := \tilde{F}-\mu(\tilde{F}): H/\Gamma \to \CC$  and
\begin{align*}
&\sum_{n \in P_j} \Big(h_{[y',y]}^{(\W,A)}(n)-1\Big) \tilde{F}(g^*(n)\Lambda)\\
&= \sum_{n \in P_j} \Big(h_{[y',y]}^{(\W,A)}(n)-1\Big) \bar{F}(g^*(n)\Lambda)
+ O\bigg(\Big|\sum_{n \in P_j} \Big(h_{[y',y]}^{(\W,A)}(n)-1\Big)\Big|\bigg) \\
&= \sum_{n \in P_j} \Big(h_{[y',y]}^{(\W,A)}(n)-1\Big) \bar{F}(g^*(n)\Lambda)
+ O\bigg(\frac{|P_j|}{\log w(N)}\bigg)
\end{align*}
by Lemma \ref{lem:equid-short-progressions}.
We may thus assume that $\int_{H/\Gamma} \tilde{F} = 0$.

(3) By the previous two steps it remains to bound
$$
\sum_j \sum_{n \in P_j } 
\Big(h_{[y',y]}^{(\W,A)}(n)-1\Big)\tilde{F}(g^*(n) \Lambda),
$$
where we may assume that $\int_{H/\Gamma} \tilde{F} = 0$ and that $(g^*(n) \Lambda)_{n \leq 2N/\W}$ is 
$Q^{-c'B/2}$-equidistributed.
Since $P_j$ has common difference $\tilde q < Q$, the bound on the diameter of $P_j$ implies that 
$|P_j|\gg N/(\tilde q Q \W \log Q)$.
Note further that $Q \ll Q_0^{O_{B,d,\dim G}(1)} \ll \delta(N)^{-C}$ for some 
$C = O_{B,d,\dim G}(1)$ since $Q_0 \leq \delta(N)^{-1}$.
We may suppose that $N$ is sufficiently large for $C \leq (\log N)^{1/4}$ to hold.
Define $\delta'(N) = \delta(N)^{-C}$.

We may thus apply Proposition \ref{prop:equid-non-corr} with 
$W q$ replaced by $\W \tilde q = W(q \tilde q)$, with
$\delta$ replaced by $\delta'$, with 
$E_0 = c'B/2$, 
$g = g^*$, $G/\Gamma = H/\Lambda$, $\X = \mathcal Y$, and with $N_1 \gg N/(Q \log Q) > N \delta'(N)^{E_1/2}$ 
(assuming that $E_1 > 2$, which we may) to deduce that
\begin{align*}
\sum_j \sum_{n \in P_j} 
\Big(h_{[y',y]}^{(\W,A)}(n)-1\Big)\tilde{F}(g^*(n) \Lambda)
&\ll_{d} \sum_j (1 + \|\tilde F\|_{\mathrm{Lip}}) \delta'(N)^{E_1}\frac{N}{\W \tilde q} \\
&\ll_{d} (1 + Q^{C''} \| F\|_{\mathrm{Lip}}) \delta'(N)^{E_1} \frac{\tilde q QN \log Q}{\W \tilde q} \\
&\ll_{d} (1 + \| F\|_{\mathrm{Lip}}) Q^{O(1)} \delta'(N)^{E_1} \frac{N}{\W} \\
&\ll_{d} (1 + \| F\|_{\mathrm{Lip}})  \delta(N) Q_0 \frac{N}{\W},
\end{align*}
provided that $B$, and hence $E_1$, is sufficiently large to imply the final bound.
This completes the proof of the deduction of \eqref{eq:thm-non-corr-bound} and, hence, Theorem \ref{thm:non-corr}.
\end{proof}

\subsection{Sparse families of linear subsequences of equidistributed nilsequences} \label{sec:linear-subsequences}

In Section \ref{subsec:MV} we will relate our task of bounding the one-parameter correlation \eqref{eq:prop-bound} to that of bounding a bilinear sum.
This reduction naturally leads to the problem of understanding equidistribution properties in families of linear subsequences of polynomial sequences that arise as follows.
Let $(g(n)\Gamma)_{n\leq N}$ be a polynomial sequence and consider the family of sequences
$$
\{n \mapsto (g(mn)\Gamma)_{n\leq N/m}\}_{m \in S([y',y]) \cap [M,2M)},
$$
where the parameter $m \in [M,2M)$ is further restricted to the sparse set $S([y',y])$, yielding a sparse family of subsequences.
Our aim in this subsection is to show that if $(g(n)\Gamma)_{n\leq N}$ is $\delta$-equidistributed for a suitable choice of $\delta$, then almost all sequences in this family are $\delta^{1/C}$-equidistributed.
Equidistribution properties in unrestricted families of linear subsequences have been studied by the first author in \cite[\S7]{Mat-multiplicative}.
In this section we show that, thanks to the sparse recurrence result from Section \ref{sec:bootstrap}, 
the unrestricted result \cite[Proposition 7.4]{Mat-multiplicative} can indeed be extended to the sparse situation where $m \in S([y',y])$.
Part of this section follows \cite[\S7]{Mat-multiplicative} very closely.

The proof of Proposition \ref{prop:linear-subseqs} below uses the notion of a {\em horizontal character} (\cite[Definition 1.5]{GT-polyorbits}) on a nilmanifold $G/\Gamma$, which is defined to be a continues additive homomorphism $\eta: G \to \RR/\ZZ$ which annihilates $\Gamma$.
The set of horizontal characters may be equipped with a height function $|\eta|$ as defined in 
\cite[Definition 2.6]{GT-polyorbits}.
This specific height function is called the \emph{modulus} of $\eta$.
All that is relevant to us in the present paper is that there are at most $Q^{O(1)}$ horizontal characters 
$\eta:G\to \RR/\ZZ$ of modulus $|\eta| \leq Q$.

If $\eta:G \to \RR/\ZZ$ is a horizontal character and $g \in \mathrm{poly}(\ZZ,G_{\bullet})$, where $G_{\bullet}$ is a filtration of degree $d$, then $\eta \circ g: \ZZ \to \RR/\ZZ$ is a polynomial of degree at most $d$.
For an arbitrary polynomial $P: \ZZ \to \RR/\ZZ$ of degree at most $d$, we may define two sets of 
coefficients, $\alpha_0, \dots, \alpha_d$ and $\beta_0, \dots, \beta_d$, in $\RR/\ZZ$ via
$$
P(n) = \alpha_0 + \alpha_1 \binom{n}{1} + \dots + \alpha_d \binom{n}{d}
= \beta_d n^d + \dots + \beta_1 n + \beta_0.
$$
The {\em smoothness norm} (\cite[Definition 1.5]{GT-polyorbits}) of $P$ with respect to $N$ is then defined to be
$$
\|P\|_{C^{\infty}[N]}
= \sup_{1 \leq j \leq d} N^j \|\alpha_j\|_{\RR/\ZZ}
$$
and we have
$$
\|P\|_{C^{\infty}[N]} 
\ll_d \sup_{1 \leq j \leq d} N^j \|\beta_j\|_{\RR/\ZZ}
\qquad \text{and} \qquad
\sup_{1 \leq j \leq d} N^j \|q\beta_j\|_{\RR/\ZZ}
\ll  \|P\|_{C^{\infty}[N]}
$$
for some positive integer $q \ll_d 1$ by \cite[Lemma 3.2]{GT-polyorbits}.

Smoothness norms and horizontal characters allow one to characterise $\delta$-equidistributed polynomial sequences in the following sense:
\begin{lemma}[Green--Tao \cite{GT-polyorbits}, Theorem 2.9]
\label{lem:leibman}
Let $m_G$ and $d$ be non-negative integers, let $0 < \delta < 1/2$ and let
$N \geq 1$.
Suppose that $G/\Gamma$ is an $m_G$-dimensional nilmanifold together with a 
filtration $G_{\bullet}$ of degree $d$ and that $\mathcal{X}$ is a 
$\delta^{-1}$-rational Mal'cev basis adapted to $G_{\bullet}$. 
Suppose that $g \in \mathrm{poly}(\ZZ, G_{\bullet})$.
If $(g(n)\Gamma)_{n \leq N}$ is not $\delta$-equidistributed, then there 
exists a non-trivial horizontal character $\eta$ with 
$0 < |\eta| \ll \delta^{-O_{d, m_G}(1)}$ such that
$$\|\eta \circ g \|_{C^{\infty}[N]} \ll \delta^{-O_{d, m_G}(1)},$$
where $C_d$ is a sufficiently large constant depending only on $d$.
\end{lemma}

In order to pass between the notions of equidistribution and total equidistribution for polynomial sequences, we shall use the following lemma, which is \cite[Lemma 7.2]{Mat-multiplicative}.
\begin{lemma} 
\label{lem:equi/totally-equi}
 Let $N$ and $A$ be positive integers and let $\delta: \NN \to [0,1]$ be a 
 function that satisfies $\delta(x)^{-t} \ll_t x$ for all $t>0$.
 Suppose that $G$ has a $\delta(N)^{-1}$-rational Mal'cev basis adapted to 
 the filtration $G_{\bullet}$.
 Then there is  $1\leq B \ll_{d, \dim G} 1$ such that the following holds provided $N$ is sufficiently large.
 If $g \in \mathrm{poly}(\ZZ,G_{\bullet})$ is a polynomial sequence 
 such that $(g(n)\Gamma)_{n\leq N}$ is $\delta(N)^A$-equidistributed for some $A>B$, then
 $(g(n)\Gamma)_{n\leq N}$ is totally $\delta(N)^{A/B}$-equidistributed.
\end{lemma}

With these preparations in place, we now turn towards the main result of this section.

\begin{proposition}[Equidistribution in sparse families of linear subsequences] \label{prop:linear-subseqs}
Let $N$ be a large positive parameter, let $d \geq 0$, and let $K' > 0$ and $K> \max(2K',2)$. 
Suppose that $1 \leq y' \leq (\log N)^{K'}$ and $(\log N)^K < y \leq N^{\mu}$, where $\mu = \mu(d) \in (0,1)$ is sufficiently small depending on $d$.
Let $\delta: \RR_{>0} \to \RR_{>0}$ be a function of $N$ that satisfies 
$\log_5 N \ll \delta(N)^{-1}$ and $\delta(N)^{-B} \ll_B \log_2 N$ for all $B>0$.
Suppose that $(G/\Gamma, G_{\bullet})$ is a nilmanifold  
together with a filtration $G_{\bullet}$ of degree $d$ and a $\delta(N)^{-1}$-rational Mal'cev basis adapted to it.
Let $g \in \mathrm{poly}(\ZZ,G_{\bullet})$ be a polynomial sequence and suppose  
that the finite sequence $(g(n)\Gamma)_{n\leq N}$ is totally $\delta(N)^{E_1}$-equidistributed in $G/\Gamma$ for some $E_1 \geq 1$.
Then there is some $c_1 \in (0,1)$ depending on $d$ and $\dim G$ such that the following assertion holds for all integers
\begin{equation} \label{eq:M-range}
M \in [N^{1/2},N/y^{1/2}]
\end{equation}
provided that $K$ is sufficiently large depending on $d$, and that $c_1E_1 \geq 1$. 

Given any sequence $(A_m)_{m\in \NN}$ of integers satisfying $|A_m| \leq m$, 
write $g_m(n) := g(mn+A_m)$ and let $\mathcal{B}_{M}$ denote the set of integers 
$m \in [M,2M) \cap S([y',y])$ for which 
$$(g_m(n)\Gamma)_{n \leq N/m}$$
fails to be totally $\delta(N)^{c_1 E_1}$-equidistributed. 
Then
$$
\# \mathcal{B}_{M}
\ll \Psi(2M,[y',y]) \delta(N)^{c_1 E_1}.
$$
\end{proposition}
\begin{proof}
Let $M$ be a fixed integer in the range \eqref{eq:M-range} and let $c_1 > 0$ to be determined in the course of the proof.
Suppose that $E_1>1/c_1$.
It follows from Lemma \ref{lem:equi/totally-equi}  
that for every $m \in \mathcal{B}_{M}$, the sequence $(g_m(n)\Gamma)_{n\leq N/ m}$
fails to be $\delta(N)^{c_1E_1B}$-equidistributed on $G/\Gamma$ for some 
$1\leq B \ll_{d, \dim G} 1$.
By Lemma \ref{lem:leibman},  
there is a non-trivial horizontal character $\eta_m: G \to \RR/\ZZ$ of magnitude 
$|\eta_m| \ll {\delta(N)}^{-O_{d, \dim G}(c_1 E_1)}$ such that
\begin{align} \label{eq:smoothness-D}
\| \eta_m \circ g_m\|_{C^{\infty}[N/M]} \ll {\delta(N)}^{-O_{d, \dim G}(c_1 E_1)}.
\end{align}
For each non-trivial horizontal character $\eta: G \to \RR/\ZZ$ we define the 
set
$$
\mathcal{M}_{\eta} = 
\left\{ m \in \mathcal{B}_M : \eta_m = \eta \right\}.
$$
Note that this set is empty unless $|\eta| \ll {\delta(N)}^{-O_{d, \dim G}(c_1 E_1)}.$
Suppose that
$$
\# \mathcal{B}_M \geq \Psi(2M,[y',y]) \delta(N)^{c_1E_1}.
$$
Since there are only $M^{O(1)}$ horizontal characters of modulus bounded by $M$, it follows from the pigeon hole principle that there is some $\eta$ of modulus
$|\eta| \ll {\delta(N)}^{-O_{d, \dim G}(c_1 E_1)}$ such that
$$
\# \mathcal{M}_{\eta} \geq \Psi(2M,[y',y]) {\delta(N)}^{c_1 E_1 C}
$$
for some $C \asymp_{d, \dim G} 1$.
Suppose
$$
\eta \circ g(n) 
= \beta_d n^d + \dots \beta_1 n + \beta_0. 
$$
Then
$$
\eta \circ g_m(n) 
= \eta \circ g (mn + A_m)
= \alpha_d^{(m)} n^d + \dots + \alpha_1^{(m)} n + \alpha_0^{(m)},
$$
where
\begin{equation} \label{eq:alpha-beta}
\alpha_j^{(m)} = m^j \sum_{i=j}^d \binom{i}{j} A_m^{i-j} \beta_i, \qquad (0 \leq j \leq d). 
\end{equation}
The bound \eqref{eq:smoothness-D} on the smoothness norm asserts that
\begin{align*} 
\sup_{1 \leq j \leq d} 
\frac{N^j}{M^j} 
\| \alpha_j^{(m)} \|
\ll {\delta(N)}^{-O_{d, \dim G}(c_1 E_1)}, 
\end{align*}
which by downwards induction combined with \eqref{eq:alpha-beta} implies
\begin{align*} 
\sup_{1 \leq j \leq d} 
\frac{N^j}{M^j} 
\| \beta_j m^j  \|
\ll {\delta(N)}^{-O_{d, \dim G}(c_1 E_1)}.
\end{align*}
Hence,
\begin{align*} 
\| \beta_j m^j  \|
\ll {\delta(N)}^{-O_{d, \dim G}(c_1 E_1)} (M/N)^j , \qquad(1 \leq j \leq d), 
\end{align*}
for every $m \in \mathcal{M}_\eta$.

In view of the lower bound on $\#\mathcal{M}_\eta$, we seek to apply 
Theorem \ref{thm:strong-recurrence} with $k=j \leq d$ and with the cut-off parameter $N$ in the theorem replaced by $M$.
Observe that all assumptions in Theorem \ref{thm:strong-recurrence} about the relation between the parameters $N$, $y$, $y'$, as well as, those on the function $\delta$ and the size of $\eps$ are sufficiently flexible to allow for $N$ to be replaced by any $M$ with $\log M \asymp \log N$.
Note that $N/M > y^{1/2} \geq (\log N)^{K/2} \gg (\log M)^{K/2}$, which allows us to choose
$$
\eps = {\delta(N)}^{-O_{d, \dim G}(c_1 E_1)} (M/N)^j 
< {\delta(N)}^{-O_{d, \dim G}(c_1 E_1)} {y}^{-j/2} < {\delta(N)}^{-O_{d, \dim G}(c_1 E_1)} (\log M)^{-jK/2}
$$ and $C_1 = jK/2$ in the application of Theorem~\ref{thm:strong-recurrence}.
Our assumptions on $\delta(N)$ imply that $\eps < {\delta^*}/2$ 
for any $\delta^*$ of the form ${\delta(N)}^{c_1 E_1 C}$ with $C \asymp_{d, \dim G} 1$ as soon as $N$ is sufficiently large.
Hence, it follows from Theorem \ref{thm:strong-recurrence} that there exists an integer 
$1 \leq q_j \ll {\delta(N)}^{-O_{d, \dim G}(c_1 E_1)}$ such that
\begin{equation*} 
\| q_j \beta_j \|
\ll {\delta(N)}^{-O_{d, \dim G}(c_1 E_1)} (M/N)^{-j} M^{-j} 
= {\delta(N)}^{-O_{d, \dim G}(c_1 E_1)} N^{-j}.
\end{equation*}
Thus,
\begin{align}\label{eq:alpha_j-kappa_j-D}
\beta_j = \frac{a_j}{\kappa_j} + \tilde\beta_j,
\end{align}
where $\kappa_j \mid q_j $, $\gcd(a_j,\kappa_j)=1$
and 
$$0 \leq \tilde\beta_j \ll{\delta(N)}^{-O_{d, \dim G}(c_1 E_1)} N^{-j}.$$
Hence,
\begin{align}\label{eq:kappa_j-D}
\| \kappa_j \beta_j \|
\ll  {\delta(N)}^{-O_{d, \dim G}(c_1 E_1)} N^{-j}.
\end{align}
Let $\kappa=\lcm(\kappa_1, \dots, \kappa_d)$ and set 
$\tilde \eta = \kappa \eta$.
We proceed in a similar fashion as in \cite[\S3]{GT-nilmobius}: 
The above implies that
$$
\|\tilde \eta \circ g (n)\|_{\RR/\ZZ} \ll
{\delta(N)}^{-O_{d, \dim G}(c_1 E_1)} n/N, \qquad (n \leq N),
$$
which is small provided $n$ is not too large.
Indeed, if $N'= \delta(N)^{c_1 E_1 C''} N$ for some sufficiently large constant 
$C'' \geq 1$ depending only on $d$ and $\dim G$, and if $n \in \{1, \dots , N'\}$, 
then
$$
\|\tilde \eta \circ g (n)\|_{\RR/\ZZ} \leq 1/10.
$$
Let $\chi: \RR/\ZZ \to [-1,1]$ be a function of bounded Lipschitz norm that 
equals $1$ on $[-\frac{1}{10},\frac{1}{10}]$ and satisfies $\int_{\RR/\ZZ} 
\chi(t)\d t = 0$.
Then, by setting $F := \chi \circ \tilde \eta$, we obtain a Lipschitz function 
$F:G/\Gamma \to [-1,1]$ that satisfies $\int_{G/\Gamma} F = 0$ and 
$\|F\|_{\mathrm{Lip}} \ll {\delta(N)}^{-O_{d, \dim G}(c_1 E_1)}$.
By choosing $c_1$ sufficiently small depending on $d$ and $\dim G$ we may ensure that
$$\|F\|_{\mathrm{Lip}} < \delta(N)^{-E_1}$$
and, moreover, that
$$
N'> \delta(N)^{E_1}N.
$$
This choice of $N'$, $F$ and $c_1$ implies that, for all sufficiently large values of $N$, we have
$$
\Big| \frac{1}{N'} \sum_{1 \leq n \leq N'} F(g(n)\Gamma) \Big|
= 1 > \delta(N)^{E_1}  \|F\|_{\mathrm{Lip}},
$$
which contradicts our assumption that
$(g(n)\Gamma)_{n\leq N}$ is totally $\delta(N)^{E_1}$-equidistributed.
This completes the proof of the proposition.
\end{proof}

\section{Non-correlation with equidistributed nilsequences} \label{sec:equid-nilsequences}

In this section we prove Proposition \ref{prop:equid-non-corr}, that is to say we show that 
the function $m \mapsto \1_{S([y',y])}(\W m+A')$ and its weighted version $m \mapsto h_{[y',y]}^{(W,A)}(qm+a)$ are orthogonal to highly equidistributed nilsequences for $w(N)$-smooth values of $q$ and $0\leq a <q$.

Both of these functions are $W$-tricked versions of {\em sparse multiplicative functions}.
Building on a Montgomery--Vaughan-type decomposition, which allows one to replace twisted sums of multiplicative functions by bilinear expressions, it was proved in \cite{Mat-multiplicative} that $W$-tricked and centralised versions of {\em `dense' multiplicative functions} are orthogonal to nilsequences.
We shall show that this approach can be extended to the sparse setting we are looking at here in order to establish
Proposition \ref{prop:equid-non-corr}.

\subsection{Removing the weight}
We will prove Proposition \ref{prop:equid-non-corr} by a sequence of reductions. 
Recall that
$$
h_{[y',y]}^{(W,A)}(m) 
= \frac{\phi(W)}{W}h_{[y',y]}(Wm+A) 
= \frac{\phi(W)}{W} \frac{N^{\alpha} (Wm+A)^{1-\alpha}}{\alpha \Psi(N,[y',y])} \1_{S([y',y])}(Wm+A),
$$
where $\gcd(A,W)=1$.
Our first step is to remove the weight and reduce the correlation estimate for $h_{[y',y]}^{(W,A)}(qm+a)$ to one that only involves the characteristic function $\1_{S([y',y])}(\W m+A')$, where $\W=Wq$ and $A'=Wa+A$.

\begin{lemma}[Removing the weight]
Let $\W= Wq$ and $A'= Wa+A$ and define for any $x \in \NN$ the quantity
$$T(x) := \sum_{N/\W< m \leq x} \1_{S([y',y])}(\W m + A') F(g(m)\Gamma).$$
Then the conclusion \eqref{eq:prop-bound} of Proposition \ref{prop:equid-non-corr} holds provided that
\begin{equation} \label{eq:T-bound}
 T(x) \ll_{d, \dim G, \|F\|_{\mathrm{Lip}}, E_1}  \delta(N)^{E_1} \frac{\Psi(N,[y',y])}{\phi(\W)}
\end{equation}
for all $x \in (N/\W,2N/\W]$. 
\end{lemma}

\begin{proof}
To start with, recall that 
$$
\phi(\W)^{-1} = (\phi(W)q)^{-1} \prod_{p \mid q, p>w(N)} (1-p^{-1})^{-1} \asymp (\phi(W)q)^{-1}
$$
by \eqref{eq:phi-and-W-trick} and note that
$$
\frac{N^{\alpha} (\W m+A')^{1-\alpha}}{\alpha N} \asymp 1
$$
for $N/\W < m < (N+N_1)/\W$.
With this in mind, partial summation shows that
\begin{align*}
 &\frac{\W}{N}\frac{\phi(W)}{W}\sum_{N/\W < m < (N+N_1)/\W} h_{[y',y]}(\W m + A') F(g(m)\Gamma)\\
 &=\frac{1}{\Psi(N,[y',y])}\frac{\W \phi(W)}{W} 
 \sum_{N/\W < m < (N+N_1)/\W} 
 \frac{N^{\alpha} (\W m+A')^{1-\alpha}}{\alpha N}
 \1_{S([y',y])}(\W m + A') F(g(m)\Gamma)\\
 &\ll \frac{|T((N+N_1)/\W)| \phi(\W)}{\Psi(N,[y',y])} \\
 & \qquad
 + \frac{\phi(\W)}{\Psi(N,[y',y])} \sum_{N/ \W < m < (N+N_1)/\W } \frac{|T(m)|}{\alpha} 
 \left|  \left(\frac{\W (m+1)+A'}{N} \right)^{1-\alpha} - 
 \left(\frac{\W m+A'}{N} \right)^{1-\alpha} \right| \\
 &\ll \frac{\phi(\W)}{\Psi(N,[y',y])} \Bigg\{ \left|T\left(\frac{N+N_1}{\W}\right) \right|
 +  
 \sum_{\substack{N/ \W < m \\< (N+N_1)/\W }} \frac{|T(m)|(1-\alpha)}{\alpha N} 
  \left| \int_{\W (m+1)+A'}^{\W m+A'} (N/t)^{\alpha} ~dt \right| \Bigg\}  \\
 &\ll \frac{\phi(\W)}{\Psi(N,[y',y])} \Bigg\{ \left|T\left(\frac{N+N_1}{\W}\right) \right|
 + \frac{\W}{N} \sum_{N/ \W < m < (N+N_1)/\W} |T(m)| \Bigg\} \\
 &\ll_{d, \dim G, \|F\|_{\mathrm{Lip}}, E_1}  \delta(N)^{E_1},
\end{align*}
provided that the stated bounds on $T(x)$ hold.
\end{proof}

\subsection{Montgomery--Vaughan-type reduction to a Type II estimate} \label{subsec:MV}

Our proof strategy for establishing \eqref{eq:T-bound} is to employ a Montgomery--Vaughan-type decomposition which replaces the the given one-parameter correlation by a bilinear sum and eventually allows us to reduce \eqref{eq:T-bound} to a non-correlation estimate of the von Mangoldt function with nilsequences.
Since the parameter $y$ can be very small and since the correlation involving the von Mangoldt function will have length $y^c$ in the end, this approach requires a careful choice of the cut-off parameter $w(N)$ in the $W$-trick and the parameter $\delta(N)$ that controls the level of equidistribution in the nilsequence to ensure that uniform non-correlation estimates for the von Mangoldt function, valid over the whole range produced by the Montgomery--Vaughan decomposition, can be deduced.

\begin{proposition}[Reduction to a bilinear correlation] \label{prop-bilinear}
The bound \eqref{eq:T-bound} holds for all $x \in (N/\W, 2N/\W]$ provided that 
for any sequence $(A_n)_{n \in \NN}$ with $|A_n| \leq n$,
for any sequence $(A'_n)_{n \in \NN}$ of integers $0< A'_n < \widetilde W$ coprime to $\widetilde W$ and for any 
sufficiently large $E_0 >1$, we have
\begin{align*}
&\frac{1}{\log N}  
\Big| 
\sum_{\substack{1 < n \leq N+N_1}} 
\sum_{\substack{\widetilde{W} m+A'_n \in [y',y]\\ N < n(\widetilde{W} m+A'_n) \leq N+N_1}}
\1_{S([y',y])}(n) \Lambda(\widetilde{W} m+A'_n) F(g(mn + A_n)\Gamma) 
\Big| \\
&\ll \delta(N)^{E_1} \frac{\Psi(N,[y',y]))}{\phi(W)q}
\end{align*}
for all $N_1 \in(0,N]$ and all $g \in \mathrm{poly}(\ZZ, G_{\bullet})$ such that $(g(n)\Gamma)_{n\leq N}$
is totally $\delta(N)^{E_0}$-equidistributed. 
The implied constant is allowed to depend on $d$, $\dim G$, $\|F\|_{\mathrm{Lip}}$ and $E_1$.
\end{proposition}

Proposition \ref{prop-bilinear} will be deduced from the following simple lemma, which is inspired by a bound from
Montgomery and Vaughan \cite{mv77}, but does not involve a second moment.
\begin{lemma} \label{lem:MV-decomposition}
Suppose that $f : \NN \to \CC$, let $N$ be a positive parameter and let 
$\mathcal{S}(N) \subset [N/2,N] \cap \NN $ denote a set.
Then
\begin{align*}
\sum_{n \in \mathcal{S}(N)} f(n) 
\leq \frac{1}{\log N} \sum_{n \in \mathcal{S}(N)} |f(n)|  
 + \frac{1}{\log N} \Big|\sum_{nm  \in \mathcal{S}(N)} f(nm)\Lambda(m) \Big|.
\end{align*} 
\end{lemma}
\begin{proof}
This is an immediate consequence of the bound
\begin{align*}
\sum_{N/2 \leq n \leq N} \log(N/n)f(n) \leq \log 2 \sum_{N/2 \leq n \leq N} |f(n)|.
\end{align*}
\end{proof}

Note that we have
$$
T(x) =
\sum_{N/\W < m \leq x} \1_{S([y',y])}(\W m+A') F(g(m)\Gamma)
= \sum_{\substack{N<n \leq x\W \\n\equiv A' \Mod{\W}}} \1_{S([y',y])}(n) F\Big(g\Big(\frac{n-A'}{\widetilde W} \Big)\Gamma\Big).
$$
Writing $N_1:= x\W-N \in (0,N]$ and applying Lemma \ref{lem:MV-decomposition} to this expression yields
\begin{align} \label{eq:T(x)-bd}
&T(x) \\
\nonumber
&\ll \frac{1}{\log N }  
\Big| \sum_{\substack{N < mn \leq N + N_1\\ mn \equiv A' \Mod{\W}}} 
     \1_{S([y',y])}(mn) \Lambda(m) F\Big(g\Big(\frac{nm-A'}{\widetilde W}\Big)\Gamma\Big)  \Big| 
  + O\Big(\frac{\Psi(N,[y',y];\W,A')}{\log N}\Big).
\end{align}  
If $A'_n \in \{0,\widetilde{W}-1\}$ is such that $n A'_n \equiv A' \Mod{\widetilde{W}}$, we may expand the congruence condition on $m$ and replace $m$ by $\W m + A'_n$.
With this choice of $A'_n$ we have 
$$\frac{n (\W m + A'_n) - A'}{\W} = n m +A_n$$
for some integer $A_n$ with $|A_n| \leq n$.
Hence the first term in the bound on $T(x)$ may be rewritten as
\begin{align*}
\mathcal{M}(x):=
\frac{1}{\log N }  
\Big| \sum_{1\leq n\leq N+N_1} \sum_{\substack{N/n < \widetilde{W}m+A'_n \\ \leq (N + N_1)/n }} 
     \1_{S([y',y])}(n (\widetilde{W}m+A'_n)) \Lambda(\widetilde{W}m+A'_n) F(g(nm+A_n)\Gamma)  \Big| 
\end{align*}
The error term in \eqref{eq:T(x)-bd} is negligible in view of Theorem \ref{thm:equid-in-APs}, Lemma \ref{lem:BrTen-2.1} and \eqref{eq:Hil-Ten-Thm1}, provided that $\log^{-1} N \ll_{E_1} \delta(N)^{E_1}$ which holds for the choice of $\delta(N)$ for Proposition \ref{prop:equid-non-corr}.
Thus \eqref{eq:T-bound} holds provided we can show that, under the assumptions of Proposition
\ref{prop:equid-non-corr}, the bound
\begin{equation} \label{eq:aux-condition}
\mathcal{M}(x)
\ll_{d, \|F\|_{\mathrm{Lip}}} \delta(N)^{E_1} \frac{\Psi(N,[y',y]))}{\phi(\W)}
\end{equation}
holds for all $N_1 \in (0,N]$. 
To complete the proof of Proposition \ref{prop-bilinear}, it remains to show that we can replace the function 
$\Lambda(m) \1_{S([y',y])}(m)$ by $\Lambda(m) \1_{[y',y]}(m)$ in the condition \eqref{eq:aux-condition}.
This is the content of the following lemma.

\begin{lemma}[Removing large prime powers]
We have
$$
\frac{1}{\log N }  
\sum_{\substack{N < nm \leq N + N_1\\ nm \equiv A' \Mod{\W} \\ \Omega(m)\geq 2 \\ m > y}} 
     \1_{S([y',y])}(nm) \Lambda(m) 
\ll_B \delta(N)^{B} \frac{\Psi(N,[y',y]))}{\phi(\W)}     
 $$
\end{lemma}

\begin{proof}
Theorem \ref{thm:equid-in-APs} and Lemma \ref{lem:BrTen-2.4.i-ii} (i) show that the left hand side above is bounded by
\begin{align*}
&\ll \frac{\log y}{\log N} \sum_{p \leq y} \sum_{k \geq \max(2,(\log y)/\log p)}
 \max_{(A'',\W)=1} \Psi\left(\frac{N}{p^k},[y',y];\W,A''\right) \\
&\ll \frac{\log y}{\log N} \frac{\Psi(N,[y',y])}{\phi(\W)} \sum_{p \leq y} \sum_{k \geq \max(2,(\log y)/\log p)}
 p^{-\alpha k}
 +O\Big(\frac{\log y}{\log N} \frac{y^2}{\log y}\Big),
\end{align*}
where the error term trivially bounds the contribution of those choices of $p^k$ for which $N/y < p^k < (N+N_1)/y'$, i.e.\ to which Lemma \ref{lem:BrTen-2.4.i-ii} (i) does not apply. This error term is acceptable since $y< N^{\eta}$ for some sufficiently small $\eta>0$.
The sum over primes in the main term satisfies
\begin{align*}
\sum_{p \leq y} \sum_{k \geq \max(2,(\log y)/\log p)}
 p^{-\alpha k} 
&\ll \sum_{p \leq y^{1/2}} y^{-\alpha} + \sum_{\sqrt{y}< p \leq y} p^{-2\alpha} 
\ll \frac{y^{1/2-\alpha}}{\log y} + (y^{1/2})^{-2\alpha + 1} \\
&\ll y^{-1/2} u \log u 
\ll_{\eps} \frac{\log N}{\log y} (\log N)^{-K/2 + \eps}
\ll_B \delta(N)^B \frac{\log N}{\log y}
\end{align*}
for all $B>0$, since $\alpha > 1/2$ and $\delta(N) = \exp(-C_0 \sqrt{\log_4 N})$.
\end{proof}

\subsection{Explicit bounds for the correlation between $\Lambda$ and nilsequences}
In view of the inner sum in the bilinear expression in Proposition \ref{prop-bilinear}, we may reduce the problem of establishing that bound to a non-correlation estimate for the von Mangoldt function with equidistributed nilsequences.
For this purpose, we require explicit bounds for correlations of length $\xi$ of the $W$-tricked von Mangoldt function with nilsequences.
A specific requirement on these bounds is that they work with the same $W$-trick (determined by $w(N)$ and independent of $\xi$) and are uniform over a large range of $\xi$.
Taking into account that $y$ can be as small as $(\log N)^{K}$, we shall prove a result that is applicable to 
$\xi \in [\log N, N]$ if $w(N)$ is suitably chosen (cf.\ Remark \ref{rem:choice-of-parameters} (ii) below).

\begin{theorem}
 \label{thm:explicit-Lambda}
Let $N_0>2$ be a large constant and $N>N_0$ a parameter. 
Let $x_0, w, \delta, \kappa: \NN_{>N_0} \to \RR_{>0}$ be functions that are defined for all sufficiently large integers and satisfy the relations $\kappa(N) \geq 1$, $1 \leq w(N) \leq \frac{1}{2} \log_2 x_0(N)$ and 
$\kappa(N)^2 w(N)^{-1/\kappa(N)} \ll_B \delta(N)^B$ for all $N>N_0$ and all $B \geq 1$.
Suppose that $w(N) \to \infty$ as $N \to \infty$.

Suppose that $G/\Gamma$ is a filtered nilmanifold of dimension $\dim G \leq \kappa(N)$ and complexity at most 
$\delta(N)^{-1}$, let $d\geq 0$ denote its degree, and let $g \in \mathrm{poly}(\ZZ,G_{\bullet})$.
Let $1\leq q \leq (\log x_0(N))^E$, $1 \leq E \ll 1$, be an integer, let $W = P(w(N))$ and write $\W = Wq$.
If $0<A'<\W$ is an integer such that $(\W,A')=1$, then
$$
\frac{\W}{\xi}
\sum_{n \leq \xi/\W} 
\bigg(\frac{\phi(\W)}{\W} \Lambda(\W n+A') - 1 \bigg) F(g(n)\Gamma)
\ll_{d,B} (1+ \|F\|_{\mathrm{Lip}})\delta(N)^{B} 
$$
for all $\xi \in [x_0(N),N]$, all $B \geq 1$,  
and for all Lipschitz functions $F: G/\Gamma \to \CC$. 
\end{theorem}

\begin{rems} \label{rem:choice-of-parameters}
\hspace{2em}
\begin{itemize}
 \item[$(i)$] Note that $\xi/\W = \xi^{1+o(1)}$ by the definition of $W$ and the conditions on $w(N)$, $\delta(N)$ and $q$. More precisely, we have $\xi \geq x_0(N)$, $\W=Wq$, $W = \exp(w(N) + o(1)) = (\log x_0(N))^{1/2 + o(1)}$, and 
 $q \leq (\log x_0(N))^{O(1)}$.

\item[$(ii)$] The choices $x_0(N) := \log N$, $w(N) := \frac{1}{2}\log_3 N$, $\kappa(N) = (\log_5 N)^{C}$ 
with $1 \leq C = O(1)$ and $\delta(N) = \exp(-C_0 \sqrt{\log_4 N})$ with $1 \leq C_0 \leq (\log_4 N)^{1/4}$ and
$N>N_0$ for sufficiently large $N_0$ are permissible in the theorem above.

\item[$(iii)$] Observe that the parameter choices in $(ii)$ are consistent with the assumptions of 
Proposition \ref{prop:linear-subseqs} as well as the assumptions of Theorem \ref{thm:strong-recurrence}. 

\end{itemize}

\end{rems}

\begin{proof}
This proof is a small modification of the proof of \cite[Lemma 9.5]{Mat-multiplicative} 
given in the appendix to that paper. 
The starting point of that proof is the following decomposition of the von Mangoldt function as 
$\Lambda = \Lambda^{\flat} + \Lambda^{\sharp}$.
Let $\id_{\RR}(x)=x$ denote the identity function and let $\chi^{\flat} + \chi^{\sharp} = \id_{\RR}$ be a smooth 
decomposition with the property that $\supp (\chi^{\sharp}) \subset (-1,1)$ and 
$\supp (\chi^{\flat}) \subset \RR \setminus [-1/2,1/2]$.
Then, for any $\gamma \in (0,1)$,
$$
\frac{\phi(\W)}{\W}\Lambda(\W n + A') - 1
= \frac{\phi(\W)}{\W}\Lambda^{\flat}(\W n + A') +
  \Big(\frac{\phi(\W)}{\W}\Lambda^{\sharp}(\W n + A') - 1 \Big),
$$
where, cf.\ \cite[(12.2)]{GT-linearprimes}, 
$$
\Lambda^{\sharp}(n) 
= - \log \xi^{\gamma} 
 \sum_{d \mid n} \mu(d) \chi^{\sharp} 
 \Big( \frac{\log d}{\log \xi^{\gamma}} \Big) 
  \qquad (|t| \geq 1 \Rightarrow \chi^{\sharp} (t) = 0)
$$
is a truncated divisor sum, where
$$
\Lambda^{\flat}(n)
= - \log \xi^{\gamma} 
 \sum_{d\mid n} \mu(d) \chi^{\flat} 
 \Big( \frac{\log d}{\log \xi^{\gamma}} \Big) 
  \qquad (|t| \leq 1/2 \Rightarrow \chi^{\flat} (t) = 0)
$$
is an average of $\mu (d)$ running over large divisors of $n$.

It follows as in \cite[\S12]{GT-linearprimes} and \cite[Appendix A]{Mat-multiplicative} from the orthogonality of the M\"obius function with nilsequences that
$$
\frac{\W}{\xi}\sum_{n\leq \xi/\W} 
\frac{\phi(\W)}{\W}\Lambda^{\flat}(\W n+A')
F(g(n)\Gamma) 
\ll_{\|F\|_{\mathrm{Lip}}, G/\Gamma, B}
(\log \xi)^{-B}, \qquad (B>0).
$$
Here, it is important that $\W \leq (\log x_0(N))^{O(1)} \leq (\log \xi)^{O(1)}$.

Concerning the contribution from $\Lambda^{\sharp}$, define $\lambda^{\sharp}:\NN \to \RR$,
$$\lambda^{\sharp}(n) 
:= \frac{\phi(\W)}{\W}\Lambda^{\sharp}(\W n+A') - 1.$$
Since $\X$ is $\delta(N)^{-1}$-rational, \cite[Lemmas A.2 and A.3]{Mat-multiplicative} imply that the following bound holds with $m = 2^d \dim G$ and for every $\eps>0$:
\begin{align} \label{eq:U_k-U_k-star}
\nonumber
\frac{\W}{\xi}
\sum_{n \leq \xi/\W} \lambda^{\sharp}(n)
F(g(n)\Gamma)
&\ll
\eps(1 + \|F\|_{\mathrm{Lip}})
+ 
\left\| \lambda^{\sharp} \right\|_{U^{d+1}[\xi/\W]} 
(m/\eps)^{2m} \delta(N)^{-O(m)}\\
&\ll
\eps(1 + \|F\|_{\mathrm{Lip}})
+ 
\frac{m^{2m} \left\| \lambda^{\sharp} \right\|_{U^{d+1}[\xi/\W]}}
{(\eps^2 \delta(N)^{O(1)})^{m}}.
\end{align}
We shall show below that
$\left\| \lambda^{\sharp} \right\|_{U^{d+1}[\xi/\W]}  \ll_d w(N)^{-1/2^{d+1}}$. 
Choosing $\eps = \delta(N)^{B}$  
and recalling the assumption that 
$\kappa(N)^2 w(N)^{-1/\kappa(N)} \ll_B \delta(N)^B$ and that $m =2^d \dim G \leq 2^d \kappa(N)$, 
it follows that the above is bounded by
$$
\ll_{d,B}
(1 + \|F\|_{\mathrm{Lip}}) \delta(N)^{B}
+ \delta(N)^{O_d(B\kappa(N))}
\ll_{d,B} (1+ \|F\|_{\mathrm{Lip}}) \delta(N)^{B},
$$
as required.

The uniformity norm $\left\| \lambda^{\sharp} \right\|_{U^{d+1}[\xi/\W]}$, can, as in \cite[Appendix A]{Mat-multiplicative}, be estimated using \cite[Theorem D.3]{GT-linearprimes} since
$w(N) \leq \frac{1}{2} \log_2 x_0(N)$ is sufficiently small 
(c.f.\ the `important convention' in \cite[\S 5]{GT-linearprimes}).
In our case, the system of forms takes the shape
$$
\Psi_{\mathcal{B}}(n,\mathbf h) 
= \Big(\widetilde W(n+\boldsymbol\omega \cdot \mathbf h) + A'\Big)_{\omega \in \mathcal{B}}~,
\quad (n,\mathbf h) \in \ZZ \times \ZZ^{d+1},
$$
where $\mathcal{B} \subset \{0,1\}^{d+1}$ is any non-empty subset.
The corresponding set $\mathcal{P_{\Psi_{\mathcal{B}}}}$ of exceptional primes consists of those primes dividing $\W=Wq$,
i.e.\ $|\mathcal{P_{\Psi_{\mathcal{B}}}}| = \pi(w(N)) \ll \log_2 \xi / \log w(N)$.
We further have 
$$
\prod_{p \not\in \mathcal{P}_{\Psi_{\mathcal{B}}}}
\beta_{p}^{(\mathcal{B})} 
= 1 + O_d\left(\frac{1}{w(N)}\right)
\qquad
\text{and}
\qquad
\prod_{p \in \mathcal{P}_{\Psi}}\beta_{p}^{(\mathcal{B})} 
= \left(\frac{\W}{\phi(\W)}\right)^{|\mathcal{B}|}.
$$
Provided the constant $\gamma$ is chosen sufficiently small, it follows from \cite[Theorem D.3]{GT-linearprimes} applied with
$K_z = \{(n,\mathbf{h}): 0 < n + \boldsymbol \omega \cdot \mathbf h \leq z
\text{ for all } \boldsymbol \omega \in \{0,1\}^{d+1} \}$, 
where $z=\xi/\W$,
that
\begin{align*}
\left\|\lambda^{\sharp}\right\|_{U^{d+1}[z]}^{2^{d+1}} 
&= \frac{\vol(K_z)}{z^{d+2}}
\sum_{\mathcal{B} \subseteq \{0,1\}^{d+1} } 
(-1)^{|\mathcal{B}|}
\prod_{p \not\in P_{\Psi_{\mathcal{B}}}} \beta_{p}^{(\mathcal{B})}
+ O_{d}\Big((\log \xi^{\gamma})^{-1/20} \exp(O_d(|\mathcal{P_{\Psi_{\mathcal{B}}}}|)\Big)\\
&\ll_{d} 
\frac{\vol(K_z)}{z^{d+2}}
\frac{1}{w(N)}
+ \exp\left(-\frac{\log_2 \xi}{20}  + O_d((\log_2 \xi)/\log w(N))\right)\\
&\ll_{d} 
\frac{1}{w(N)}.
\end{align*}
\end{proof}

As it provides a different approach which could prove useful for future generalisations, we include a second, alternative, proof for the special case of Theorem \ref{thm:explicit-Lambda} in which $q=1$. 
This proof is based on one of the main results of Tao and Ter\"av\"ainen \cite{TT}. 
For this proof to work, $w(n)$ needs to be redefined as $w(n)=c' \log_3 N$ for some sufficiently small constant $c'>0$.

\begin{proof}[Alternative proof for special case of Theorem \ref{thm:explicit-Lambda}]
Assume for simplicity the choice of parameters from Remark \ref{rem:choice-of-parameters} (ii) except for
$w(N) = \frac{1}{2}\log_3 N$, which we replace by $w(n)=c' \log_3 N$ for some small $c'>0$.
Further, let $q=1$ so that $\W=W$ and $A'=A$.

Define $\lambda_{W,A} = W^{-1}\phi(W) \Lambda(W\cdot + A) - 1$. 
By the proof of \cite[Corollary 1.5]{TT} (see the end of Section 8 in that paper), we have
\begin{equation} \label{eq:TT-Uk-norm}
\|\lambda_{W,A}\|_{U^k[(\xi-A)/W]} \ll (\log_2 \xi)^{-c} 
\end{equation}
provided that the constant $c'>0$ in the definition of $w(N)$ is sufficiently small to ensure that the estimate
\cite[(8.19)]{TT} holds with $W = P(w(N))$ and $N=x_0(N)$.

Proceeding as in the proof of Theorem \ref{thm:explicit-Lambda} above, we use \cite[Lemmas A.2 and A.3]{Mat-multiplicative} which show that the following bound holds with $m = 2^d \dim G$ and for every $\eps>0$:
\begin{align*}
\frac{W}{\xi}
\sum_{n \leq (\xi-A)/W} \lambda_{W,A}(n)
F(g(n)\Gamma)
&\ll
\eps(1 + \|F\|_{\mathrm{Lip}})
+ 
\left\| \lambda_{W,A} \right\|_{U^{d+1}[(\xi-A)/W]} 
(m/\eps)^{2m} \delta(N)^{-O(m)}.
\end{align*}
Choosing $\eps = \delta(N)^B$ and recalling that $\log_2 \xi \gg \log_2 x_0(N) \gg \log_3 N$, the above is seen to be bounded by 
$$
\ll_{d} (1 + \|F\|_{\mathrm{Lip}}) \delta(N)^B 
 + \frac{\kappa(N)^2 (\log_2 \xi)^{-c}}{\delta(N)^{O_d(B\kappa(N))}}
 \ll_{d} (1 + \|F\|_{\mathrm{Lip}}) \delta(N)^B
$$
where the constants $c$ and the final implied constant depend on the exponent in the bound~\eqref{eq:TT-Uk-norm} from \cite[Section 8]{TT} as well as on the degree $d$ of the filtration.
\end{proof}

\subsection{Conclusion of the proof of Proposition \ref{prop:equid-non-corr}}

In view of the reductions carried out in the previous subsections it remains to show that the condition of Proposition~\ref{prop-bilinear} is in fact valid in order to complete the proof of the equidistributed non-correlation estimate stated in Proposition \ref{prop:equid-non-corr}.
Hence, the proof of Proposition \ref{prop:equid-non-corr} is complete once we have established the following lemma.

\begin{lemma}
Let $E_1\geq1$. Under the assumptions of Proposition \ref{prop:equid-non-corr} and with the notation $\W = Wq$, we have 
\begin{align*}
\frac{1}{\log N}  
\Big| 
\sum_{\substack{n \leq N}}
\sum_{\substack{\W m + A'_n \in \\ [y',\min(y, N/n)]}}
 \1_{S([y',y])}(n) \Lambda(\W m + A'_n) 
F(g(mn + A_n)\Gamma) 
\Big|
\ll \delta(N)^{E_1} \frac{\Psi(N,[y',y]))}{\phi(\W)}
\end{align*}
for all sequences $(A_n)_{n \in \NN}$, $(A'_n)_{n \in \NN}$ of integers such that $|A_n|\leq n$ and 
$\gcd(A'_n, \W) = 1$, and 
for all $g \in \mathrm{poly}(\ZZ,G_{\bullet})$ such that the finite sequence $(g(n)\Gamma)_{n \leq N}$ is totally 
$\delta(N)^{E_0}$-equi{-}distributed for some sufficiently large $E_0>1$.
The implied constant may depend on the degree $d$ of $G_{\bullet}$, $\dim G$, $\|F\|_{\mathrm{Lip}}$ and $E_1$.
\end{lemma}

\begin{proof}
Splitting the summation range of the outer sum into three intervals and abbreviating
$$
G(m,n) = 
 \1_{S([y',y])}(n) \Lambda(\W m + A'_n) 
F(g(mn + A_n)\Gamma),
$$
we obtain
\begin{align} \label{eq:aux4}
\sum_{\substack{n \leq N}}
\sum_{\substack{\W m + A'_n \in \\ [y',\min(y, N/n)]}} G(m,n) 
&= 
\sum_{n < \delta^{2E_1} N/y~} 
\sum_{\substack{y' \leq \W m + A'_n \leq y}}
G(m,n)\\
\nonumber
&\qquad+
\sum_{\delta^{2E_1} N/y < n \leq N/y^{2/3}~} 
\sum_{\substack{y' \leq \W m + A'_n \leq \min (N/n,y)}}
G(m,n) \\
\nonumber
&\qquad+\sum_{0 \leq k \leq K~}
\sum_{\substack{N/y^{2/3} < n \\ \leq N/(y' 2^k)}}
\sum_{\substack{\W m + A'_n \in \\ ~~[y'2^k, \min(y'2^{k+1}, N/n)] }}
G(m,n),
\end{align}
where $K = \lceil \log (y^{2/3}/y') / \log 2 \rceil$ and in particular $y'2^K \asymp y^{2/3}$.

The decomposition above has been chosen in such a way that the contributions from the initial and final segment are negligible while in the middle segment the summation range of the inner sum is guaranteed to be sufficiently long and 
$\min(y,N/n)$ is not too small compared to $N/n$.
More precisely, by Lemma \ref{lem:BrTen-2.4.i-ii} (i), the final term is trivially bounded above by
\begin{align*}
\sum_{0\leq k \leq K} \frac{y'2^k}{\phi(\W)} \Psi\Big(\frac{N}{y'2^k},[y',y]\Big)
& \ll \frac{\Psi(N,[y',y])}{\phi(\W)}\sum_{0\leq k \leq K} (y'2^k)^{(1-\alpha)} 
\ll \frac{\Psi(N,[y',y])}{\phi(\W)} (y'2^K)^{(1-\alpha)} \\
&\ll y^{2(1-\alpha)/3} \frac{\Psi(N,[y',y])}{\phi(\W)} 
\ll (\log N)^{2/3} \frac{\Psi(N,[y',y])}{\phi(\W)},
\end{align*}
which is acceptable in view of the $(\log N)^{-1}$ factor in the statement.

Similarly, the first term is trivially bounded above by
\begin{align*}
\frac{y}{\phi(\W)} \Psi\Big(\delta(N)^{2E_1} \frac{N}{y},[y',y]\Big)
\ll \delta(N)^{2 E_1 \alpha} y^{1-\alpha} \frac{\Psi(N,[y',y])}{\phi(\W)} 
\ll (\log N) \delta(N)^{E_1} \frac{\Psi(N,[y',y])}{\phi(\W)},
\end{align*}
since $\alpha > 1/2$, which is also acceptable.

It remains to analyse the middle range and we start by observing that the summation range of the inner sum over 
$\W m + A'_n$ is always sufficiently long, that is of length $\gg \delta(N)^{E_1} m_1$ if $m_1$ denotes the upper endpoint of the range.
In fact, $m_1 = \min(y,N/n) \geq y^{2/3} > y^{1/2} > y'$, which shows that 
$|\min(y,N/n) - y'| \sim \min(y,N/n)$ and moreover $y' = o(\delta(N)^{B} \min(y,N/n)/\phi(\widetilde{W}))$ for any $B \asymp 1$.

In order to apply Proposition \ref{prop:linear-subseqs} on the equidistribution of linear subsequences of an equidistributed nilsequence, we dyadically decompose the sum over $n$ in the middle range of \eqref{eq:aux4}
into $O(\log N)$ intervals $(M_j,2M_j]$, where
$\delta(N)^{2E_1} N/y \leq M_j < N/y^{2/3}$, and one potentially shorter interval.
Note that $N^{1/2} < N y^{-2} < N \delta(N)^{2E_1} y^{-1} \leq M_j$, assuming that $\eta < 1/4$ (so that
$y < N^{1/4}$) and that $N$ is sufficiently large.
For each $M_j$ as above consider the sum
\begin{align} \label{eq:aux3}
 \sum_{\substack{n \sim M_j \\ n \leq N/y^{2/3}}} 
 \1_{S([y',y])}(n)
 \sum_{\substack{ \W m + A'_n \in \\ [y', \min (y,N/n)]}}
 \Lambda(\W m + A'_n) F(g(mn + A_n)\Gamma)
\end{align}
Let $E_2>1$ be a parameter to be chosen later.
Then, by Proposition \ref{prop:linear-subseqs}, 
all but $O(\delta(N)^{E_2} \Psi(M_j,[y',y]))$ of the finite sequences 
$$ (g_n(m)\Gamma)_{m \leq N/n}, \qquad n \in (M_j,2M_j] \cap S([y',y]),$$
where $g_n(m):= g(mn + A_n)$, are totally $\delta(N)^{E_2}$-equidistributed in $G/\Gamma$ provided that the parameter $E_0>1$ from our assumptions is sufficiently large.
We bound the contribution to \eqref{eq:aux3} from all exceptional values of $n \in (M_j,2M_j] \cap S([y',y])$ 
trivially by 
$$
O(\delta(N)^{E_2} \Psi(M_j,[y',y])) \frac{\min(y, N/M_j)}{\phi(\W)}.
$$
For all other values of $n \in (M_j,2M_j] \in S([y',y])$, we seek to apply Theorem~\ref{thm:explicit-Lambda}
with the choice of parameters given in Remark \ref{rem:choice-of-parameters} (ii).
The lower bound $M_j \geq \delta(N)^{2E_1} N/y$ implies that
$y \geq \delta^{2E_1} N/M_j \geq \delta^{E_2/2} N/n$ for all $n \in (M_j,2M_j]$ provided that $E_2> 4 E_1$.
Thus $\min(N/n,y) \geq \delta^{E_2/2} N/n$ and the sequence 
$$(g_n(m)\Gamma)_{m \leq \min(N/n,y)}$$
is totally $\delta^{E_2/2}$-equidistributed in $G/\Gamma$ whenever $n$ is non-exceptional, and 
$\xi = \min (y, N/n)$ satisfies $\xi \in [\log N, N]$.
Finally, the upper bound on $y'$ implies that for any $B \asymp 1$, 
$y' = o(\delta(N)^{B} \min(y,N/M_j)/\phi(q))$. 
For non-exceptional $n$ as above, it thus follows from Theorem~\ref{thm:explicit-Lambda}, applied with $B=E_2/4$, that
\begin{align*}
  &\frac{\phi(\W)}{\min(y, N/M_j)}  \sum_{\substack{ \W m + A'_n \in \\ [y', \min (y,N/n)]}} 
  \Lambda(\W m + A'_n) F(g(mn + A_n)\Gamma)\\
  &\qquad =  \sum_{\substack{ \W m + A'_n \in \\ [y', \min (y,N/n)]}} F(g(mn + A_n)\Gamma) 
  + O_{d, E_1, \|F\|_{\mathrm{Lip}}} \Big(
    \delta(N)^{E_2/4} \Big) 
  \ll_{d, E_1, \|F\|_{\mathrm{Lip}}} \delta(N)^{E_2/4} ~,
\end{align*}
where we used that $p < w(N)$ for every $p \mid q$.

Taking the contributions from both exceptional and non-exceptional $n \in  (M_j,2M_j] \in S([y',y])$ into account, the expression \eqref{eq:aux3} is thus bounded by
\begin{align*} 
\ll_{d, E_1, \|F\|_{\mathrm{Lip}}}
 (\delta(N)^{E_2/4}+\delta(N)^{E_1}) \min(y, N/M_j) \frac{\Psi(M_j,[y',y])}{\phi(\W)}.
\end{align*}
Summing over all $j$, the contribution of the middle segment to \eqref{eq:aux4} is therefore bounded above by
\begin{align} \label{eq:aux5} 
\nonumber
&\ll_{d, E_1, \|F\|_{\mathrm{Lip}}} 
(\delta(N)^{E_2/4}+\delta(N)^{E_1}) \sum_{j} \min(y, N/M_j) \frac{\Psi(M_j,[y',y])}{\phi(\W)} \\
&\ll_{d, E_1, \|F\|_{\mathrm{Lip}}} 
\frac{\delta(N)^{E_1}}{\phi(\W)} \sum_{j} \sum_{n \sim M_j} \1_{S([y',y])}(n) \min(y, N/n). 
\end{align}
We shall now make use of the fact that the inner sum in the middle range is guaranteed to be sufficiently long in order to deduce that we make an $\delta(N)^{E_1}$-saving on average on each inner sum provided $E_0$ and, hence, $E_2$ are sufficiently large.
Reversing the steps in the Montgomery--Vaughan-type decomposition will then complete the proof.

Turning towards the details, recall that for any $n \sim M_j$ the interval $[y',\min(y, N/n)]$ has length 
$|\min(y, N/n) - y'| \sim \min(y, N/n)$, so that 
$$
\min(y, N/n) \ll
|\min(y, N/n) - y'|
\ll  \sum_{y' \leq m \leq \min(y, N/n)} \Lambda(m), \qquad (n \sim M_j),
$$
by the prime number theorem.
By combining this estimate with the bound \eqref{eq:aux5}, the contribution of the middle segment to  \eqref{eq:aux4} is seen to be bounded above by: 
\begin{align*}
&\ll
\frac{\delta(N)^{E_1}}{\phi(\W)} \sum_{j} \sum_{n \sim M_j} \1_{S([y',y])}(n) \min(y, N/n) \\
&\ll
\frac{\delta(N)^{E_1}}{\phi(\W)} \sum_{j} \sum_{n \sim M_j} \1_{S([y',y])}(n) 
\sum_{y' \leq m \leq \min(y, N/n)} \Lambda(m) \1_{S([y',y])}(m)\\
&\ll
\frac{\delta(N)^{E_1}}{\phi(\W)} 
\sum_{\delta(N)^{2E_1} N/y \leq n < 2N/y^{2/3}~~}
\sum_{y' \leq m \leq \min(y, N/n)} \Lambda(m) \1_{S([y',y])}(mn)\\
&\ll
\frac{\delta(N)^{E_1}}{\phi(\W)}
\sum_{n < 2N} \1_{S([y',y])}(n) \log n \\
&\ll (\log N)
\delta(N)^{E_1} \frac{\Psi(N,[y',y])}{\phi(\W)},
\end{align*}
provided that $E_0$ is sufficiently large (to ensure that $E_2>4E_1$), and where the implied constant 
may depend on $d$, $\dim G$, $E_1$ and $\|F\|_{\mathrm{Lip}}$.
The lemma follows when taking into account the $(\log N)^{-1}$-factor in the statement.
\end{proof}

\section{The proof of Theorem \ref{thm:main-theorem}} \label{sec:conclusion}

We will use the transferred generalised von Neumann theorem \cite[Proposition~7.1]{GT-linearprimes} combined with a simple majorising function in order deduce Theorem \ref{thm:main-theorem} from Theorem~\ref{thm:main-uniformity}. 
There are two choices of marjorants for $h_{[y',y]}$ readily available in the case where $y\geq N^{\eps}$ for any fixed $\eps \in (0,1)$.
These can still be used for somewhat smaller values of $y$.
We state both majorant constructions below.

\subsection{GPY-type sieve majorant}
Consider the following GPY-type sieve majorant for numbers free from prime factors $p < y'$:
$$\Lambda_{\chi,y'}(n) = \log y' \Big(\sum_{d\mid n} \mu(d) \chi\Big(\frac{\log d}{\log y'}\Big) \Big)^2,$$
where $\chi:\RR \to \RR_{\geq 0}$ is a smooth function with support in $[-1,1]$ which takes the value $\chi(x)=1$ for all $x \in [-1/2,1/2]$.
In particular,
$$\Lambda_{\chi,y'}(n) 
= \log y' \Big(\sum_{\substack{d\mid n \\ d\leq y'}} \mu(d) \chi\Big(\frac{\log d}{\log y'}\Big) \Big)^2,$$
which implies that
$\Lambda_{\chi,y'}(n) = \log y'$ for all integers $n$ that are free from prime factors $p \leq y'$.
Since $\alpha = 1 - O((\eps \log N)^{-1}\log_2 N)$ for $y=N^{\eps}$,
Lemma \ref{lem:MV} shows that $\zeta(\alpha,y) \asymp \log y \asymp \log N$ and 
$g_{P(y')} \asymp (\log y')^{-1}$. 
Thus, by \eqref{eq:Hil-Ten-Thm1} and Lemma \ref{lem:BrTen-2.1}, we have
$$h_{[y',N^{\eps}]}(n) \leq \frac{(N/n)^{\alpha}n}{\Psi(N,[y',N^{\eps}])} 
\ll \frac{N \log N}{N g_{P(y')}(\alpha) \zeta(\alpha,N^{\eps})} \ll \log y'.$$
This shows that
$$h_{[y',N^{\eps}]}(n) \ll \Lambda_{\chi,y'}(n), \qquad(N< n \leq 2N).$$
Hence, $\Lambda_{\chi,y'}(n)$ satisfies the majorisation property (1) from above.
It follows from \cite[Theorem D.3]{GT-linearprimes} that the $W$-tricked version of $\Lambda_{\chi,y'}(n)$ satisfies the \emph{linear forms condition} as stated in \cite[Definition 6.2]{GT-linearprimes}.
We omit the proof here as it is essentially contained in \cite[Appendix D]{GT-linearprimes} and only mention that it relies on the fact that $\Lambda_{\chi,y'}$ carries the structure of a truncate divisor sum.

\subsection{The normalised characteristic function $\1_{S([y',N])}$ } \label{sec:cramer}
An alternative majorant function in the dense setting is given by the function
$$
\nu(n) = \frac{P(y')}{\phi(P(y'))} \1_{S([y',2N])}(n), \qquad (n \leq 2N).
$$
Since
$$
\1_{S([y',2N])}(n) = \1_{(n,P(y'))=1}(n)
$$
for $n \leq 2N$, this function corresponds to the \emph{Cram\'er model} for the von Mangoldt function that was
studied in recent work of Tao and Ter\"av\"ainen \cite{TT}: we have
$$\nu(n) = \Lambda_{\text{Cram\'er}, y'} (n) := \frac{P(y')}{\phi(P(y'))}\1_{(n,P(y'))=1}(n).$$
for all $n \leq 2N$.

\begin{lemma} \label{lem:majorant-N-eps}
Let $\eps \in (0,1)$. Then we have
$$h_{[y',y]}(n) \ll \Lambda_{\text{Cram\'er}, y'} (n)$$
uniformly for all $y>N^{\eps}$, all $n \in (N,2N]$ and all sufficiently large $N$.
\end{lemma}
\begin{proof}
Since $\1_{(n,P(y'))=1} = \1_{S([y',2N])}(n) $ for $n \leq 2N$,
it remains to note that
$$
h_{[y',y]}(n) 
\leq \frac{n (N/n)^{\alpha}}{\Psi(N,[y',N^{\eps}])} 
\ll \frac{N \log N}{N g_{P(y')}(\alpha) \zeta(\alpha,N^{\eps})}
\ll \log y' \asymp \frac{P(y')}{\phi(P(y'))} 
$$
when $n \sim N$ and $y> N^{\eps}$.
\end{proof}

Let $2 < w(N) < y' = (\log N)^{K'}$, $W = P(w(N))$, $0<A<W$ and $(A,W)=1$. 
Then, by \cite[Corollary 5.3]{TT}, we have
\begin{equation} \label{eq:U-k-norm-N-eps}
 \left\|\frac{\phi(W)}{W}\nu(W\cdot + A) - 1 \right\|_{U^k[\frac{N-A}{W}]} 
 \ll_k w(N)^{-c}
\end{equation}
for some constant $c>0$.
We observe that, using $\Lambda_{\chi,y'}(n)$ as a majorant for $\nu$, it follows from this uniformity norm estimate and the generalised von Neumann Theorem \cite[Proposition~7.1]{GT-linearprimes} that the $W$-tricked version of 
$\nu$ satisfies the linear forms condition from 
\cite[Definition 6.2]{GT-longprimeaps}.

\subsection{Proof of Theorem \ref{thm:main-theorem}}
Let $y_0 \leq y \leq N^{\eta}$, where $\eta \in (0,1)$ is sufficiently small depending on $d$ and $y_0$ will be determined in the course of the proof, and let $w(N) \leq y' \leq (\log N)^{K'}$ for some fixed $K'\geq 1$.
The transferred version of the quantitative inverse theorem for $U^k$-norms of Manners \cite{Manners} that Tao and Ter\"av\"ainen prove in \cite[Theorem 8.3]{TT} allows us to deduce explicit $U^k$-norm estimates from Theorem \ref{thm:main-theorem}, which we will later combine with the transferred generalised von Neumann theorem \cite[Proposition 7.1]{GT-linearprimes}. More precisely, Theorem \ref{thm:main-uniformity} implies that 
if $\Psi(N,[y',y]) P(y')/(N \phi(P(y'))) \geq (\log_4 N)^{-1/2}$, then
\begin{align*}
\bigg| \frac{W}{N}
\sum_{N/W < n \leq 2N/W} 
(g_{[y',y]}^{(W,A)}(n)-1)
F(g(n)\Gamma)
\bigg| 
&\ll (1 +\|F\|_{\mathrm{Lip}}) (\log_4 N)^{-1} \\
&\ll \frac{1 +\|F\|_{\mathrm{Lip}}}{\sqrt{\log_4 N}} \frac{N \phi(P(y')) }{\Psi(N,[y',y]) P(y')} 
\end{align*}
for all filtered nilmanifolds of dimension $O((\log_5 N)^{O(1)})$ and complexity $Q_0$ that is bounded by
$$Q_0 < \exp (\frac{1}{2} \sqrt{\log_4 N}),$$
and all nilsequences attached to it.
Observe that 
$$
(\log_4 N)^{-1/2} = o( \exp(- \exp (C_1/{\tilde \delta}^{C_2})))
$$
and
$$\exp \exp (C_1/{\tilde \delta}^{C_2}) = o(\exp (\frac{1}{2} \sqrt{\log_4 N}))$$ 
for all constants $C_1, C_2 >1$ if $\tilde \delta = 1/\log_7 N$.
Note further that the function 
from Section~\ref{sec:cramer} may be used as a trivial majorising function for the following rescaled version of 
$g_{[y',y]}^{(W,A)} (n)$, which essentially corresponds to $\1_{S([y',y])}(n)$. We have:
$$
\frac{P(y')}{\phi(P(y'))}
\frac{\Psi(N,[y',y])}{N} g_{[y',y]}^{(W,A)}(n) \ll \nu(n), \qquad(n \leq (N-A)/W~).
$$
By \eqref{eq:U-k-norm-N-eps}, we have  
$$
\| \nu - 1 \|_{U^k[(N-A)/W]} \ll_k (\log_3 N)^{-c} \ll_B {\tilde \delta}^B
$$
for all $B \geq 1$.
Hence the transferred quantitative $U^k$ inverse theorem of Manners \cite{Manners} as stated in \cite[Theorem 8.3]{TT}, implies that uniformly for all $0<A<W$ with $(W,A)=1$ we have
$$
\| g^{(W,A)}_{[y', y]} - 1 \|_{U^k([(N-A)/W])} \ll \frac{\tilde \delta N \phi(P(y')) }{\Psi(N,[y',y]) P(y')},
$$
which is small compared to the mean value $1$ of $g^{(W,A)}_{[y', y]}$ provided that
$$\Psi(N,[y',y]) P(y')/(N\phi(P(y'))) = o(\log_7 N).$$

In order to apply the generalised von Neumann theorem \cite[Proposition 7.1]{GT-linearprimes} in our situation, observe that the $o(1)$ error term in \cite[Proposition 7.1']{GT-linearprimes} corresponds to the error term in the 
$(D,D,D)$-linear forms condition for $\nu$. 
It follows from \cite[Proposition 5.2]{TT} that any function of the form
$$
\tilde \nu(n) = \frac{1}{t} \sum_{i=1}^t \frac{\phi(W)}{W} \nu^{(N)}(W n + A_i)
$$
with $0<A_i<W$, $\gcd(A_i,W)=1$ for all $1 \leq i \leq t$ satisfies the $(D,D,D)$-linear forms condition 
from \cite[Definition 6.2]{GT-linearprimes} for any given $D = O_{r,s,L}(1)$ with an error term of the form
$$
w(N)^{-c} = (\log_3 N)^{-c}.
$$
Applying \cite[Proposition 7.1]{GT-linearprimes} with 
$$f_i(n) =
\frac{P(y')}{\phi(P(y'))}
\frac{\Psi(N,[y',y])}{N} g_{[y',y]}^{(W,A)}(n)
$$ and with $\tilde \nu$ as above, we obtain
$$
\sum_{\n \in \ZZ^s \cap (N\mathfrak{K})/W} 
  \prod_{j=1}^r g_{[y',y]} (W \psi_j(\n) + A_j)
  = \left(\frac{W}{\phi(W)}\right)^r 
  \left\{\frac{N^s \vol \mathfrak{K}}{W^s} + o_{s,r,\|\Psi\|}\Big((N/W)^s\Big)\right\}
$$
provided that 
$$
\| g^{(W,A)}_{[y', y]} - 1 \|_{U^k([(N-A)/W])} = o\left(\Big(\frac{\Psi(N,[y',y]) P(y')}{N \phi(P(y'))}\Big)^{r-1}\right)
$$
for $k=O_{r,s}(1)$, as well as
$$
w(N)^{-c} = (\text{linear forms condition error term}) 
= o\left(\Big(\frac{\Psi(N,[y',y])P(y')}{N \phi(P(y'))}\Big)^{r}\right).
$$
These conditions are certainly satisfied if 
$$
\frac{\Psi(N,[y',y])P(y')}{N \phi(P(y'))} \gg (\log_8 N)^{-1}.
$$
Since 
$$
\frac{\Psi(N,[y',y])P(y')}{N \phi(P(y'))} \asymp 
\frac{\Psi(N,y)}{N} \prod_{p \leq y'} \frac{1-p^{-\alpha(N,y)}}{1-p^{-1}}
\leq \frac{\Psi(N,y)}{N} = u^{-u+o(u)}
$$
by Lemma \ref{lem:BrTen-2.1}, it follows that $N^{-1}\Psi(N,y) \gg (\log_8 N)^{-1}$.
Let $y_0$ be such that 
$$
\frac{\Psi(N,y_0)}{N} = (\log_8 N)^{-1}
$$
and write $u_0= (\log N)/ \log y_0$.
Then 
$$
\frac{\log N}{\log y_0}=u_0 < u_0^{u_0+o(u_0)} = \log_8 N
$$
and, if $y>y_0$, then
$$1 - \alpha(N,y) \ll \frac{\log (u_0 \log (u_0+2))}{\log y_0} \ll \frac{(\log_8 N)^2}{\log N}.$$
Hence, the first part of Lemma \ref{lem:MV} applies to $g_{P(y')}(\alpha(N,y))$ and yields
$$g_{P(y')}(\alpha(N,y)) \asymp \frac{1}{\log y'} \asymp \frac{\phi(P(y'))}{P(y')}$$
for $y\geq y_0$ and $y' = (\log N)^{K'}$. Thus
$$
\frac{\Psi(N,[y',y])P(y')}{N \phi(P(y'))} \asymp 
\frac{\Psi(N,y)}{N} \gg (\log_8 N)^{-1}
$$
for $y\geq y_0$. 
From 
$$
\exp(u_0^2) > u_0^{u_0+o(u_0)} = \log_8 N,
$$
which holds for sufficiently large $N$, we deduce that $y_0 < N^{1/\sqrt{\log_9 N}}$.

Suppose now that $y_0 \leq y \leq N^{\eta}$. Then we obtain:
\begin{align*}
  &\sum_{\n \in \ZZ^s \cap N\mathfrak{K}} \prod_{j=1}^r g_{[y',y]} (\psi_j(\n)+a_j) \\
  &= \sum_{\substack{\mathbf{A} \in  \{0, \dots W-1\}^s}}
  \sum_{\substack{W\n+\mathbf{A} \\ \in \ZZ^s \cap N\mathfrak{K}}} 
  \prod_{j=1}^r g_{[y',y]} (\psi_j(W \n + \mathbf{A})+a_j)\\
  &= \sum_{\substack{\mathbf{A} \in  \{0, \dots W-1\}^s}}
  \sum_{\n \in \ZZ^s \cap (N\mathfrak{K}-\mathbf{A})/W} 
  \prod_{j=1}^r g_{[y',y]} (W \psi_j(\n) + \psi_j(\mathbf{A})+a_j) \\
  &= \left(\frac{W}{\phi(W)}\right)^r 
  \left\{\frac{N^s \vol \mathfrak{K}}{W^s} + o_{s,r,\|\Psi\|}\Big((N/W)^s\Big)\right\}
  \sum_{\substack{\mathbf{A} \in  \{0, \dots W-1\}^s}}
  \prod_{j=1}^r \1_{\gcd (W,\psi_j(\mathbf{A})+a_j)=1}\\
  &= \left\{N^s \vol \mathfrak{K} + o_{s,r,\|\Psi\|}(N^s)\right\}
  \prod_{p < w(N)} \beta_p,
\end{align*}
where
$$
\beta_p = \frac{1}{p^s} 
  \sum_{\mathbf{u} \in (\ZZ/p\ZZ)^s} \prod_{j=1}^r \frac{p}{p-1}
  \1_{\psi_j(\mathbf{u})+a_j \not\equiv 0 \Mod p}.
$$
Theorem \ref{thm:main-theorem} now follows on recalling\footnote{Observe that our local factors are identical to those defined in \cite[(1.6)]{GT-linearprimes}} from \cite[Lemma 1.3]{GT-linearprimes} that the condition that the forms $\psi_i$ be pairwise linearly independent over $\QQ$ implies that $\beta_p = 1 + O_{s,r,L}(p^{-2})$ for all primes $p > w(N)$.

\end{document}